\documentclass[11pt]{amsart}
\usepackage[
left=1in,top=1in,right=1in,bottom=1in]{geometry}

\usepackage{graphicx}
\usepackage[all]
{xy}
\usepackage{url}

\usepackage[boxed,vlined,lined]{algorithm2e}

\usepackage{enumerate}
\usepackage{amsfonts,amssymb}
\usepackage{mathrsfs}

\numberwithin{equation}{section}

\theoremstyle{plain}
\newtheorem{theorem}{Theorem}[section]
\newtheorem{lemma}[theorem]{Lemma}

\newtheorem{proposition}[theorem]{Proposition}
\newtheorem{corollary}[theorem]{Corollary}

\theoremstyle{definition}

\newtheorem{convention}[theorem]{Convention}
\newtheorem{definition}[theorem]{Definition}

\newtheorem{exampleth}[theorem]{Example}
\newenvironment{example}{\begin{exampleth}}{\hfill $\diamond$\\ \end{exampleth}}

\theoremstyle{remark}
\newtheorem{remark}[theorem]{Remark}

\newcommand{\tightcdot}{\!\cdot\!}

\makeatletter
\@namedef{subjclassname@2010}{%
 \textup{2010} Mathematics Subject Classification}
\makeatother

\newcommand \QQ {\mathbb{Q}} 
\newcommand \ZZ {\mathbb{Z}}
\newcommand \CC {\mathbb{C}}
\newcommand \RR {\mathbb{R}}

\renewcommand{\AA}{\zeta}

\DeclareMathOperator{\Spec}{\operatorname{Spec}}

\DeclareMathOperator{\conv}{\operatorname{conv}}
\DeclareMathOperator{\mult}{\operatorname{m}}

\DeclareMathOperator{\vol}{\operatorname{vol}}

\DeclareMathOperator{\trop}{trop}

\DeclareMathOperator{\an}{an}
\DeclareMathOperator{\id}{id}
\DeclareMathOperator{\outerEdge}{out}
\DeclareMathOperator{\intEdge}{int}
\DeclareMathOperator{\Idd}{Id}

\newcommand \cP {\mathcal{P}}

\newcommand \ww {\omega}

\DeclareMathOperator{\init}{in}

\newcommand{\sigmaint}{\sigma^{\circ}}

\DeclareMathOperator {\val}{val}
\DeclareMathOperator {\Star}{\ensuremath{Star}}

\DeclareMathOperator {\Trop}{Trop}

\newcommand{\PS}{\CC\{\!\{t
\}\!\}}
\newcommand{\PSN}{\CC(\!(t^{1/N})\!)}
\newcommand{\PSR}{\CC\{\!\{t^{\RR}\}\!\}}

\title{How to repair tropicalizations of plane curves using modifications}

\author
{
Maria Angelica Cueto
}
\address{
Mathematics Department, 
Columbia University, 
2990 Broadway, 
New York, NY 10027, USA.
}
\email{macueto@math.columbia.edu}
\author[Hannah Markwig]{Hannah Markwig${}^{\S}$}
\address{Universit\"at des Saarlandes, Fachr.
  Mathematik, Postfach 151150, 66041 Saarbr\"ucken, Germany.}
\email{hannah@math.uni-sb.de}
\thanks{${\S}$ \emph{Corresponding author}}


\begin{document}

\begin{abstract} Tropical geometry is sensitive to embeddings of
  algebraic varieties inside toric varieties.
  The purpose of this paper is to advertise tropical modifications as
  a tool to locally repair bad embeddings of plane curves, allowing
  the re-embedded tropical curve to better reflect the geometry of the input
  curve. Our approach is based on the close connection between
  analytic curves (in the sense of Berkovich) and tropical curves. 
We investigate the effect of these tropical modifications on
  the tropicalization map defined on the analytification of the given curve. 

  Our study is motivated by the case of plane elliptic cubics, where
  good embeddings are characterized in terms of the $j$-invariant.
  Given a plane elliptic cubic whose tropicalization contains a cycle,
  we present an effective algorithm, based on non-Archimedean methods,
  to linearly re-embed the curve in dimension 4 so that its
  tropicalization reflects the $j$-invariant.  We give an alternative
  elementary proof of this result by interpreting the initial terms of
  the $A$-discriminant of the defining equation as a local
  discriminant in the Newton subdivision.
  \end{abstract}
 \maketitle

  \section{Introduction}\label{sec:intro}

Tropical geometry is a piecewise-linear shadow of algebraic geometry that
preserves important geometric invariants.
Often, we can derive statements about  algebraic varieties by means
of these (easier) combinatorial objects.  One general difficulty in
this approach is that the tropicalization strongly depends on the
embedding of the algebraic variety. Thus, the task of finding a
suitable embedding or  repairing a given ``bad'' embedding to obtain
a nicer tropicalization becomes essential for many applications.  The
purpose of this paper is to advertise tropical modifications as a tool
to locally repair embeddings of plane curves, as suggested by
%
Mikhalkin in his ICM 2006 lecture~\cite{Mi06}.

An important and motivating example is the case of elliptic curves. 
In \cite[Example 3.15]{Mi06}, Mikhalkin proposed the
cycle length of a tropical plane elliptic cubic to be the tropical
counterpart of the classical $j$-invariant.  Inspired by this remark and using
Gr\"obner fan techniques, Katz, Markwig and the second author proved
that when the elliptic cubic is defined over the Puiseux series field, the
valuation of the $j$-invariant is generically reflected on the cycle
length of the tropical curve~\cite{KMM07}. For special choices
of coefficients, this length can be shorter than expected. These
non-generic situations have a very explicit characterization. First,
the cycle in the tropical curve must contain a vertex of valency at least
four, and second, the initial form of the discriminant of the cubic
must vanish. Thus, in the case of plane  elliptic cubics, or more
generally, for elliptic curves embedded smoothly into a toric surface, 
the question of what constitutes a good embedding from the tropical
perspective has a precise answer: the cycle length should reflect the negative valuation of the $j$-invariant.

One of the main contributions of the present paper is an algorithm that
recursively repairs bad embeddings when the tropical plane elliptic 
cubic contains a cycle. The power of
Algorithm~\ref{alg:repairElliptic} lies in its simplicity: it only
uses linear tropical modifications of the plane, and linear
re-embeddings of the original curve. Furthermore, this result is achieved in dimension 4.  This approach has an additional
advantage. Rather than drastically changing the polyhedral structure
of the input tropical curve, it keeps its relevant features. It only
adds missing edges and changes tropical multiplicities. The output tropical
curve has the expected cycle length.  We view this as a possibility to
``locally repair'' the problematic initial embedding.

The case of elliptic curves suggests itself as a playground for
uncovering the deep connections between Berkovich's theory,
tropicalizations, and re-embeddings.  More concretely, let $X$ be a
smooth elliptic curve and $\mathscr{X}$ be a semistable regular model
of $X$ over a discrete valuation ring. Let us assume that $X$ has bad
reduction. Then, the minimal Berkovich skeleton of the complete
analytic curve $\widehat{X}^{\an}$ is homotopic to a circle and it can
be obtained from the dual graph of the special fiber of $\mathscr{X}$.
Foundational work of Baker, Payne and Rabinoff proves that when the
embedding induces a faithful tropicalization on the cycle, the length
appearing in the minimal Berkovich skeleton induced by its canonical
metric equals the corresponding lattice length in the tropicalization
of $X$~\cite[Section 6]{BPR11}.  Notably, \cite[Section 7]{BPR11}
provides examples where the cycle in a tropicalization of a smoothly
embedded elliptic curve is shorter, or longer, than the negative
valuation of the $j$-invariant. Good embeddings of elliptic curves
with bad reduction are those where the minimal skeleton of the
complete analytic curve is reflected in the associated tropical curve.

Characterizing good embeddings of curves in terms of their minimal
Berkovich skeleta has one clear advantage compared to the study of
tropicalizations: it is intrinsic to the curve.  Work of Payne shows
that the Berkovich space $X^{\an}$ is the limit of all
tropicalizations of $X$ with respect to closed embeddings into
quasiprojective toric varieties (see~\cite[Theorem 4.2]{Pay09}). We
view $X^{\an}$ as a topological object incorporating all choices of
embeddings.

After the investigation by Baker, Payne and Rabinoff~\cite{BPR11}, the
meaning of suitable embeddings of curves for tropicalization purposes
becomes precise: they should induce \emph{faithful
  tropicalizations}. That is, the corresponding tropical curve must be
realized as a closed subset of $X^{\an}$, and this identification
should preserve both metric structures. Faithful tropicalizations of
Mumfurd curves of genus 2 have been recently studied by Wagner
in~\cite{Till14}. In the case of plane elliptic cubics, we can
reinterpret the main result of~\cite{KMM07} in the language of
Berkovich's theory by saying that the tropicalization to a $3$-valent
cubic is always faithful on the cycle. This follows from the fact that
all edges in a tropical cubic with a cycle have multiplicity $1$,
see~\cite[Theorem 6.24 and 6.25]{BPR11}.

Two natural questions arise from the previous discussion. First, can
we effectively construct embeddings of a given curve that induce
faithful tropicalizations? Can we do so without computing a minimal
Berkovich skeleton of the complete curve?  Again, the case of elliptic
cubics is a fantastic playground for exploring this question, since
the faithfulness on its cycle can easily be characterized in terms of
the $j$-invariant. Following this approach, Chan and Sturmfels
described a procedure to put any given plane elliptic cubic with bad
reduction into honeycomb form~\cite{CS13}. The honeycomb form is
$3$-valent and has edges of multiplicity $1$, hence it induces a
faithful tropicalization and the cycle has the expected
length. Although running in exact arithmetic, their method involves
the resolution of a univariate degree 6 equation. Each solution is
expressed as a Laurent series in the sixth root of the multiplicative
inverse of the $j$-invariant. The solution is constructed recursively,
one term at a time.  This re-embedding completely alters the structure
of the original tropical curve.

In contrast, our approach allows us to give a positive and effective
answer to the questions above.  Our algorithm to repair embeddings of
plane elliptic cubics relies on methods we develop in
Section~\ref{sec:repa-trop} for arbitrary plane
curves. Theorems~\ref{thm:redVertex} and~\ref{thm:fatEdge} allow us to
locally repair certain embeddings of curves using a linear tropical
modification of the plane.  They should be viewed as a partial answer
to the questions above. They hold under certain constraints imposed by
the local topology of the input tropical curve.  Nonetheless, these
two technical results suffice to completely handle the case of plane
elliptic cubics.  As a byproduct, we enrich Payne's result
\cite[Theorem 4.2]{Pay09} for plane elliptic cubics connecting the
Berkovich space to the limit of all tropicalizations by a concrete
procedure that gives the desired tropically faithful embedding using
only linear tropical modifications of the plane.  Our experiments in
Section~\ref{sec:experiments} suggest that the techniques introduced
in this paper may be extended to other combinatorial types of tropical
curves, although new ideas will be required to generalize
Theorems~\ref{thm:redVertex} and~\ref{thm:fatEdge}.

We have mentioned already that $A$-discriminants of cubic polynomials
play a key role when studying the $j$-invariant of a plane elliptic
cubic. In the same spirit, Theorems~\ref{thm:redVertex} and~\ref{thm:fatEdge} also involve
``local discriminants'' associated to certain maximal cells in the
Newton subdivision of the input plane
curve. In Section~\ref{sec:repa-ellipt-cubics} we derive
Algorithm~\ref{alg:repairElliptic} in an elementary fashion, by relating the global discriminant of the cubic to the local
discriminants mentioned above. This is the content of
Corollary~\ref{cor:initialDiscriminant}. Theorem~\ref{thm:factorizationFormula}  provides a factorization
formula for initial forms of discriminants of planar configurations. We expect this result to
have further applications besides Algorithm~\ref{alg:repairElliptic}.

\smallskip

The rest of the paper is organized as follows. In
Section~\ref{sec:trop-modif-line}, we introduce notation and discuss
background on tropicalizations, modifications and linear
re-embeddings.  In Lemma~\ref{lem:ModifViaProjections}, we
characterize linear re-embeddings of plane curve induced by tropical
modifications along straight lines in terms of charts and coordinates
changes of $\RR^2$.  Thus, we can visualize the repaired embeddings by
means of collections of tropical plane curves. In
Sections~\ref{sec:berk-skel-curv} and~\ref{sec:discriminants} we
discuss preliminaries involving Berkovich skeleta and
$A$-discriminants, which play a prominent role in our study.

In Section~\ref{sec:repa-trop}, we present our two main technical
tools to locally repair embeddings of smooth plane curves by linear
re-embeddings. Our proof builds upon Berkovich's theory,
$A$-discriminants of plane configurations, and
Lemma~\ref{lem:ModifViaProjections}. By using linear tropical
modifications and coordinate changes of $\RR^2$, the tropical
re-embedded curve will faithfully represent a subgraph of a skeleton
of the analytic curve induced by its set of punctures.

In Section~\ref{sec:trop-ellipt-curv}, we focus our attention on plane
elliptic curves and present Algorithm~\ref{alg:repairElliptic}.  We
provide two independent proofs of its correctness.  The first one
relies on the techniques developed in Section~\ref{sec:repa-trop} and
is discussed in Section~\ref{sec:repair-elliptic}. The second one is
elementary: it is based purely on discriminants of plane
configurations.  We present it in
Section~\ref{sec:repa-ellipt-cubics}. The main result of this section
is Theorem~\ref{thm:factorizationFormula}, which relates global and
local discriminants of planar point configurations.

In Section~\ref{sec:experiments} we provide several experimental
evidence to support the use of our repairing techniques in examples
that are not covered by Theorems~\ref{thm:redVertex}
and~\ref{thm:fatEdge}. We view this last section as a motivation to
further study this topic.

\section{Preliminaries and Motivation}\label{sec:prel-motiv}
We work with ideals $I$ defining irreducible subvarieties of a torus
and denote their \emph{tropicalizations} by $\Trop(I)$ (see
e.g.~\cite{MS09}).  Throughout this paper we often work with complete
curves and their minimal Berkovich skeleta. For this reason, we always
consider our ideals inside honest polynomial rings rather than Laurent
polynomial rings. For tropicalization, we consider the intersection of
the curve with the algebraic torus in the given embedding.

For concrete computations, we fix the field of \emph{generalized
  Puiseux-series} $\PSR$, with valuation taking a series to its
leading exponent. For tropicalizations, we use the negative of the
valuation, i.e.\ the \emph{max-convention}. We denote our algebraic
coordinates by $x,y,z$, whereas we indicate the tropical coordinates
by $X,Y,Z$. We use analogous conventions in higher dimensions.

\subsection{Tropical modifications and linear re-embeddings}\label{sec:trop-modif-line}
Tropical modifications appeared in \cite{Mi06} and have since then
found several interesting applications, e.g.\ \cite{AR07,ABBR13,
  BLdM11, Sha10}.  Here, we concentrate on modifications of the plane
$\RR^2$ along linear divisors.

Let $F=\max\{A,B+ X,C+ Y\}$ be a linear tropical polynomial with $A,
B$ and $C$ in $\mathbb{T}\RR:=\RR\cup\{-\infty\}$. The graph of $F$
considered as a function on $\RR^2$ consists of at most three linear
pieces. At each break line, we attach two-dimensional cells spanned in
addition by the vector $(0,0,-1)$ (see e.g.\ \cite[Construction
3.3]{AR07}). We assign multiplicity $1$ to each cell and obtain a
balanced fan in $\RR^3$. It is called the \emph{modification of
  $\RR^2$ along $F$}.

Let $f=a+bx+cy\in \PSR[x,y]$ be a lift of $F$, i.e.\ $-\val(a)=A$,
$-\val(b)=B$ and $-\val(c)=C$.  We fix an irreducible polynomial $g\in
\PSR[x,y]$ defining a curve in the torus $(\PSR^*)^2$. The
tropicalization of $I_{g,f}=\langle g, z-f\rangle\subset \PSR[x,y,z]$
is a tropical curve in the modification of $\RR^2$ along $F$. We call
it the \emph{linear re-embedding} of the tropical curve $\Trop(g)$
\emph{with respect to $f$}.

For almost all lifts $f$, the linear re-embedding coincides with the
modification of $\Trop(g)$ along $F$, i.e.~we only bend $\Trop(g)$ so
that it fits on the graph of $F$ and attach some downward ends.
However, for some choices of lifts $f$, the part of $\Trop(I_{g,f})$
in the cells of the modification attached to the graph of $F$ contains
more attractive features.  We are most interested in these special
linear re-embeddings. The following example illustrates this
phenomenon.

\begin{example} We fix a plane elliptic cubic defined by
\begin{align*}
  g(x,y)=&-t^2x^3+t^{200}x^2y+(t^2+t^4)xy^2+t^{14}y^3+(-3t^3-t^{200})x^2+(t^3+t^5-t^6+t^{12}-t^{202})y^2\\
&+(1+2t^{201})xy
  +(-3t^4+t^{200}-2t^{201})x+(t+t^2+t^{202})y+(2t^2-t^5+t^{201}-t^{202}).
\end{align*}
We aim to modify the tropical curve $\Trop(g)$ along the vertical line
$X=-1$ in $\RR^2$. This line corresponds to a tropical polynomial
$F=\max\{-1,X\}$. Its lifting $f$ is of the form $f=x+\AA t$ where
$\AA\in \PS$ has valuation 0. The tropicalization $\Trop(I_{g,f})$
depends only on the initial coefficient of $\AA$. Indeed, unless this
coefficient is one, this tropical curve coincides with the
modification of $\Trop(g)$ along $F$.
Figure~\ref{fig:OriginalEllipticModification} shows the special linear
re-embedding when $\AA =1$.
\end{example}

\begin{figure}[tb]
  \centering
\includegraphics[scale=0.31]{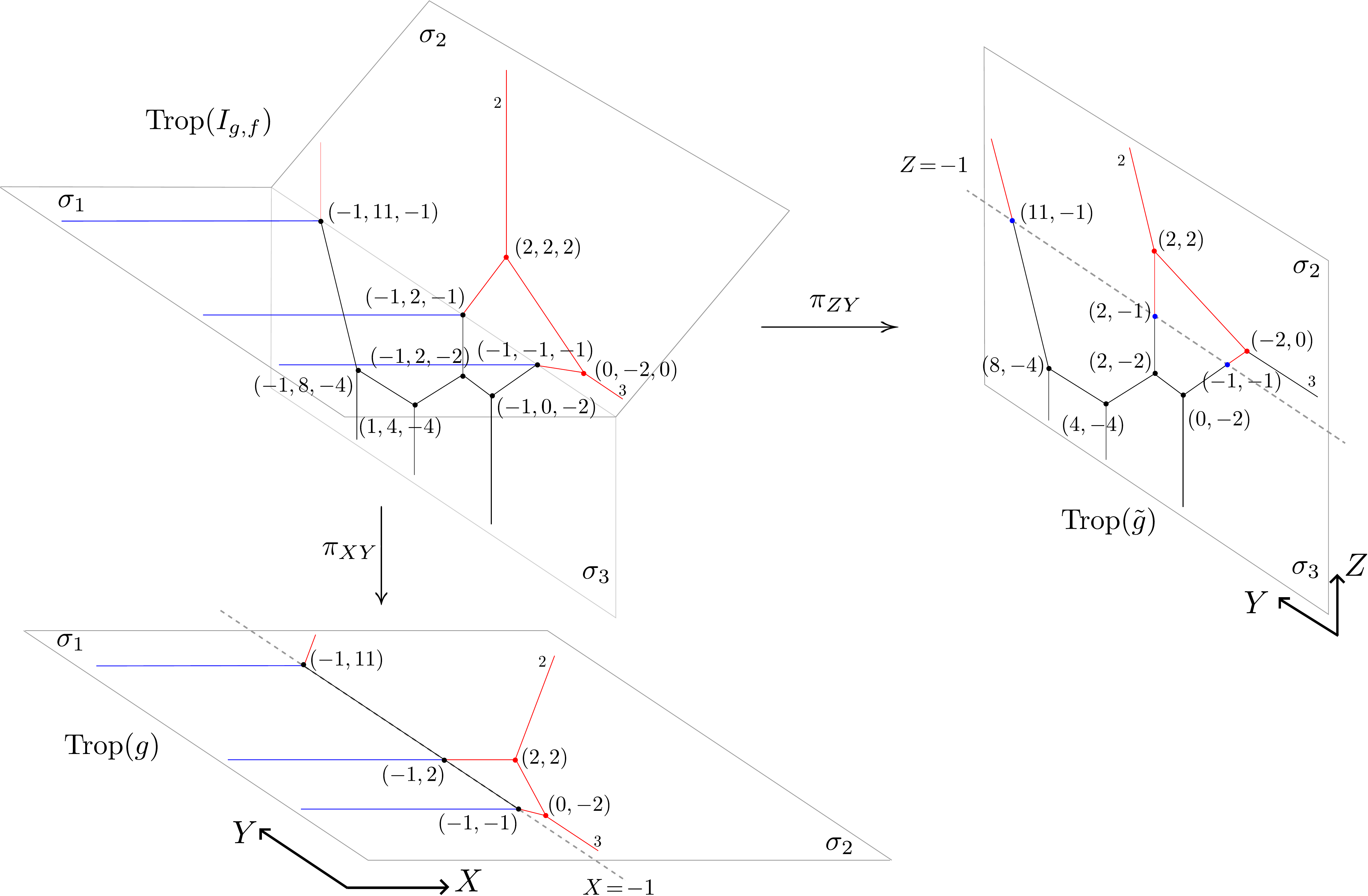}


\caption{A special re-embedding of the tropical curve $\Trop(g)$ with
  respect to $f=x+t$ together with the projections $\pi_{XY}$ and
  $\pi_{ZY}$. The original curve is drawn at the bottom, whereas a new
  curve $\Trop(\tilde{g})$ appears on the right. The central picture
  shows the tropical curve $\Trop(I_{g,f})$ in the modified
  plane.}\label{fig:OriginalEllipticModification}
  \end{figure}
  Our main focus in Section~\ref{sec:repa-trop} will be on
  modifications of $\RR^2$ along vertical lines.  These modifications
  are induced by tropical polynomials of the form $F=\max\{X,l\}$,
  with $l\in \QQ$.  Their liftings are of the form $f=x+\AA t^{-l}$
  where $\AA\in \PS$ has valuation 0. As we see in
  Figure~\ref{fig:OriginalEllipticModification}, the modified plane contains three
  maximal cells:
\[
\sigma_1=\{X\leq l, Z=l\}, \quad \sigma_2=\{X\geq l, Z=X\}, \quad and \quad
\sigma_3=\{X=l, Z\leq l\}. 
\]
By construction, $\sigma_3$ is the unique cell of the modification of $\RR^2$
attached to the graph of $F$. We let $\sigmaint_i$ denote the relative
interior of the cell $\sigma_i$, for $i=1, 2, 3$.

We describe $\Trop(I_{g,f})$ by means of two projections:
\begin{enumerate}
\item the projection $\pi_{XY}$ to the coordinates $(X,Y)$ produces the
  original curve $\Trop(g)$,
\item the projection $\pi_{ZY}$ gives a new tropical plane curve
  $\Trop(\tilde{g})$ inside the cells $\sigma_2$ and $\sigma_3$, where
  $\tilde{g}=g(z-\AA\,t^{-l},y)$. The polynomial $\tilde{g}$ generates the
elimination ideal $I_{g,f}\cap \PS[y,z]$.
\end{enumerate}
Notice that the projection $\pi_{XZ}$ gives no information about $\Trop(I_{g,f})$ since it maps any
tropical curve to the tropical line with vertex $(l,l)$.
The following lemma explains how to reconstruct the curve
$\Trop(I_{g,f})$ inside this modified plane using
the two relevant projections above.

\begin{lemma}\label{lem:ModifViaProjections}
  The linear re-embedding $\Trop(I_{g,x+\AA t^{-l}})$ in the
  modification of $\RR^2$ along the linear tropical polynomial
  $F=\max\{X, l\}$ is completely determined by the two tropical plane
  curves $\Trop(g)$ and $\Trop(\tilde{g})$, where
  $\tilde{g}(z,y)=g(z-\AA \,t^{-l},y)$. In particular, the vertices of
  $\Trop(I_{g, x+\AA t^{-l}})$ along the line $\{X=Z=l\}$ are 
  the endpoints of the connected components of $( \Trop(g)\times \RR)\cap
 (\RR\times  \Trop(\tilde{g}))\cap\{X=Z=l\}$.
\end{lemma}
\begin{proof} First, we fix a point $(X,Y,Z)$ with either $X\neq l$ or
  $Z\neq l$, thus in the relative interior of one of the cells
  $\sigma_i$, for $i=1,2,3$. We claim that $(X,Y,Z)$ belongs to
  $\Trop(I_{g,x+\AA t^{-l}})$ if and only if one of the following two
  conditions hold:
\[ Z= \max\{X,l\}\text{ and }(X,Y)\in \Trop(g), \quad \text{ or }
 \quad X= \max\{Z,l\}\text{ and }(Z,Y)\in \Trop(\tilde{g}).
\]
  The first implication follows directly from the Fundamental theorem
  of tropical algebraic geometry (see e.g.~\cite[Theorem 3.2.5]{MS09}) and the fact that $g,
  \tilde{g}, \pm (z-(x+\AA t^{-l})) \in I_{g,x+\AA
    t^{-l}}$. For the converse, 
we use the same result to lift a  point  $(Z,Y)\in
  \Trop(\tilde{g})\cap \QQ^2$ with $Z\neq l$ to a unique point
 $(X,Y,Z)$ in $\Trop(I_{g,x+\AA
    t^{-l}})\cap(\sigmaint_2\cup \sigmaint_3)$. This
  point satisfies $X=\max\{Z,l\}$.   Analogously, any point  $(X,Y) \in \Trop(g)\cap \QQ^2$ lifts uniquely to a point
  $(X,Y,Z)$ in $\Trop(I_{g,x+\AA
    t^{-l}})\cap(\sigmaint_1\cup\sigmaint_2)$,  where $Z=\max\{X,l\}$.

  It follows that the set of points in $\Trop(I_{g,x+\AA t^{-l}})$
  outside the line $\{X=Z=l\}$ is completely determined by the two
  projections $\Trop(g)$ and $\Trop(\tilde{g})$. It remains to prove
  that we can also detect the tropical multiplicities and all points
  in $\Trop(I_{g,x+\AA t^{-l}})
  $ from these two projections. To see this, notice first that the
  multiplicities of all edges of $\Trop(I_{g,x+\AA t^{-l}})$ whose
  relative interior lies in $\sigmaint_i$ for $i=1,2,3$ coincide with
  the corresponding multiplicities of the projected edges in
  $\Trop(g)$ or $\Trop(\tilde{g})$, respectively. This follows from
  the unique lifting property discussed above and the generalized
  push-forward formula for multiplicities of Sturmfels-Tevelev in the
  non-constant coefficients case
  ~\cite[Theorem~8.4]{BPR11}. We can also compute the multiplicity of
  an edge on the line $\{X=Z=l\}$ by comparing the multiplicities of
  the preimages of the edge in the two charts. An edge of multiplicity
  zero should be interpreted as a phantom edge. This concludes our
  proof.
\end{proof}

Using the previous result we can visualize the modification of $\RR^2$
along a vertical line and the effect of the linear re-embedding on the
tropical curve $\Trop(g)$ by means of the two relevant projections.
The colors and cell labels on the projections and the modified plane
in Figure~\ref{fig:OriginalEllipticModification} indicate the nature
of the fibers of each projection. The dashed line on each projection
represents the image of the vertical line used to modify $\RR^2$.  We
keep these conventions throughout this paper.

By Lemma~\ref{lem:ModifViaProjections} we know that the features of
$\Trop(I_{g,x+\AA t^{-1}})\cap \sigmaint_3$ are encoded in the
polynomial $\tilde{g}(z,y):=g(z-\AA \,t^{-l}, y)$.  For special values
of $\AA$, the Newton subdivision of $\tilde{g}$ is unexpected and
yields an interesting behavior in $\Trop(I_{g,x+\AA t^{-l}})\cap
\sigmaint_3$. We observe this phenomenon in
Figure~\ref{fig:OriginalEllipticModification}: the cycle on the
tropical curve $\Trop(g)$ was placed to the right of the vertical line
$X=-1$, but in $\Trop(\tilde{g})$ this cycle has been prolonged and
its leftmost vertical edge has been pushed from the line $Z=-1$ to the
line $Z=-2$. This example illustrates the general principle described
in the title of this paper. We discuss it further in
Section~\ref{sec:repa-trop}.

As we mention earlier, our goal is to use linear tropical
modifications to repair embeddings of plane curves. Let $J\subset
\PS[z_1,z_2,z_3,\ldots, z_r]$  be a linear ideal defining a plane in $\PS^{r}$. We re-embed the curve $g(z_1,z_2)$ via the ideal $ g + J$. As in Lemma~\ref{lem:ModifViaProjections}, we can construct $\Trop(g+J)$ from suitable 2-dimensional projections. 

In order to do so, we find generators of $J$ adapted to a fixed 2-cell
$\sigma$ of $\Trop(J)$. We let $(Z_i,Z_j)$ be the local coordinates of
$\sigma$. Then, the corresponding variables $z_i,z_j$ must be linearly
independent on $J$ and we can find unique polynomials $f_k\in
\PS[z_i,z_j]$ for $k\neq i,j$ such that
\begin{equation}
J=\langle  z_k-f_k, k\neq i,j\rangle.
\label{eq:coordsJ}
 \end{equation}

\begin{proposition}\label{pr:projectionsJ}
  Le $\ww\in \Trop(J)$ and fix a two-dimensional cell $\sigma$ of
  $\Trop(J)$ with local coordinates $(Z_i,Z_j)$ containing $\ww$.
  Then, the ideal $\init_{\ww}(g+J)/\init_{\ww}(J)\subset \CC[z_1^{\pm}, \ldots,
  z_r^{\pm}]/\init_{\ww}(J)$ is isomorphic to the localization
  $\init_{(\ww_i,\ww_j)}(\tilde{g}(z_i,z_j)) [S^{-1}]\subset
  \CC[z_i^{\pm}, z_j^{\pm}][S^{-1}]$, where $\tilde{g}(z_i,z_j)=( g +
  J) \cap \PS[z_i,z_j]$ and $S$ is the the multiplicatively closed set
  generated by all $\init_{(\ww_i,\ww_j)}(f_k)$, $k\neq i,j$.
\end{proposition}
\begin{proof} 
  To simplify notation, we consider all initial ideals in the
  statement defined by $\ww$, rather than by the projection
  $(\ww_i,\ww_j)$.  By definition, $ \init_{\ww}(g),
  \init_{\ww}(\tilde{g}) \subset \init_{\ww} (g+J)$.

  For each $k\neq i,j$ we write $f_k:= a_kz_i + b_k z_j + c_k$ for
  suitable $a_k,b_k,c_k\in \PS$.  In order to prove the statement, we
  study the interplay of $a_k,b_k,c_k$ with $\ww\in \QQ^r$.
  By~\eqref{eq:coordsJ}, any point $z$ in the plane defined by $J$
  with $-\val(z)=\ww$ is uniquely determined by its $(i,j)$
  coordinates.  Since $\ww\in \sigma\subset \Trop(J)$, the fundamental
  theorem of tropical algebraic geometry ensures that
\begin{equation}
\ww_k=
\max\{-\val(a_k)+\ww_i, -\val(b_k)+w_j, -\val(c_k)\}\quad \text{ for all }\quad k\neq
i,j.\label{eq:initialFormsJ}
\end{equation}
Hence, $\init_{\ww}(z_k-f_k)=z_k -\init_{\ww}(f_k)$ and we conclude
that $ \CC[z_i^{\pm}, z_j^{\pm}]\cap \init_{\ww} J =0$ because $\dim
J=\dim (\init_{\ww} J)=2$. Therefore, the generators
from~\eqref{eq:coordsJ} give a basis to compute $\init_{\ww}J$, i.e.\
$\init_{\ww}J=\langle z_k-\init_{\ww}(f_k): k\neq i,j\rangle$, In
particular, all elements of $S$ are units in $\CC[z_1^{\pm},\ldots,
z_r^{\pm}]/\init_{\ww}J$.

As a consequence, we construct an isomorphism $\varphi\colon
\CC[z_1^{\pm}, \ldots, z_r^{\pm}]/\init_{\ww}J \to \CC[z_i^{\pm 1},
z_j^{\pm}] [S^{-1}]$ by
\begin{equation}
  \label{eq:mapLocalCoord}
  \varphi(\overline{z_i})=z_i, \; \varphi(\overline{z_j})=z_j ,\; \text{ and }\; \varphi(\overline{z_k})=\init_{\ww}(f_k) \text{ for all }k\neq i,j.
\end{equation}
This map induces an isomorphism between the ideals
$(\init_{\ww}(\tilde{g})+\init_{\ww}J)/\init_{\ww}J$ and
$\init_{\ww}(\tilde{g})[S^{-1}]$.  To prove the statement, we show
that $\init_{\ww}(\tilde{g})$ generates the quotient ideal
$\init_{\ww}(g+J)/\init_{\ww}J$.  

Recall that $z_i,z_j$ and the elements of $S$ are units in the domain
of $\varphi$. We pick a $\ww$-homogeneous polynomial $h\in
\init_{\ww}(g+J)\cap \CC[z_i, z_j]$ and show that $h\in \langle
\init_{\ww}\tilde{g}\rangle $.  By~\cite[Lemma 2.12]{FJT07}, we know
that $h$ is the initial form of an element $f\in g+J$. We write
$f:=p(z_i,z_j)\tilde{g} + \sum_{k\neq i,j} q_k(z_k-f_k)$.  Since
$\init_{\ww}(z_k-f_k)$ contains $z_k$ in its support but $h\in
\CC[z_i,z_j]$, an easy induction on $r$ ensures that
$f=p\tilde{g}$. Thus,
$\init_{\ww}(f)=\init_{\ww}(h)\init_{\ww}(\tilde{g})$, as we wanted to
show. This concludes our proof.
\end{proof}

\subsection{Berkovich skeleta of curves and faithful tropicalization}\label{sec:berk-skel-curv}

In this section, we outline the required background on Berkovich
analytic curves, their skeleta and their relationship with
tropicalizations of curves. For the sake of brevity and simplicity, we
restrict our exposition to the topological aspects of analytic
curves. These features are captured by skeleta of curves. We follow
the approach developed by Baker, Payne and Rabinoff
in~\cite{BPR11,BPRContempMath}.

Let $K$ be an algebraically closed, complete non-Archimedean valued
field $K$ with absolute value $|\,.\,| =\exp(-\val(\,.\,))$. Our main
example of interest is $K=\CC\{\!\{t^{\RR}\}\!\}$, i.e.\ the field of
generalized Puiseux series. Given an algebraic curve $C$ defined over
$K$ we let $C^{\an}$ denote its analytification. The analytification
$A^{\an}$ of an affine curve $\Spec(A)$ is the space of multiplicative
seminorms $\|\; \|\colon A\to \RR_{\geq 0}$ that satisfy the
non-Archimedean triangle inequality $\|f+g\|\leq \max\{\|f\|, \|g\|\}$
and extend the absolute value on $K$. Its topology is the coarsest one
such that all evaluation maps $\operatorname{ev}_f\colon A^{\an}\to
\RR_{\geq 0}$ $\|\cdot\| \mapsto \|f\|$ are continuous for $f\in A$.
The analytification $C^{\an}$ of a general curve $C$ is glued from the
analytification of an affine open cover.  It can be shown that
$C^{\an}$ possesses a piecewise linear structure and it is locally
modeled on an $\RR$-tree~\cite[$\S 5.8$]{BPRContempMath}. The
$K$-points of $C$ are embedded as a subset of the leaves of this
tree. The complement of the set of leaves carries a canonical metric
given by shortest paths.

In~\cite{Berk1990}, Berkovich introduced the notion of \emph{skeleta}
of an analytic space as suitable polyhedral subsets that capture the
topology of the whole space.  They are constructed from semistable
formal models. Equivalently, they can be defined by means of
semistable vertex sets $V$ of $C^{\an}$~\cite[$\S 1.2$,
Theorem~1.3]{BPRContempMath}. They have the structure of a finite
metric graph with vertex set $V$. We denote them by $\Sigma(C,V)$. For
any choice of $V$, there exists a deformation retract
\begin{equation}
\tau_{\Sigma(C,V)}\colon C^{\an}\twoheadrightarrow \Sigma(C,V)\label{eq:retraction}
\end{equation}
(see~\cite{BPRContempMath,Berkovich_LocContractI}). Semistable vertex
sets form a poset under inclusion and induce refinement of the
corresponding
skeleta~\cite[Proposition~3.13(1)]{BPRContempMath}. 

\begin{definition}
  We say $\Sigma(C,V)$ is a \emph{minimal skeleton} of $C^{\an}$ if
  $V$ is minimal.
\end{definition}
Such minimal skeletons exist by~\cite[$\S 4.16$]{BPRContempMath}. The
Stable reduction theorem ensures that if the Euler characteristic of
$C$ is at most 0, then there is a unique set-theoretic minimal
skeleton of $C^{\an}$~\cite[Theorem 4.22]{BPRContempMath}. This is the
case when $C$ is smooth and non-rational. In this situation, we write
$\Sigma(C)$, or $\Sigma(I)$ whenever $C$ is defined by the ideal $I$.

From now on, let us assume that $C$ is a smooth connected algebraic
curve over $K$ and let $\widehat{C}$ denote its smooth completion. Let
$D=\widehat{C}\smallsetminus C$ be its set of punctures. These
punctures are contained in distinct connected components of
$\widehat{C}^{\an}\smallsetminus V$. By construction, a semistable
vertex set $V$ of $C^{\an}$ is also a semistable vertex set of
$\widehat{C}^{\an}$. In particular,
by~\cite[Proposition~3.13]{BPRContempMath} we know that
$\Sigma(\widehat{C},V)\subset \Sigma(C,V)$.  The closure of
$\Sigma(C,V)$ in $\widehat{C}^{\an}$ equals $\Sigma(C,V)\cup D$. We
call it the \emph{extended skeleton} of $\widehat{C}^{\an}$ with
respect to $V$ and the punctures $D$ and we denote it by
$\widehat{\Sigma}(V,D)$. Whenever the minimal skeleton of
$\widehat{C}^{\an}$ is unique, as in Example~\ref{ex:EllipticSkeleta}
below, the extended skeleton depends solely on the set of
punctures. Following the previous notation, when the smooth,
non-rational curve $\widehat{C}$ is defined by an ideal $I$, we write
$\widehat{\Sigma}(I)$ for the complete extended skeleton.

\begin{example} [Elliptic curves]\label{ex:EllipticSkeleta}
  Let $C$ be a smooth elliptic curve defined over $K$. If
  $\widehat{C}$ has good reduction, then the minimal skeleton of
  $\widehat{C}^{\an}$ is a point. If $\widehat{C}$ has bad reduction,
  then the minimal skeleton of $\widehat{C}$ is homeomorphic to a
  circle: its corresponding semistable vertex set is a point~\cite[$\S 7.1$]{BPR11}. Larger
  semistable vertex sets $V$ will yield larger skeleta obtained from
  $\Sigma(\widehat{C})$ by attaching finite trees to this circle along
  points in $V\cap \Sigma(\widehat{C})$.
\end{example}
From the previous discussion, it is clear that skeleta of analytic
curves share many properties with tropicalizations of algebraic
curves.  Their interplay was studied in depth by Baker, Payne and
Rabinoff in~\cite{BPR11}. As we next discuss, the precise relationship
is captured by the tropicalization map and Thuillier's non-Archimedean Poincar\'e-Lelong
formula~\cite[Theorem 5.15]{BPRContempMath}.

Let $C\subset (K^*)^n$ be an embedded curve, and fix a basis
$\{y_1,\ldots, y_n\}$ of the character lattice of the torus. Let
$f_i\in K(C)$ be the image of $y_i$ for $i=1,\ldots, n$. The
tropicalization map $\trop\colon C(K)\to \RR^n$ given by $x \mapsto
(\log(|f_1(x)|), \ldots, \log(|f_n(x)|)$ extends naturally to a
continuous map
  \begin{equation} \label{eq:tropMapFull}
\trop\colon  C^{\an} \to \RR^n \qquad \|\cdot\| \mapsto (\log (\|f_1\|), \ldots, \log(\|f_n\|)).
\end{equation}
The image of this map is precisely the tropical curve
$\Trop(C)$~\cite[$\S 3$]{Gubler13}. Given any semistable vertex set $V$ of
$C^{\an}$, the map~\eqref{eq:tropMapFull} factors through the retraction
$\tau_{\Sigma(C,V)}$ by~\cite[Theorem 5.15 (1)]{BPRContempMath}. In particular, the resulting map 
\begin{equation}
\trop\colon
\Sigma(C,V)\twoheadrightarrow \Trop(C)\label{eq:tropMap}
\end{equation}
 is a
surjection.  This last map will be our main focus of interest.

By the Poincar\'e-Lelong formula~\cite[Theorem 5.15]{BPRContempMath},
the maps $\trop$ from~\eqref{eq:tropMapFull} and~\eqref{eq:tropMap}
are piecewise affine, with integer slopes. Furthermore, they are
affine on each edge of the skeleton $\Sigma(C,V)$.  The stretching
factor on each edge is known as its \emph{relative multiplicity}.  If
an edge $e$ gets contracted to a single point in $\Trop(C)$, we set
$m_{\operatorname{rel}}(e)=0$. The map $\trop$ is \emph{harmonic},
i.e.\ the image of every point in $C^{\an}$ and $\Sigma(C,V)$ under
$\trop$ is \emph{balanced} in the following sense: only finitely many
edges in the star of a point $x$ in $C^{\an}$ (resp.~$\Sigma(C,V)$)
are not contracted by $\trop$, and these edges satisfy the identity
\begin{equation}
\sum_{e \in T_x} d_e\trop(x)= 0\label{eq:balancing}.
\end{equation}
Here, $T_x$ denote the tangent directions of $x$, i.e.\ the nontrivial
geodesic segments starting at $x$, up to equivalence at $x$ (as
in~\cite[$\S 5.11$]{BPRContempMath}). The outgoing slope $d_e\trop(x)$
is $0$ if $\trop$ contracts $e$ and it equals
$m_{\operatorname{rel}}(e)$ times the primitive direction of the edge
$e'$ of $\Trop(C)$ that contains the (possibly unbounded) segment
$\trop(e)$.

By refining the polyhedral structure of $\Trop(C)$ we may assume that
the map from~\eqref{eq:tropMap} is a morphism of 1-dimensional
complexes. The balancing condition yields the following identity
between tropical and relative multiplicities, as
in~\cite[Proposition~4.24]{BPR11}:
\begin{equation}
  m_{\Trop}(e')=\sum_{\substack{ e\in \Sigma(C,V)\\ \trop(e)=e'}}m_{\operatorname{rel}}(e).\label{eq:tropMultVsRelMult}
\end{equation}

By~\cite[Proposition~4.24]{BPR11}, this formula can also be used to
relate tropical and relative multiplicities of vertices $e'$ on
tropical curves and vertices $e$ of skeleta of analytic curves, when
the map $\trop$ from~\eqref{eq:tropMap} is a morphism of 1-dimensional
complexes.  As in the case of edges, the tropical multiplicity of a
vertex $\ww$ of $\Trop(C)$ counts the number of irreducible components
(with multiplicities) in the initial degenerations of the input ideal
defining $I$ with respect $\ww$. Rather than giving the precise
definition for the relative multiplicity of a vertex $v$ in
$\Sigma(C,V)$, we present two of its crucial properties, as
in~\cite[Corollary 6.12]{BPR11}. Namely, $m_{\operatorname{rel}}(v)$
is a non-negative integer and $m_{\operatorname{rel}}(v)>0$ if and
only if $v$ belongs to an edge of $\Sigma(C,V)$ mapping
homeomorphically onto its image via $\trop$.

\begin{definition} Consider a skeleton $\Sigma(C,V)$ of $C^{\an}$ and
  a finite subgraph $\Gamma$ on it.  We say a closed embedding
  $C\hookrightarrow (K^*)^n$ \emph{faithfully represents} $\Gamma$ if
  $\trop$ maps $\Gamma$ homeomorphically and isometrically onto its
  image in $\RR^n$.
\end{definition}

Using embeddings of curves in proper toric varieties $Y_{\Delta}$ that
meet the dense torus, we can extend the previous definition to
complete curves. We consider those toric varieties $Y_{\Delta}$ for
which the morphism $\widehat{C} \to Y_{\Delta}$ is a closed immersion
and use the \emph{extended tropicalization maps} from~\cite{Pay09},
obtained by gluing the previous constructions on each toric strata
along open inclusions, with the convention that $\log(0):=-\infty$.

  We say that $\trop\colon \widehat{C}^{\an} \to \Trop(\widehat{C})$
  is \emph{faithful} if it faithfully represents a skeleton of
  $\widehat{C}$.
%
  By definition, a faithful tropicalization of $C$ restricts to a
  homeomorphism from a suitable skeleton of $C^{\an}$ to a subgraph of
  the tropical curve $\Trop(C)$.  Thus, constructing an embedding of
  the given curve that yields such a homeomorphism can be viewed as a
  first step towards a faithful tropicalization of curves. Relative
  multiplicities on edges and the isometric requirements should be
  address in a second step.  

  In Section~\ref{sec:trop-ellipt-curv}, we focus our attention on
  tropical faithfulness of plane elliptic cubics with bad reduction,
  embedded in $(K^*)^2$ or in a surface in $(K^*)^n$. Their
  completions admit a closed embedding $\widehat{C}\hookrightarrow
  \mathbb{P}^{n-1}$.  The minimal skeleton $\Sigma$ of
  $\widehat{C}^{\an}$ lies in $C^{\an}$ and is homeomorphic to a
  circle. Our goal is to find a linear re-embedding of a given curve
  that faithfully represents $\Sigma$.

We first discuss how to detect non-closed embeddings of skeleta by looking at the tropical curve.
\begin{definition}
  Let $\Trop(g)$ be a tropicalization of the plane curve defined by
  $g$, and $v$ a vertex of $\Trop(g)$.  We say that $v$ is
  \emph{locally reducible} if the star of $v$ in $\Trop(g)$ is a
  reducible $1$-dimensional complex that is balanced at $v$, i.e.\ if
  it can be written as the union of two non-zero complexes with
  multiplicities that are balanced $v$.  In particular, if
  $\Star_{\Trop(g)}(v)$ is the union of $s$ edges $e_1,\ldots, e_s$
  adjacent to $v$ with multiplicities $m_1,\ldots, m_s$, we can find
  $\tilde{m}_1,\ldots, \tilde{m}_s$ with $0\leq \tilde{m}_i\leq m_i$
  for all $i$ such that the resulting complex with multiplicities is
  balanced at $v$, contains an edge of positive multiplicity and it
  does not agree with $\Star_{\Trop(g)}(v)$ as complexes with
  multiplicities.
\end{definition}

\begin{lemma}\label{lem:nonHomeo} 
  Consider a non-rational smooth curve $C$ defined by an ideal $I$ and
  let $\widehat{\Sigma}(I)$ be extended skeleton defined with respect
  to the set of punctures $D_I$.  Assume that $\trop\colon
  \widehat{\Sigma}(I)\smallsetminus D_I\to \Trop(I)$ is not a closed
  embedding.  Then one of the following conditions hold:
\begin{enumerate}[(1)]
\item $\Trop(I)$ has an edge of higher multiplicity, or a locally reducible vertex $v$ with
  $m_{\Trop}(v)\geq 2$;
\item $\Trop(I)$ faithfully represents a unique subgraph $\Gamma$ of $\widehat{\Sigma}(I)\smallsetminus D_I$. 
\end{enumerate}
\end{lemma}
\begin{proof} To simplify notation, write
  $\Sigma:=\widehat{\Sigma}(I)\smallsetminus D_I$. After refining the
  structure of $\Sigma$ and $\Trop(I)$, we may assume without loss of
  generality that $\Sigma$ has no loop edges and that $\trop$ is a map
  of connected 1-dimensional abstract complexes.

  Since $\trop$ is not a closed embedding, one of the following
  conditions hold:
  \begin{enumerate}[(i)]
  \item the images of several edges intersect in
  more than a point;
\item there exists a vertex $v$ of $\Trop(I)$ where the fiber
  $\trop^{-1}(v)$ in $\Sigma$ is not connected;
\item there exists a vertex $v$ of $\Trop(I)$ such that $\trop^{-1}(v)$ is connected and it is not a singleton.
  \end{enumerate}

  First, assume that (i) holds and let $e$ be a segment of an edge of
  $\Trop(I)$ where the images of several edges overlap. We conclude
  from~\eqref{eq:tropMultVsRelMult} that $e$ lies in an edge of
  $\Trop(I)$ of higher multiplicity.

We now analyze conditions (ii) and (iii).
We consider the
  stars of all vertices $\rho\in \trop^{-1}(v)$ in the abstract cell
  complex $\Sigma$.  By~\eqref{eq:balancing} we know that the images
  of all stars $\Star_{\Sigma}(\rho)$ under the tropicalization map
  are balanced at $v$. In particular, 
\begin{equation}
\Star_{\Trop(I)}(v)= \bigcup_{\rho \in V(\trop^{-1}(v))} \trop(\Star_{\Sigma}(\rho)).\label{eq:StarDecomposition}
\end{equation}
The decomposition in the right-hand side
of~\eqref{eq:StarDecomposition} contains at least one non-singleton
component. In order to show that $v$ is a locally reducible vertex, we
seek to find two vertices $\rho,\rho'\in V(\Sigma)$ where
$\trop(\Star_{\Sigma}(\rho))$ and $\trop(\Star_{\Sigma}(\rho'))$ are
both nontrivial. In this situation,~\cite[Proposition 4.24, Corollary
6.12]{BPR11} ensure that $m_{\Trop}(v)\geq
m_{\operatorname{rel}}(\rho) +m_{\operatorname{rel}}(\rho')\geq
1+1=2$.

 To simplify notation, fix $\Sigma':= \Sigma\smallsetminus
 \trop^{-1}(v)$. Assume (ii) holds, and decompose $\trop^{-1}(v)$ into
 its connected components $\{\Sigma_1,\ldots, \Sigma_r\}$, where
 $r\geq 2$. Each component is closed in $\Sigma$. Since $\Sigma$ is
 connected, we conclude that $\Sigma_i \cap
 \overline{\Sigma'} 
 \neq \emptyset$ for all $i=1,\ldots, r$.  Since $\trop$ is a morphism
 of complexes, we can pick a vertex $\rho_i$ in $\Sigma_i \cap
 \overline{\Sigma'}$ for each $i=1,\ldots,r$. By construction,
 $\trop(\Star_{\Sigma}(\rho_i))\neq v$ for all $i=1,\ldots, r$.  We
 conclude that $v$ is locally reducible and $m_{\Trop}(v)\geq r\geq
 2$.

 Finally, assume (iii) holds. Then the fiber $\trop^{-1}(v)$ in
 $\Sigma$ is a connected graph with at least two vertices. If $v$ is
 not locally reducible, then the
 decomposition~\eqref{eq:StarDecomposition} is trivial, and so there
 is a unique vertex $\rho$ of $\trop^{-1}(v)$ whose star in $\Sigma$
 does not map entirely to $v$ under $\trop$. We conclude that $
 \trop^{-1}(v)\cap \overline{\Sigma'}=\{\rho\}$ and $\{\rho\} \cup
 \Sigma'$ is connected and surjects onto $\Trop(I)$ via $\trop$.  We
 conclude that each edge $e$ in $\Star_{\Trop(I)}(v)$ is the image of
 at least one edge $e'$ in $\Star_{\Sigma'\cup \{\rho\}}(\rho)$. Thus,
 we can construct a subgraph $\Gamma$ in $\Star_{\Sigma'\cup
   \{\rho\}}(\rho)$ that is homeomorphic to $\Star_{\Trop(I)}(v)$ by
 $\trop$.  In addition, assuming condition (i) does not occur, we know
 that $\Gamma=\Star_{\Sigma'\cup \{\rho\}}(\rho)$ and $\trop$ induces
 an isometry between $\Gamma$ and $\Star_{\Trop(I)}(v)$.

As a consequence, if condition \emph{(1)} in the statement fails, by
iterating the previous construction over all vertices of $\Trop(I)$,
we can find a unique subgraph $\Gamma$ of
$\widehat{\Sigma}(I)\smallsetminus D_I$ that maps isometrically to
$\Trop(I)$ under the map $\trop$.  This concludes our proof.
\end{proof}

\begin{remark}
  From the proof of Lemma~\ref{lem:nonHomeo} we can also extract the
  following information. Assume that the images under $\trop$ of two
  adjacent edges $e$ and $e'$ of $\Sigma$ with a unique common
  endpoint $w$ are two line segments that partially overlap. Call
  $\rho$ and $\rho'$ the non-common endpoints of $e$ and $e'$, and
  assume $\trop(\rho')\in \trop(e)$. Then, the point $\trop(\rho')$
  will be a locally reducible vertex of $\Trop(I)$ and its star
  contains the straight line with direction $\trop(e)$. In the case of
  complete overlap, the vertex $\trop(\rho)=\trop(\rho')$ will also be
  locally reducible. We know that $\Star_{\Trop(I)}(\trop(\rho))$
  contains the high multiplicity edge $\trop(e)$, but we cannot
  guarantee that it contains a straight line. Finally, when the edges
  $e$ and $e'$ of $\Sigma$ have two common endpoints and do not get
  contracted by $\trop$, their image $\trop(e)$ will be contained in
  an edge of $\Trop(I)$ of multiplicity $m>1$.
\end{remark}

The following special instance of Lemma~\ref{lem:nonHomeo} will be
useful in Section~\ref{sec:repair-elliptic}, were we discuss elliptic
plane cubics with bad reduction.
\begin{corollary}\label{cor:nonHomeoCycle}
  Let $C$ be an elliptic cubic curve over $\PS$ with bad reduction,
  embedded linearly by an ideal $I$. Assume $\Trop(I)$ contains a
  cycle but $\trop\colon \widehat{\Sigma}(I)\smallsetminus D_I\to
  \Trop(I)$ is not faithful on the cycle. Then, the cycle contains a
  locally reducible vertex $v$ and $m_{\Trop}(v)= 2$.
\end{corollary}
\begin{proof}
  We write $\Sigma:=\widehat{\Sigma}(I)\smallsetminus D_I$, and assume
  that $\trop$ is a morphism of 1-dimensional complexes.  Since $C$
  has bad reduction, we know that $\widehat{\Sigma}(I)\smallsetminus
  D_I$ contains a unique cycle $\Gamma$. The tropical cycle is
  contained in $\trop(\Gamma)$, but the latter may also contain other
  edges of $\Trop(I)$.

  Next, we analyze $\trop(\Gamma)$. Since $C$ is defined by a cubic
  polynomial $g$ in the plane, all the edges in the cycle of
  $\Trop(C)$ have multiplicity 1. Given an edge $e$ of the cycle of
  $\Trop(I)$, expression~\eqref{eq:tropMultVsRelMult} ensures that
  exactly one edge $e'$ of $\Sigma$ lies in $\trop^{-1}(e)$ and,
  moreover, this edge lies in $\Gamma$ and $trop$ induces an isometry
  between $e$ and $e'$. Since $\trop$ is not faithful on the cycle of
  $\Trop(I)$, we know that $\Gamma$ contains at least one edge that
  either gets contracted by $\trop$ or that map to an edge of
  $\Trop(I)$ outside the cycle. In both cases, we can find two
  distinct vertices $\rho, \rho'$ of $\Gamma$ that map to the same
  vertex $v$ in the cycle of $\Trop(I)$ and are contained in two edges
  of $\Gamma$ that are mapped isometrically to edges in the cycle of
  $\Trop(I)$. By~\cite[Corollary 6.12]{BPR11}, $m_{\Trop}(v)\geq
  1+1=2$.  The decomposition \eqref{eq:StarDecomposition} ensures that
  $v$ is locally reducible, as desired.

For the reverse inequality, we analyze the combinatorics of the
support of $\init_{v}(g)$, i.e. of the dual cell to $v$ in the Newton subdivision of $g$. By Figure~\ref{fig:shapesandfeeding}, this
support is a trapezoid of height 1 and one of whose basis has length
1. Therefore, $\init_{v}(g)$ has at most two components, i.e.\
$m_{\Trop(v)}\leq 2$. This concludes our proof.
\end{proof}
Consider a smooth non-rational plane curve in $ (K^*)^2$ defined by an
irreducible polynomial $g(x,y)\in K[x,y]$ and its linear re-embedding
via the ideal $I_{g,f}\subset K[x,y, z]$ as in
Section~\ref{sec:trop-modif-line}.  This re-embedding alters the
skeleton of the analytic curve in a concrete way. Consider completions
of these two curves, their sets of punctures $D_g$ and $D_{I_{g,f}}$
and the corresponding extended skeleta $\widehat{\Sigma}(g)$ and
$\widehat{\Sigma}(I_{g,f})$. Notice that $D_g\subseteq
D_{I_{g,f}}$. These skeleta only differ by some additional ends that
we attach to $\widehat{\Sigma}(g)$ to obtain
$\widehat{\Sigma}(I_{g,f})$ (see
Figure~\ref{fig:OriginalEllipticModification}).  The bounded part of
$\widehat{\Sigma}(g)$ can be identified with the corresponding bounded
part of $\widehat{\Sigma}(I_{g,f})$ using the following commutative
diagram: \begin{equation} \xymatrix{C^{\an}
    \ar[rr]^-{\tau_{{\widehat{\Sigma}}(I_{g,f})}}\ar[drr]_-{\tau_{{\widehat{\Sigma}(g)}}}
    &&{\widehat{\Sigma}(I_{g,f})}\smallsetminus D_{I_{g,f}} \ar@{>>}[r]^-{\trop}\ar[d] & \Trop(I_{g,f}) \ar@{>>}[d]^{\pi_{XY}}\\
    & & {\widehat {\Sigma}(g)} \smallsetminus {D_g}
    \ar@{>>}[r]^{\trop}& \Trop(g).  }
\label{eq:contr/fold}\end{equation}
This diagram allows us to define two key notions: decontraction and unfolding of edges via linear re-embeddings.
\begin{definition} 
  If $\trop\colon \widehat{\Sigma}(g)\smallsetminus D_g\to \Trop(g)$
  contracts a fixed bounded edge but $\trop\colon
  \widehat{\Sigma}(I_{g,f})\smallsetminus D_{I_{g,f}}\to
  \Trop(I_{g,f})$ does not, we say that the linear
  re-embedding \emph{decontracts} this edge.  

  Assume next that a segment in a bounded edge $e$ of $\Trop(g)$ is
  obtained by overlapping the images of several edges of
  $\widehat{\Sigma}(g)\smallsetminus D_g$ in more than one
  point. Refine the structure of $\Trop(g)$ and let $e$ be this
  segment. If $\pi^{-1}_{XY}(e)$ is the union of images of finitely
  many edges from $\widehat{\Sigma}(I_{g,f})\smallsetminus
  D_{I_{g,f}}$ that pairwise intersect in at most one point, we say
  that the linear re-embedding \emph{unfolds} the edge $e$.
\end{definition}
The union of edges that unfolds $e$ need not be connected. Example~\ref{ex:twoStepsExample} and Figure~\ref{fig:TwoStepPart1}
show the decontraction of an edge.  Example~\ref{ex:64a} and
Figure~\ref{fig:64aTropical} illustrate the unfolding phenomenon. In
both cases, we recover the curve $\Trop(I_{g,f})$  from the drawn
projections $\pi_{XY}$ and $\pi_{ZY}$ using
Lemma~\ref{lem:ModifViaProjections}.

\subsection{$A$-discriminants}
\label{sec:discriminants}
The notion of $A$-discriminants for configurations of points in
$\ZZ^k$ was introduced and further developed by Gelfand, Kapranov and
Zelevinsky in~\cite{GKZ}. We present the theory in its original
formulation for Laurent polynomials. In our applications we only deal
with polynomials with non-negative exponents defined over
$\PS$.

Throughout this section, we let $K$ be an algebraically closed
field. We fix a configuration $A$ of $m$ points in $\ZZ^k$, and a
Laurent polynomial supported on $A$:
\[
g(\underline{x})=\sum_{a\in A} c_{a}\;\underline{x}^{a}  \quad \in\; K[x_1^{\pm},\ldots, x_k^{\pm}].
\]
We  use the multiplicative notation $\underline{x}^{a_i}:=
x_1^{a_{i1}}x_2^{a_{i2}}\ldots x_k^{a_{ik}}$ if $a_i=(a_{i1},\ldots, a_{ik})\in A$.

For generic choices of coefficients $(c_{a_i})_{i=1}^m$, the
polynomial $g$ has no singularities in the algebraic torus
$(K^*)^k$. However, for special choices of coefficients, singularities
do appear. Such special situations (and their algebraic closure) are
determined by the ideal $J_c=J\cap K[c_{a_1}, \ldots, c_{a_m}]$, where
\[
J=\operatorname{Rad}\big(\langle g(\underline{x}),\frac{\partial g}{\partial
  {x_1}}(\underline{x}),\ldots, \frac{\partial g}{\partial
  {x_k}}(\underline{x})\rangle\big) \subset K[c_{a_1},\ldots, c_{a_m}][x_1^{\pm},\ldots,x_k^{\pm}].
\]
It can be shown that whenever $J_c$ is a principal ideal, its unique
generator is irreducible and can be defined over $\ZZ$. The
\emph{$A$-discriminant} $\Delta_A$ is the unique (up to sign)
irreducible polynomial with integer coefficients in the unknowns
$(c_a)_{a\in A}$ defining $J_c$. If $J_c$ is not principal, we set $\Delta_A=1$ and
refer to $A$ as a \emph{defective} configuration. 

As an example, we compute the $A$-discriminant of the trapezoid in
Figure~\ref{fig:trapezoid}, which plays a key role in
Section~\ref{sec:repa-trop}. 
\begin{lemma}\label{lm:trapezoidDiscriminant}
  Assume $n,s\geq 1$. Then, the discriminant of the trapezoid $\cP$ in
  Figure~\ref{fig:trapezoid} equals the Sylvester resultant
  $\operatorname{Res}(h_1,h_2)$ of the univariate polynomials
  $h_1(x)=a_0+a_1x+\ldots+a_nx^n$ and $h_2(x)=b_0+b_1x+\ldots+
  b_sx^s$. In particular, when $s=1$ we obtain
\begin{equation}
 \Delta_{\cP}=a_0\,b_1^n+\ldots+(-1)^{i}b_0^i\,b_1^{n-i}\,a_i+\ldots +(-1)^n b_0^n\,a_n.\label{eq:3}
 \end{equation}

 The same formulas hold if we pick any configuration $A$ of lattice
 points in $\cP$ containing all four vertices of the trapezoid, after
 replacing the corresponding variables among $a_1,\ldots, a_{n-1}$ by
 zero.
\end{lemma}
\begin{proof} 
  Since $A$-discriminants are invariant under affine transformations
  of the lattice $\ZZ^2$, we may assume that the trapezoid has
  vertices $(0,0), (p,1), (p+s,1)$ and $(0,n)$. Furthermore, $\cP$ is
  not a pyramid, so we know the planar configuration $A$ is not
  defective.  We fix a polynomial $h$ with support on the given
  trapezoid, and compute its two partial derivatives:
\begin{equation}
  h(x,y) = h_1(x) + y x^p
  h_2(x),\qquad \frac{\partial{h}}{\partial x}(x,y)= h_1'+y(h_2'+ px^{p-1} h_2), \qquad
  \frac{\partial{h}}{\partial y}(x,y)=h_2 x^p.\label{eq:1}
\end{equation}
Let $(a_0,\ldots, b_s)$ be a general point where the discriminant
vanishes. Then, $h$ admits a singular point $(x_0,y_0)$ in the torus
$(K^*)^2$. In particular, $h_2(x_0)=0$ and so
$0=h(x_0,y_0)=h_1(x_0)$. Thus, both $h_1$ and $h_2$ have a common
solution $x_0\in K^*$, so $\operatorname{Res}(h_1,h_2)=0$. We conclude
that $\operatorname{Res}(h_1,h_2)$ divides $\Delta_{\cP}$.  Since both
polynomials are irreducible over $\ZZ[c_a: a\in A]$, the result
follows.

Now, let $s=1$ and write $h_2:=b_0+b_1x$.  Assume $(x_0,y_0)$ is a
singular point of $h$ in the torus $(K^*)^2$. From~\eqref{eq:1} we
conclude that $x_0=-b_0/b_1$ and we can use the equation
$\frac{\partial{h}}{\partial x}=0$ to find the value of $y_0$. Since,
in addition, $h(x_0,y_0)=0$ we obtain
\begin{equation}
\begin{aligned}
  0&=b_1^n\big(a_0+a_1(-\frac{b_0}{b_1})+\ldots +
  a_i(-\frac{b_0}{b_1})^i+\ldots+ a_n(-\frac{b_0}{b_1})^n
  +b_0\,y_0+b_1(-\frac{b_0}{b_1})y_0\big)\\ &= a_0b_1^n
  -a_1b_0b_1^{n-1}+\ldots + (-1)^{i}a_ib_0^ib_1^{n-i}+\ldots + (-1)^n a_nb_0^n,
\end{aligned}\label{eq:2}
\end{equation}
as we wanted to show.

Conversely, if the right-hand side of~\eqref{eq:3} vanishes, then any
point $(x_0,y_0)$ constructed from the vanishing of the
partials~\eqref{eq:1} is a singularity of $\{h=0\}$. The singularity
lies in the torus if and only if $\frac{\partial{h}}{\partial
  x}(-b_0/b_1,0)\neq 0$. The latter is an open condition in the
coefficients $(a_1,\ldots, a_n,b_0,b_1)$, and it is independent on the
variable $a_0$. Since the bottom expression in~\eqref{eq:2} has degree
1 in $a_0$, we can find a unique $a_0\in K^*$ that solves the
equation~\eqref{eq:2} for a generic point $(a_1,\ldots, b_1)\in
(K^*)^{n+2}$. For this choice, the unique solution $(x_0,y_0)$ is a
singularity of $h$ in the torus.

For the third claim in the statement, it suffices to notice that all
the arguments stated above hold if we replace any of the coefficients
$a_1,\ldots, a_{n-1}$ in $h$ by zero. This concludes our proof.
\end{proof}

\begin{corollary}\label{cor:highMultAndDiscrim}
  Fix a polynomial $h$ with support contained in the trapezoid $\cP$
  in Figure~\ref{fig:trapezoid} where $a_0,a_n,b_0,b_s\in K^*$. Then,
  $h$ is reducible over $K[x^{\pm},y^{\pm}]$ if and only if
  $\Delta_{\cP}(a_0,\ldots, b_s)=0$.
\end{corollary}
\begin{proof} The result is an easy consequence of
  Lemma~\ref{lm:trapezoidDiscriminant}.  Since $h$ has degree 1 in
  $y$, a simple calculation shows that $h$ factorizes over
  $K[x^{\pm},y^{\pm}]$ if and only if $h_1$ and $h_2$ have a common
  solution in $K^*$, that is, if
  $0=\operatorname{Res}(h_1,h_2)=\Delta_{\cP}$.
\end{proof}

Plane tropical curves are dual to coherent (or regular) subdivisions of lattice polygons in
$\RR^2$. Each vertex or edge $\tau$ in the tropical curve $\Trop(g)$
is dual to a marked 2-dimensional polytope or marked edge
$\tau^{\vee}$ in the Newton subdivision of $g$. By abuse of notation, we define the discriminant of $\tau$
as the discriminant of its marked dual cell, i.e.
 \[
\Delta_{\tau}:=\Delta_{\tau^{\vee}}\in \ZZ[c_a:a\in \tau^{\vee}].\]
In Section~\ref{sec:repa-trop} we use these polynomials to measure local
faithfulness of the tropicalization map.

\section{Repairing tropicalizations}\label{sec:repa-trop}

In this section, we present our two main technical tools for repairing embedding of plane curves whose tropicalization maps are
non-closed embeddings of Berkovich skeleta, as in Lemma~\ref{lem:nonHomeo}. Theorems~\ref{thm:redVertex} and~\ref{thm:fatEdge} explain how  to locally
repair these bad behaviors by linear re-embeddings while
preserving the structure elsewhere under some restriction on the
locally reducible vertices or high multiplicity bounded edges.
Remark~\ref{rem:DisclaimerThmDecontracting} discusses possible
extensions to other types of locally reducible vertices.  In
Section~\ref{sec:repair-elliptic} we combine these two theorems to
give a symbolic algorithm to repair the cycle of a plane tropical
elliptic cubic (see Theorem~\ref{thm:repairEll} and Algorithm~\ref{alg:repairElliptic}).

Throughout this section, we assume our input to be a smooth
non-rational curve  defined by a polynomial $g\in \PS[x,y]$,
and we consider its tropicalization $\Trop(g)$ as a subvariety of the
torus $(\PS^*)^2$.  We base change the algebraic
curve to $\CC\{\!\{t^{\RR}\}\!\}$, and consider the set of punctures
$D_g$ in a smooth completion of the new curve.  We write
$\widehat{\Sigma}(g)$ for the extended skeleton of the complete
analytic curve  with respect to the set of punctures $D_g$, as
defined in Section~\ref{sec:berk-skel-curv}. Notice that the base
change operation does not affect the tropical curve $\Trop(g)$
by~\cite[Proposition 3.7]{Gubler13}.  We make the following genericity
assumption on $g$:
\begin{convention}\label{genericityassumption}
We assume $g$ is \emph{generic} in the sense that if a non-trivial linear combination of its Puiseux series coefficients does not have the expected valuation because the initial term cancels, then the valuation does not reach $\infty$, i.e.\ the linear combination does not cancel completely.
\end{convention}

For simplicity, we choose to formulate Theorems~\ref{thm:redVertex}
and~\ref{thm:fatEdge} for embedded plane curves in $(\PS^*)^2$. Since
both theorems are local in nature, they also hold if we embed
$\Trop(g)$ linearly in $\RR^r$ and the locally reducible vertex $v$ or
the edge of high multiplicity, respectively, is contained in the
interior of a top-dimensional cone of the tropical plane in
$\RR^r$. These more general versions and
Lemma~\ref{lem:ModifViaProjections} will enable us to iterate this
procedure if one linear tropical modification does not suffice to
locally repair the input tropical curve.

By refining structures, we always assume
that $\trop\colon \widehat{\Sigma}(g)\smallsetminus D_g\to \Trop(g)$
is a morphism of 1-dimensional polyhedral complexes. Our first result concerns a special class of locally reducible
vertices of $\Trop(g)$. It can be further extended by unimodular
transformations (see Remark~\ref{rem:unimodular}). 
\begin{theorem}\label{thm:redVertex}
  Let $v$ be a locally reducible vertex of $\Trop(g)$ which is dual to
  Figure~\ref{fig:trapezoid}.  We assume that $g$ satisfies the
  genericity assumption \ref{genericityassumption}.  Then, $\Delta_v$
  vanishes at $\init_v(g)$ if and only if there exist two points
  $\rho, \rho'$ of $\trop^{-1}(v)\subset
  \widehat{\Sigma}(g)\smallsetminus D_g$ satisfying the following
  conditions:
  \begin{enumerate}[(i)]
  \item the stars of $\rho$ and $\rho'$ in
    $\widehat{\Sigma}(g)\smallsetminus D_g$ are not contracted by
    $\trop$,
  \item there exists a subgraph $\Gamma$ of
    $\widehat{\Sigma}(g)\smallsetminus D_g$ containing the points
    $\rho$ and $\rho'$ that is partially contracted and/or folded by
    $\trop$.
  \end{enumerate}

  Moreover, in this situation, there is a linear re-embedding
  $I_{g,f}$ of the curve corresponding to the tropical modification
  along the vertical line $L$ through $v$ that decontracts/unfolds
  edges of $\widehat{\Sigma} (g)\smallsetminus D_g $ mapping to
  $\Trop(g)\cap L$ under $\trop$.
\end{theorem}

Theorem~\ref{thm:redVertex} should be interpreted as follows. Assume
$\Trop(g)$ contains a locally reducible four-valent vertex $v$ whose
star contains a line and two other non-parallel rays of multiplicity
one. Then the local faithfulness of $\Trop(g)$ at $v$ is measured by
the discriminant of $v$. If this discriminant vanishes, we can use a
linear tropical modification of the ambient space to find a subgraph
of $\widehat{\Sigma}(g)$ that is homeomorphic to the star of $v$ in
$\Trop(g)$ via the tropicalization map.

\begin{figure}[htb]
\includegraphics[scale=0.2]{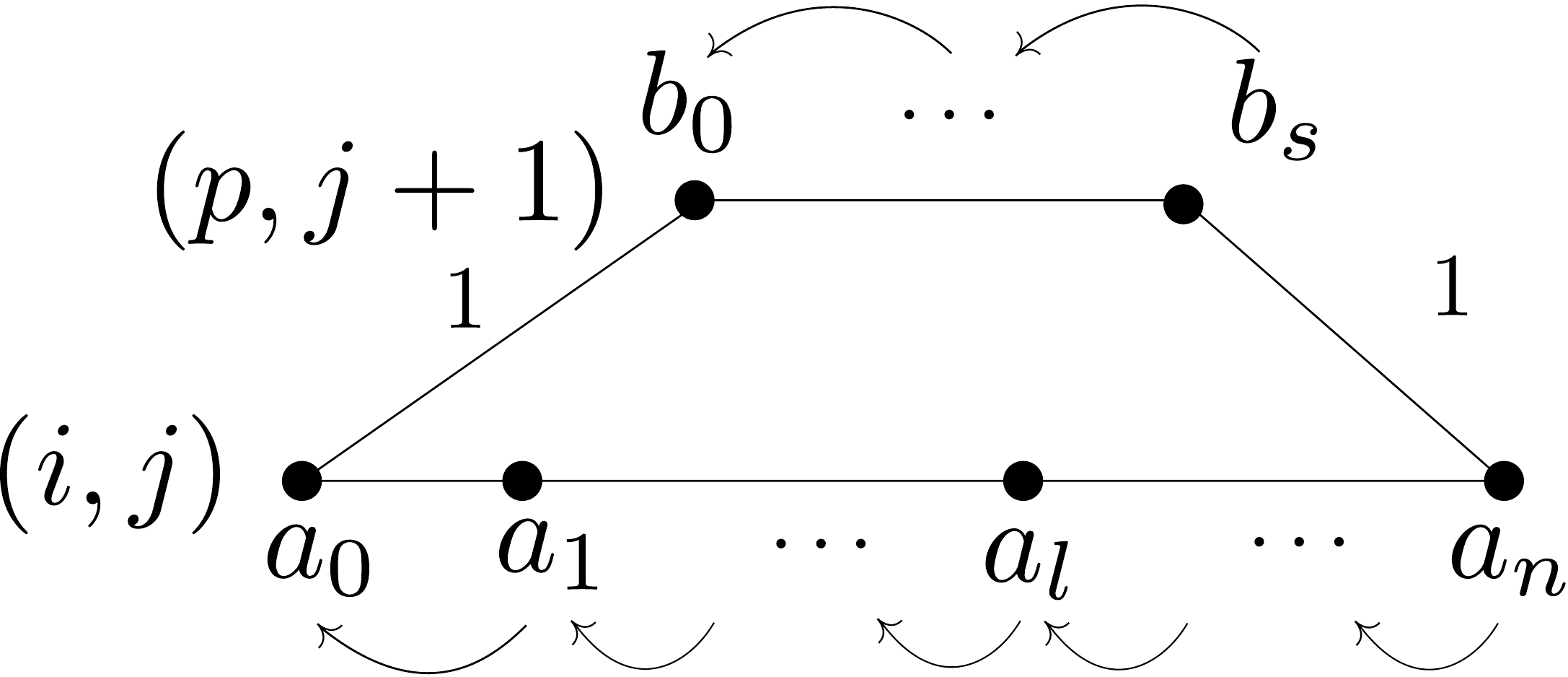}
\caption{The locally reducible vertex in $\Trop(g)$ that can be
  repaired with a linear tropical modification.  The labels
  $a_0,\ldots, a_n, b_0, \ldots, b_s \in \CC$ correspond to the
  coefficients of $\init_v(g)$.  With the exception of the four
  vertices, the remaining coefficients are allowed to be 0. The two
  bases of the trapezoid define two univariate polynomials
  $h_1=a_0+\ldots +a_n x$ and $h_2(x):=b_0+\ldots+b_s x^s$ of degrees
  $n$ and $s$, respectively. The arrows indicate the feeding process
  in the $z$ variable induced by the linear re-embedding $x\mapsto
  z-\AA$ with $\AA\in \PS$.  }\label{fig:trapezoid}
\end{figure}

\begin{remark}\label{rem:unimodular}
  A \emph{unimodular transformation} of $\ZZ^2$ is a linear
  $\ZZ$-invertible map $A\colon\ZZ^2\to \ZZ^2$ associated to a
  monomial change of coordinates $\alpha$ in $\PS[x^{\pm},y^{\pm}]$.
  Using the maps $A$ and $\alpha$ we can make other locally reducible
  vertices dual to Figure~\ref{fig:trapezoid}, and apply
  Theorem~\ref{thm:redVertex} to repair the corresponding tropical
  curve locally around $v$.  We compose the linear re-embedding and
  its lifting function $f$ and use the ideal $I_{g,f\circ
    \alpha^{-1}}:=\langle g(x,y), z-f\circ\alpha^{-1}(x,y)\rangle
  \subset \PS[x^{\pm},y^{\pm}, z^{\pm}]$ to give new coordinates to
  our curve. We can view $\Trop(g)$ inside the two charts
  $\sigma_1\cup\sigma_2$ of $\Trop(I_{g,f\circ\alpha^{-1}})$ by means
  of the inverse of the linear map $A\times \id_Z:\ZZ^3\to \ZZ^3$. For
  an illustration, see Example~\ref{ex:non-ellipticDecontraction}.

Notice that, with few exceptions, these re-embeddings are non-linear and the same outcome cannot be achieved by using only linear tropical modification. Namely, only when the line through $v$ has slopes $0$ or $1$, the corresponding monomial changes of coordinates are linear themselves, and so $I_{g,f\circ\alpha^{-1}}$ gives a linear re-embedding. Furthermore, in these two cases, we can easily adapt the techniques of Theorem~\ref{thm:redVertex} to modify along these lines directly without the need of precomposing with these monomial map  (see Figure~\ref{fig:shapesandfeeding}).   We use this strategy in  Theorem~\ref{thm:repairEll} to repair bad embeddings of plane elliptic  cubics. 
\end{remark}

Our second result concerns the unfolding of edges of high
multiplicity. As opposed to Theorem~\ref{thm:redVertex}, the
non-vanishing of the discriminant of the edge $e$ at $\init_e(g)$ 
detects non-faithfulness.
\begin{theorem}\label{thm:fatEdge}
  Let $e$ be a vertical bounded edge of $\Trop(g)$ of multiplicity
  $n\geq 2$ whose endpoints have valency 3. If the discriminant of $e$
  does not vanish at $\init_e(g)$, then the tropicalization map is not
  faithful at $e$ and we can unfold this edge with a linear
  re-embedding $I_{g,f}$ of the curve determined by a tropical
  modification. The new curve $\Trop(I_{g,f})$ contains a cycle that
  maps to $e$ via $\pi_{XY}$.
\end{theorem}
\noindent
It is worth pointing out that the previous statement gives no information regarding the reverse implication. We refer to Remark~\ref{rem:fatEdgeConverse} for more details.

\begin{remark}\label{rem:DisclaimerThmDecontracting}
  At first glance, the statement of Theorem~\ref{thm:redVertex} seems
  a bit restricted and it would be desirable to treat more general
  reducible vertices that those dual to Figure~\ref{fig:trapezoid}
  (possibly, after a unimodular transformation). We choose to avoid
  the general case for three concrete reasons. First, working with the
  general case will force the use of multivariate resultants rather
  than discriminants, even for those vertices traversed by a straight
  line. Secondly, Figure~\ref{fig:trapezoid} is the only one where the
  reducibility of a polynomial supported on it is equivalent to the
  vanishing of the discriminant. 

  Third, for reducible vertices of arbitrary shape whose components
  have valency strictly greater than $3$, linear modifications are no
  longer helpful to repair bad embeddings. In
  Section~\ref{sec:experiments} we present an example of a reducible
  vertex $v$ with two valency-3 components that can be repaired by a
  linear modification (see Example~\ref{ex:otherRedVertex}). The
  presence of these components forces us to modify $\RR^2$ along a
  tropical line with vertex $v$. Consequently, we have to use more
  charts to characterize the linear re-embedding. These projections
  are harder to describe in terms of a coordinate change.

  Likewise, the valency-three condition imposed in the statement of
  Theorem~\ref{thm:fatEdge} allows us to stay in the world of
  discriminants for all our computations.
  Examples~\ref{ex:CyleAppears} and~\ref{ex:UnfoldImpossible} in
  Section~\ref{sec:experiments} give two instances where we drop the
  valency-three condition and our methods have different
  outcomes. 
\end{remark}

The remainder of this section is devoted to the proofs of
Theorems~\ref{thm:redVertex} and~\ref{thm:fatEdge}. A series of lemmas
facilitate the exposition.  We make the following simplifications.
Let $v$ be a locally reducible vertex dual to the trapezoid $\cP$ from
Figure~\ref{fig:trapezoid} traversed by a vertical line $X=l$. Our
tropical linear modification has the form $F=\max\{X,l\}$ and its
lifting equals $f=x+\AA\,t^{-l}\in \PS[x]$ where
$\val(\AA)=0$. Furthermore, after rescaling the variables $x$ and $y$
by appropriate powers of $t$ and rescaling the resulting polynomial
$
g(t^ax,t^by)$ by a suitable power of $t$, we may assume that in the
Newton subdivision of $g$, the trapezoid $\cP$ has height 0 and all
other monomials have negative height. Here, the \emph{height} of a
monomial $\underline{x}^{\alpha}$ appearing in $g$ equals
$-\val(c_{\alpha})$, the negative valuation of the corresponding
coefficient. As a result, we can take $l=0$.

With a similar technique, we can assume the Newton subdivision of the
polynomial $g$ in Theorem~\ref{thm:fatEdge} has height $0$ at the dual
edge $e^{\vee}$, while all other monomials in $g$ have negative
height.  In this case, we unfold the edge $e$ by means of the linear
tropical modification along $\max\{X,0\}$ and a suitable lifting
$f=x+\AA \in \PS[x]$ with $\val(\AA)=0$.
 
In the following lemmas, we keep the previous assumptions. We use  the
notation $\sigma_3$ for the attached cell in the linear tropical
modification $\max\{X,0\}$, discussed at the end of
Section~\ref{sec:trop-modif-line}. 
The heart of these technical results lies in
Lemma~\ref{lem:ModifViaProjections}. We study the expected valuations
of the coefficients $\tilde{c}_{0,k}$ of $\tilde{g}=g(z-\AA,y)$ for appropriate values of $k$
and choose  $\AA$'s that make their valuations  higher
than expected. This  ensures that $\Trop(I_{g,f})\cap \sigmaint_3$
does not  consist only of downward ends attached to $\Trop(g)\cap (X=0)$.

\begin{lemma}\label{lm:expectedHeights} 
  Let $(i,j)$ and $(p,j+1)$ be the left vertices of the trapezoid
  $\cP$ from Figure~\ref{fig:trapezoid}, and let $h_1$ and $h_2$ be
  the univariate polynomials of degree $n$ and $s$ induced by the two
  bases of the trapezoid. Assume that the Newton subdivision of $g$
  achieves its maximum height (zero) at $\cP$.  Then, for any given
  $\AA\in \PS$ with $\val(\AA)=0$, the coefficients of all monomials
  $y^j, zy^j, \ldots, z^{i+n}y^j$ and $ y^{j+1}, zy^{j+1}, \ldots,
  z^{p+s}y^{j+1}$ in $\tilde{g}(z,y)=g(z-\AA,y)$ have expected
  valuations $\val(c_{i+n,j})$ and $\val(c_{p+s,j+1})$, respectively.
  The valuation of the coefficients of the monomials $z^{p+s}y^{j+1}$
  and $z^{i+n}y^j$ equals the expected one.  Moreover, the valuations
  of $\tilde{c}_{0,j}$ and $\tilde{c}_{0,j+1}$ are higher than
  expected if and only if $\init_t(\AA)=\AA_0\in \CC^*$ is a common
  solution of $h_1$ and $h_2$, and the discriminant of $\cP$ vanishes
  at the point $(a_0, \ldots, a_n, b_0,\ldots, b_s)\in
  \CC^{n+s+2}$. For such choice, $\tilde{c}_{1,j+1}$ has the expected
  valuation if and only if $h_2'(-\AA_0)\neq 0$.
\end{lemma}
\begin{proof}
  Following the notation of Figure~\ref{fig:trapezoid}, we let $a_0,
  \ldots, a_n,b_0,\ldots,b_s \in \CC$ be the constant terms of the
  coefficients $c_{\alpha}$ from $g$, with $\alpha\in \cP$.  The
  binomial expansion of each factor $(z-\AA)$ yields:
\begin{equation}
  \tilde{c}_{l,j}=\sum_{k\geq l} c_{k,j} \binom{k}{l}(-\AA)^{k-l} \quad\text{ and }\quad
  \tilde{c}_{l,j+1}=\sum_{k\geq l} c_{k,j+1} \binom{k}{l}(-\AA)^{k-l} \; \text{ for all }l.\label{eq:expValFeeding}
\end{equation}
Since $\val(\AA)= \val(c_{i+n,j})=\val(c_{p+s,j+1})=0$ and all terms
in $g$ have height at most 0, we conclude that all the coefficients
$\tilde{c}_{l,j}$ with $0\leq l\leq n+i$, and $\tilde{c}_{l,j+1}$ with
$0\leq l\leq p+s$ have expected valuation 0. Moreover,
$\val(\tilde{c}_{i+1,j+1})=\val(\tilde{c}_{i+n,j})=0$ by construction.

In particular, we  can compute the constant terms of $\tilde{c}_{0,j}$, $\tilde{c}_{0,j+1}$ and $\tilde{c}_{1,j+1}$:
\begin{equation}
    \init_t(\tilde{c}_{0,j})\!=\! (-\AA_0)^ih_1(-\AA_0);\; 
    \init_t(\tilde{c}_{0,j+1})\!=\!(-\AA_0)^p h_2(-\AA_0);\;
    \init_t(\tilde{c}_{1,j+1})\!=\!h_2(-\AA_0) -\AA_0 h_2'(-\AA_0).
\label{eq:initials}
\end{equation}
Thus, $\tilde{c}_{0,j}$ and $\tilde{c}_{0,j+1}$ have strictly positive
valuation if and only if $\AA_0$ is a nonzero common solution of $h_1$
and $h_2$. By Lemma~\ref{lm:trapezoidDiscriminant}, such $\AA_0$
exists if and only if $\Delta_{\cP}(a_0,\ldots, b_s)=0$. 
Furthermore by~\eqref{eq:initials} we deduce that $\init_t(\tilde{c}_{1,j+1})=-\AA_0h_2'(-\AA_0)$,  so $\val(\tilde{c}_{1,j+1})=0$ if and only if $h_2'(-\AA_0)\neq 0$.
\end{proof}
\begin{remark}\label{rem:UnimodularTransformationsForEllipticCase}
  The statement of Lemma~\ref{lm:expectedHeights} also holds for any
  locally reducible vertex $v$ with a line through it of slope $0$ or
  $1$, and whose dual cell is a trapezoid $\cP$ of height one, as we
  now explain.  As usual, assume that the trapezoid $\cP$ has maximal
  height 0.  By symmetry between $x$ and $y$ we need only consider the
  case of slope one, namely, when the parallel lines containing the
  bases of $\cP$ are $L:= \{y=-x+r\}$ and $L':=\{ y=-x+r-1\}$.  Let
  $b_0,\dots, b_s$ be the initial terms of the coefficients of the
  monomials in $L$, and $a_0,\ldots, a_n$ be the ones contained in
  $L'$. We let $h_1$ and $h_2$ be the polynomials supported on the two
  bases of $\cP$.  The $A$-discriminant of the trapezoid is the same
  as the one for Figure~\ref{fig:trapezoid} since these two polygons
  are related by a unimodular transformation of $\ZZ^2$.

  We write $v=(v_1,v_2)$. The linear re-embedding induced by the
  tropical modification along $L^{\perp}$ is determined by the
  function $f=x+\AA t^{v_2-v_1}y$ with $\val(\AA)=0$. The plane curve
  $\tilde{g}=g(z-\AA t^{v_2-v_1}y, y) $ is obtain by projecting
  $I_{g,f}$ to the $ZY$-plane. In the feeding process, a monomial
  $x^iy^j$ in $g$ contributes to all monomials $z^{i-k}y^{j+k}$ in
  $\tilde{g}$ for $0\leq k\leq i$. In particular, if a point $(i,j)\in
  L\cup L'$ lies to the left of the vertices with coefficients
  $a_n,b_s$, then the coefficient $\tilde{c}_{i,j}$ has expected
  valuation 0. We are interested in increasing the valuation of the
  coefficients of $\tilde{g}$ associated to the intersection points
  $(0,r)$ and $(0,r-1)$ of the $y$-axis with the lines $L$ and $L'$,
  respectively.

  By Lemma~\ref{lm:trapezoidDiscriminant}, the vanishing of
  $\Delta_{\cP}$ ensures that $h_1$ and $h_2$ have a common solution
  $\AA_0$ in $\CC^*$. The feeding process ensures that
  $\tilde{c}_{0,r}$ and $\tilde{c}_{0,r-1}$ have negative valuation if
  and only if $\init_t(\AA)=\AA_0$.\end{remark}

\begin{lemma}\label{lm:VanishingDecontractsVertex}
  Let $v$ be a locally reducible vertex of $\Trop(g)$ dual to the
  trapezoid $\cP$ in Figure~\ref{fig:trapezoid}, contained in the line
  $\{X=l\}$. Suppose that the discriminant of $\cP$ vanishes at
  $\init_v(g)$ and let $\AA_0$ be a common solution of $h_1(x)$ and
  $h_2(x)$ in $\CC^*$.  Let $f:=x+\AA t^{-l}$ be a lifting function
  with $\val(\AA)=0$ and $\init_t(\AA)=\AA_0$.  Then, the linear
  re-embedding $I_{g,f}$ produces a decontraction/unfolding of some
  edges of $\widehat{\Sigma}(g)\smallsetminus D_g$ that map to
  $\Trop(g)\cap \{X=l\}$. Furthermore, $v\times \{l\}$ is a vertex of
  $\Trop(I_{g,f})$ and its multiplicity is strictly smaller than
  $\operatorname{mult}_{\Trop(g)}(v)$.
\end{lemma}
\begin{proof}
  The claim follows from Lemma~\ref{lm:expectedHeights}.  As usual,
  assume that the trapezoid $\cP$ has height zero, and all
  coefficients of $g$ have non-negative valuation, so $l=0$ and
  $v=(0,0)$.  Set $\tilde{g}(z,y):=g(z-\AA,y)$.  The given hypotheses
  ensure that the coefficients of $y^j$ and $y^{j+1}$ in $\tilde{g}$
  have strictly positive valuation and all other coefficients have
  non-negative valuation.

  By our genericity condition~\ref{genericityassumption}, the
  coefficients $\tilde{c}_{0,j}$ and $\tilde{c}_{0,j+1}$ of
  $\tilde{g}$ are non-zero and have positive valuation. Let $1\leq
  k\leq i+n$ and $1\leq k'\leq p+s$ be minimal with the property that
  $\tilde{c}_{{k,j}}$ and $\tilde{c}_{k',j+1}$ have valuation zero. We
  let $e_0$ be the edge in the Newton subdivision of $\tilde{g}$
  determined by $(k,j)$ and $(k',j+1)$. Notice that $e_0^{\vee}$ has
  multiplicity 1, and so it is faithfully represented in
  $\widehat{\Sigma}(g)\smallsetminus D_g$ by a unique edge $E$ that
  maps to $e_0$ under $\trop$.

  Different values of $k$ and
  $k'$ yield different polyhedral structures on $\Trop(\tilde{g})\cap
  \{Z<0\}$ and thus in $\Trop(I_{g,x+\AA})\cap \sigmaint_3$ by
  Lemma~\ref{lem:ModifViaProjections}.  In turn, they produce the
  decontraction or unfolding of edges, as we now explain.

  By symmetry, we assume $k'\leq k$. We claim that the point $(0,0)$
  lies in $\Trop(\tilde{g})$: it is a vertex when $k<n+i$ or $k'<p+s$,
  and otherwise lies in the relative interior of the edge $e_0^{\vee}$
  in $\Trop(\tilde{g})$ (see Figures~\ref{fig:64aTropical}
  and~\ref{fig:TwoStepPart1}). In both cases, a segment of
  $e_0^{\vee}$ connects $(0,0)$ to a vertex $v_1$ in
  $\Trop(\tilde{g})\cap \{Z<0\}$ as in Figure~\ref{fig:folding}.  The
  vertex $v_1$ can be described as follows.  Our genericity
  condition~\ref{genericityassumption} ensures that there is a polygon
  $Q_1$ in the Newton subdivision of $\tilde{g}$ containing $e_0$ and
  a vertex $w_1=(q_1,r_1)$ where $r_1<(k'-k)(q_1-k)+j$, i.e.\ $w_1$
  lies below the line with direction $e_0$ passing through $(k,j)$.
  The polygon $Q_1$ is dual to a vertex $v_1$ of $\Trop(\tilde{g})\cap
  \{Z<0\}$. 

  The combinatorics of the stars of $(0,0)$ in $\Trop(g)$ and
  $\Trop(\tilde{g})$ and Lemma~\ref{lem:ModifViaProjections} ensure
  that $(0,0,0)$ is a vertex of $\Trop(I_{g,x+\AA})$. Let $S$ be the
  multiplicative closed set generated by $(x+\AA_0)$. Our choice of
  $\AA_0$ guarantees that the localized initial ideal $\langle
  \init_{v}(g)\rangle[S^{-1}]\subset \CC[x^{\pm}, y^{\pm}] [S^{-1}]$
  has length strictly smaller than the length of the ideal
  $\init_{v}(g)\subset \CC[x^{\pm},
  y^{\pm}]$. Proposition~\ref{pr:projectionsJ} implies that the
  multiplicity of $(0,0,0)$ in $\Trop(I_{g,x+\AA})$ is bounded above
  by $\operatorname{mult}_{\Trop(g)}(v)-1$.  \smallskip

  When $k=k'$, the edge $e_0^{\vee}$ has direction $(0,0,-1)$ and it
  connects the vertex $v_1$ of $\Trop(I_{g,x+\AA})\cap \sigmaint_3$ to
  the point $(0,0,0)$.  The map $\pi_{XY}\colon\Trop(I_{g,x+\AA})
  \rightarrow \Trop(g)$ contracts this bounded edge and all downward
  ends in $\sigmaint_3$. Bounded edges with other directions project
  homeomorphically to segments inside $\Trop(g)\cap\{X=0\}$. By
  Lemma~\ref{lem:ModifViaProjections} and~\eqref{eq:contr/fold} we
  conclude that the linear re-embedding $I_{g,x+\AA}$ decontracts the
  edge $e_0^{\vee}$ in $\widehat{\Sigma}(I_{g,x+\AA})\smallsetminus
  D_{I_{g,x+\AA}}$, so it decontracts the edge $E$ in
  $\widehat{\Sigma}(g)\smallsetminus
  D_g$. Example~\ref{ex:twoStepsExample} and
  Figure~\ref{fig:TwoStepPart1} exhibit this behavior.

  On the contrary, assume $k'<k$. This situation leads to an unfolding
  of edges, as sketched in Figure~\ref{fig:folding}. Assume the vertex
  $w_1$ constructed above has $r_1\geq j+1$. Then, the 2-cell $Q_1$ in
  the Newton subdivision of $\tilde{g}$ contains a vertex
  $w_{2}=(q_{2},r_{2})$ with $r_{2}\geq j+1$ and adjacent to
  $(k',j+1)$ by an edge $e_{2}$.  Notice that we allow the possibility
  that $w_{2}$ and $w_1$ agree. The edges $e_{2}^{\vee}$ and
  $e_0^{\vee}$ share the endpoint $v_1$. When viewed in
  $\Trop(I_{g,x+\AA})\cap \sigma_3^{\circ}$ using
  Lemma~\ref{lem:ModifViaProjections}, the same holds true for a
  segment of these two edges. The projection $\pi_{XY}$ overlaps these
  two edges along an edge of $\Trop(g)$ contained in the line
  $\{X=0\}$.  Diagram~\eqref{eq:contr/fold} implies that the linear
  re-embedding $I_{g,x+\AA}$ unfolds some edges of $\Trop(g)\cap
  \{X=0\}$.

  Finally, suppose $r_1\leq j$.  By convexity, we construct a
  (possibly empty) maximal collection of vertices $w_2,\ldots,w_m$
  with $w_l=(q_l,r_l)$ and $r_l\leq j$ for all $l$ such that
  \begin{enumerate}[(1)]
\item $w_l$ and $(k', j+1)$ are connected by an edge $e_l$,
\item for all $l\geq 2$, $w_{l-1},w_{l}, (k',j+1)$ are vertices of a polygon $Q_l$ in the Newton subdivision of $\tilde{g}$.
\end{enumerate}
By the genericity assumption~\ref{genericityassumption} we can find a
vertex $w_{m+1}=(q_{m+1},r_{m+1})$ with $r_{m+1}\geq j+1$ in the
Newton subdivision of $\tilde{q}$, connected to $(k',j+1)$ by an edge
$e_{m+1}$ and such that $e_{m+1}$ and $e_m$ are edges of a polygon
$Q_{m+1}$ in the subdivision. We let $v_{m+1}$ be the corresponding
vertex in $\Trop(\tilde{g})\cap \{Z<0\}$.

The edges $e_l^{\vee}$ of $\Trop(\tilde{g})$ have directions
$(j+1-q_l, r_l-k')$ for $l=1,\ldots, m+1$. They form a chain that
links the vertices $v_{m+1}$ and $v_1$, as in
Figure~\ref{fig:folding}. Using Lemma~\ref{lem:ModifViaProjections}
and the convexity of the Newton subdivision of $\tilde{g}$, we
conclude that the projection $\pi_{XY}$ maps these edges (and their
linking chain) to overlapping edges in $\{X=Z=0\}$.
Using diagram~\eqref{eq:contr/fold}, we conclude that the linear re-embedding
$I_{g,x+\AA}$ unfolds some edges of $\Trop(g)\cap
\{X=0\}$. Example~\ref{ex:64a} and Figure~\ref{fig:64aTropical}
capture this phenomenon.
\end{proof}
\begin{figure}
\includegraphics[scale=0.18]{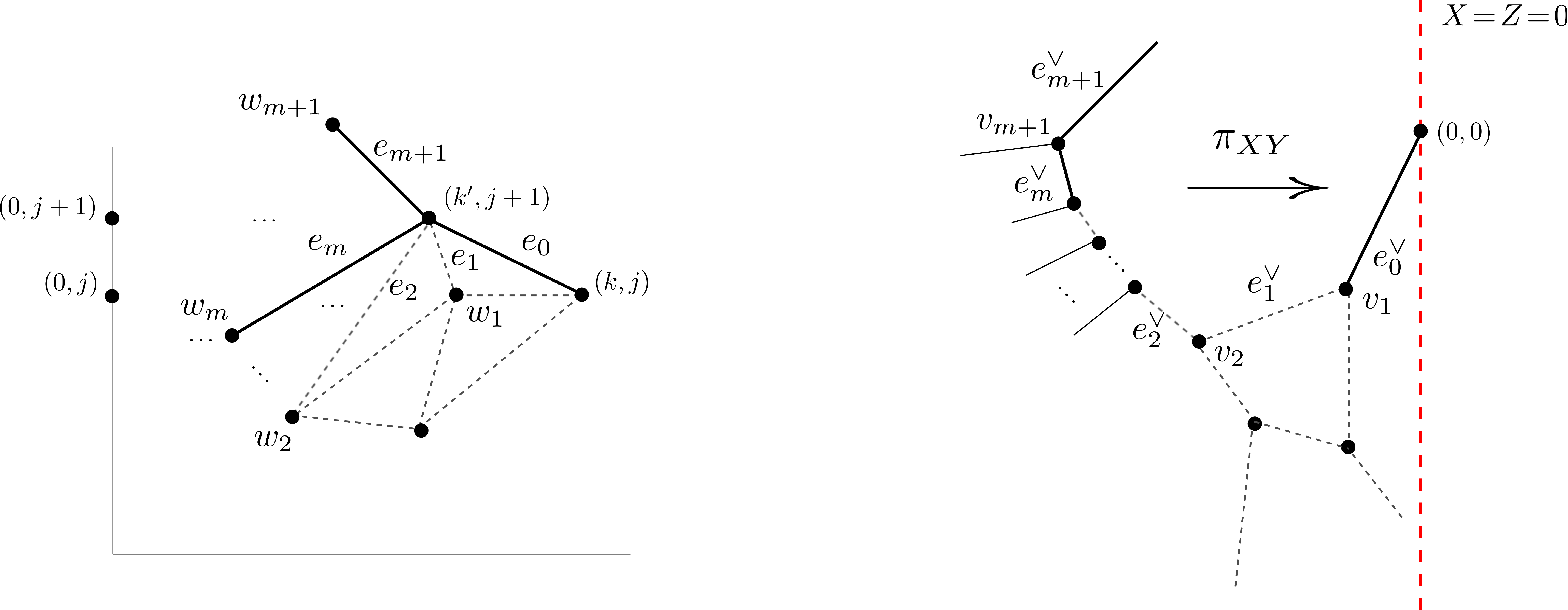}
\caption{The Newton subdivision of $\tilde{g}$ and the projection
  $\pi_{XY}$ folding edges of  $\Trop(I_{g,f})\cap \sigmaint_3$. The dashed
  segments describe a potential scheme of edges in the
  subdivision. The red dotted lines indicate the image of the
  projection.}\label{fig:folding}
\end{figure}

\begin{lemma}\label{lem:noDecontractionIfLocallySmooth}
  Let $v$ be a locally reducible vertex $v$
  in $\Trop(g)$ and assume that the discriminant $\Delta_v$ 
  does not vanish at $\init_v(g)$. Then, there exists a unique
  vertex $\rho$ in $\Sigma:=\widehat{\Sigma}(g)\smallsetminus D_g$ and a
  subgraph $\Gamma$ of $\Star_{\Sigma}(\rho)$ that maps
  homeomorphically to $\Star_{\Trop(g)}(v)$. For any other vertex
  $\rho'\in \trop^{-1}(v)$ in $\Sigma$, we have
  $\trop(\Star_{\Sigma}(\rho'))= v$.
\end{lemma}
\begin{proof}
  If the discriminant $\Delta_v$ does not vanish at $\init_v(g)$, then
  by Corollary~\ref{cor:highMultAndDiscrim}, we know that $\init_v(g)$
  is irreducible and generically reduced, thus
  $m_{\Trop}(v)=1$. By~\cite[Corollary 6.12]{BPR11} and formula \eqref{eq:tropMultVsRelMult}, all but one
  vertex in $\trop^{-1}(v)$ map their stars to $v$. We let  $\rho$ be this special
  vertex. The proof of Lemma~\ref{lem:nonHomeo} shows that we
  can construct $\Gamma$ in $\Star_{\Sigma}(\rho)$ satisfying
  the desired property.
\end{proof}
\begin{remark}\label{rem:LocalFaithfulnessAroundV}
  Notice that by construction, the multiplicity of two of the edges
  adjacent to a locally reducible vertex $v$ dual to
  Figure~\ref{fig:trapezoid} is one, and so the tropicalization is
  faithful on these edges by~\cite[Theorem 6.23]{BPR11}.  The only
  exceptions are the edges dual to the two bases of the trapezoid in
  Figure~\ref{fig:trapezoid}, whose multiplicities are $s$ and $n$.
  We cannot guarantee faithfulness of the tropicalization map over
  these edges unless $n=s=1$. Furthermore,
  Corollary~\ref{cor:highMultAndDiscrim} ensures that
  $m_{\Trop}(v)\geq 2$ if and only if $\Delta_{v}$ vanishes at
  $\init_v(g)$.

  When $\Delta_v$ does not vanish at $\init_v(g)$, the multiplicity
  one edges in $\Star_{\Trop(g)}(v)$ are faithfully represented as a
  subgraph $\Gamma$ of $\Star_{\Sigma}(\rho)$, where $\rho$ is the
  vertex of Lemma~\ref{lem:noDecontractionIfLocallySmooth}. This graph
  is a star tree.  If $n,s>1$, then $\trop$ maps each of the edges in
  $\Star_{\Sigma'}(\rho)$ outside $\Gamma$ either to $v$ or
  homeomorphically to the corresponding high multiplicity edge in
  $\Trop(g)$. The proof of Lemma~\ref{lem:nonHomeo} show that the
  previous statements remain true if we replace $\Sigma$ by the
  connected 1-dimensional complex $\Sigma'=(\Sigma\smallsetminus
  \trop^{-1}(v)) \cup \{\rho\}$.
\end{remark}

\begin{proof}[Proof of Theorem~\ref{thm:redVertex}] 
  To simplify notation, write
  $\Sigma(I):=\widehat{\Sigma}(I)\smallsetminus D_I$ for the embedding
  induced by an ideal $I$.  Assume the discriminant $\Delta_v$ does
  not vanish at $\init_v(g)$. The result follows by
  Lemma~\ref{lem:noDecontractionIfLocallySmooth}.  Conversely, suppose
  that the discriminant $\Delta_v$ vanishes at $\init_v(g)$. Then, by
  Lemma~\ref{lm:VanishingDecontractsVertex} we can decontract/unfold
  edges of $\Sigma(g)$ that map to $\Trop(g)\cap \{X=l\}$ using a
  linear re-embedding induced by a tropical modification along $L$. We
  analyze the different combinatorial structures that appear in the
  proof of the lemma to construct  $\rho$ and $\rho'$. We keep the notation used in the latter.

  Suppose the linear re-embedding produces a decontraction of an edge
  $e$ adjacent to $v$ in $\Trop(I_{g,f})$, i.e.\ $k=k'$. Since the trapezoid
  $v^{\vee}$ has height 1, we conclude that $e$ has a unique preimage $e'$
  in $\Sigma(I_{g,f})$. Call $\rho$ and $\rho'$ its ends and let $\Gamma= \{e'\}$. By
  construction, $\trop(\rho')=\trop(\rho)=v\in \Trop(g)$ and their
  stars in $\Sigma(I_{g,f})$ contain $e'$. Their stars are not
  contracted in $\Trop(I_{g,f})$ and each one contains at least three
  edges, since $\trop(\rho')$ and $\trop(\rho)$ are two distinct
  non-bivalent vertices. By balancing, one edge from each star in
  $\Trop(I_{g,f})$ maps to a segment in the line $L$ under the projection
  $\pi_{XY}$. By diagram~\eqref{eq:contr/fold}, we conclude that the
  stars of $\rho$ and $\rho'$ are not contracted by $\trop$, while the
  path between them is.

  On the contrary assume $k< k'$. Then,
  Lemma~\ref{lm:VanishingDecontractsVertex} produces a chain of edges
  $\mathscr{C}$ in $\Trop(I_{g,f})\cap \sigma_3^{\circ}$ containing
  $v$ as one of its vertices. We see $\mathscr{C}$ in the right of
  Figure~\ref{fig:folding}. Notice that $v$ lies on an edge $e$ of
  $\Trop(I_{g,f})\cap \sigma_3^{\circ}$, and $e$ has multiplicity 1.
  By construction $\Sigma(I_{g,f})$ contains a unique edge $e'$
  mapping to $e$ under $\trop$. In particular, it contains a unique
  point $\rho'$ mapping to $v$ under $\trop$.

  The chain $\mathscr{C}$ maps to $L$ under $\pi_{XY}$. By convexity,
  this chain contains a point $v'$ of $\Trop(I_{g,f})\cap
  \sigma_3^{\circ}$ that maps to $v$ under $\pi_{XY}$. Furthermore, we
  can choose $v'$ to be a vertex of $\Trop(I_{g,f})$ or a point in the
  interior of an edge whose direction is not $(0,0,-1)$. In both
  cases, $\pi_{XY}$ does not contract the star of $v'$ in
  $\Trop(I_{g,f})$. We define $\Gamma := \trop^{-1}(\sigma_3^{-1})$,
  so $\trop^{-1}(\mathscr{C})\subset \Gamma$.  This graph contains the
  preimage of $v'$. Since $\mult_{\Trop}(v')\geq 1$,
  \eqref{eq:tropMultVsRelMult} implies the existence of a point
  $\rho'$ in $\Gamma$ mapping to $v'$ but whose star is not contracted
  by $\trop$.

  By further refining the structure of $\Sigma(I_{g,f})$, we may
  assume $\rho$ and $\rho'$ are vertices of $\Sigma(I_{g,f})$, and
  hence of $\Gamma$. Diagram~\eqref{eq:contr/fold} guarantees that the
  stars of $\rho$ and $\rho'$ in $\Gamma$ are not contracted in
  $\Trop(g)$. Moreover, $\trop$ folds and/or contracts some edges of
  $\Gamma$: its image lies on $L$.
\end{proof}

\begin{proof}[Proof of Theorem~\ref{thm:fatEdge}] 
  For simplicity, we suppose that the dual cell to the vertical edge $e$ in
  the Newton subdivision of $g$ is the horizontal edge with endpoints
  $(i,j)$ and $(i+n,j)$. Without loss of generality, we further assume all
  coefficients of $g$ have non-negative valuation and
  $\val(c_{i,j})=\val(c_{i+n,j})=0$, so $e$ lies in the line
  $X=0$. We write $v=(0,B)$ and $v'=(0,B')$ with $B'< B$ for the trivalent endpoints of $e$.  Figure~\ref{fig:twoTriangles} depicts the dual cells to $e, v$ and $v'$.

  \begin{figure}[tb]
    \centering
\includegraphics[scale=0.17]{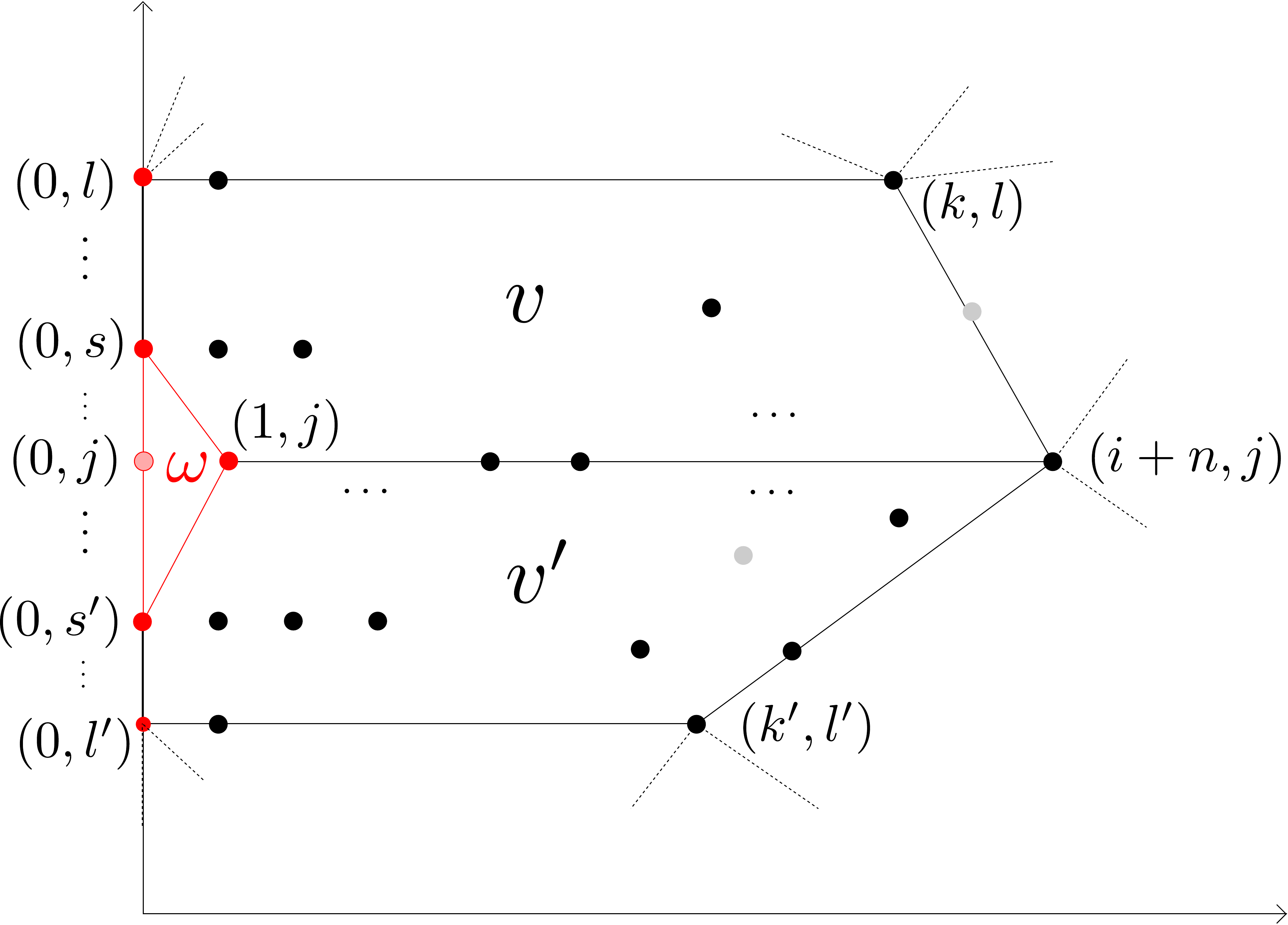}
\qquad
    \includegraphics[scale=0.17]{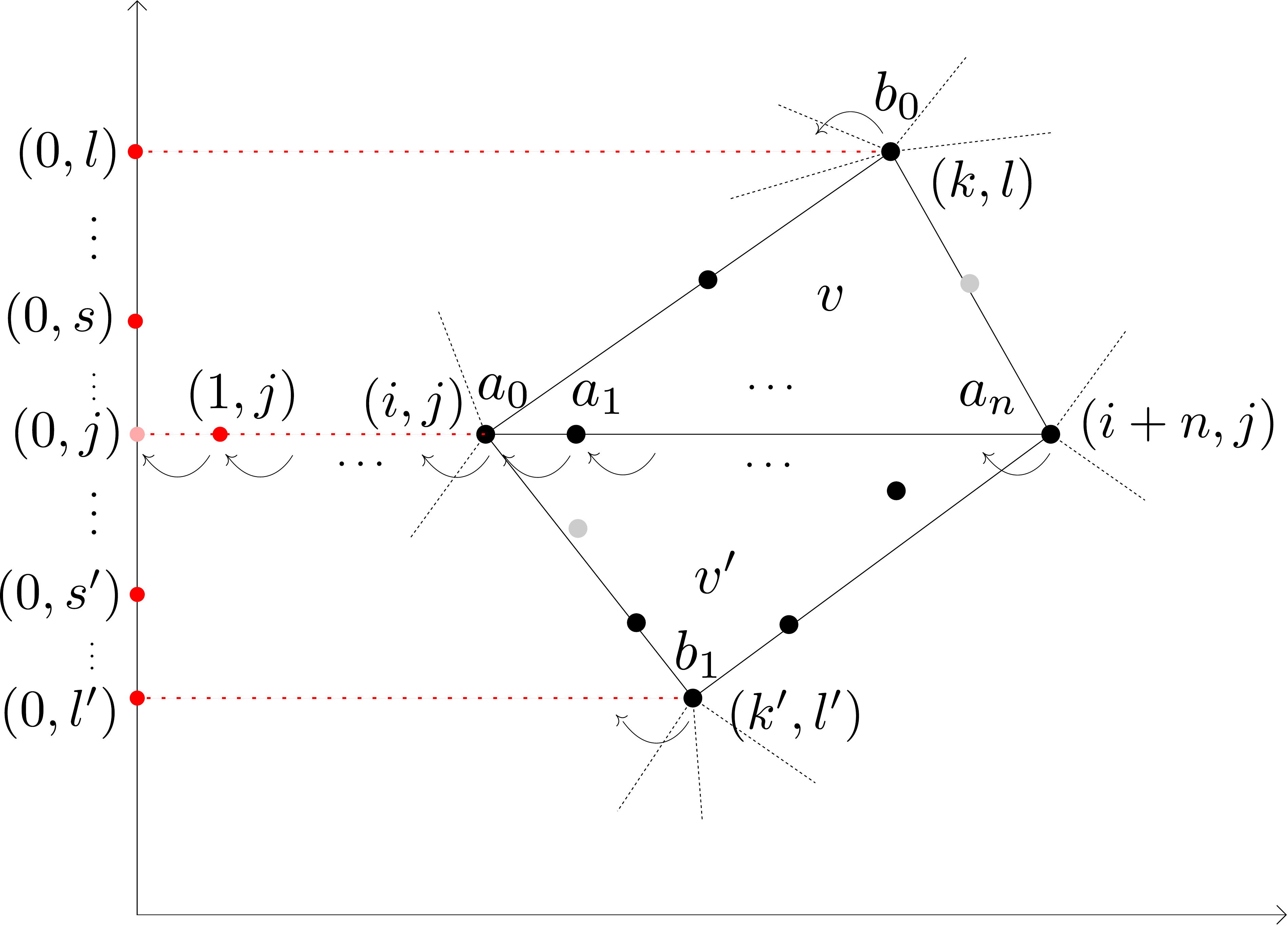}
    \caption{From right to left: The dual cells to a high-multiplicity
      vertical edge $e$ in $\Trop(g)$ and its 3-valent endpoints $v$ and
      $v'$ in the Newton subdivision of 
      $\tilde{g}=g(z-\AA,y)$ for a suitable $\AA$ with $\val(\AA)=0$ The arrows
      on the right indicate the direction of the feeding process from
      $g$ to $\tilde{g}$. The left picture shows the 2-cells dual to
      three vertices $\ww$, $v$ and $v'$ forming part of a cycle in
      $\Trop(\tilde{g})$.}
    \label{fig:twoTriangles}
  \end{figure}

  We let $a_0,\ldots, a_n$ be the initial terms of the coefficients
  of $x^iy^j,\ldots,x^{i+n}y^j$ in $g$, and $b_0, b_1\in \CC$ be the
  initial terms of the coefficients of $x^ky^{l}$ and $x^{k'}y^{l'}$
  in $g$, respectively. In particular, we know that
  $a_0,a_n,b_0,b_1\neq 0$. We define
\[h(x):=\sum_{k=0}^{n}a_k\,x^{k}=\init_e(g)/(x^iy^j) \in \CC[x].
\]
\indent  As in the proof of Lemma~\ref{lm:expectedHeights}, we work with a
  tropical modification of $\RR^2$ along $F=\max\{X,0\}$, and we let
  $\tilde{g}=g(z-\AA,y)$, where
  $\val(\AA)=0$. Write $\AA_0=\init_t(\AA)$. Expression~(\ref{eq:expValFeeding}) shows that all
  coefficients $\tilde{c}_{k,j}$ with $0\leq k\leq i+n$ have expected
  valuation 0. A similar direct calculation proves that the constant
  term of $\tilde{c}_{0,j}$ equals $(-\AA_0)^ih(-\AA_0)$, and the constant
  term of $\tilde{c}_{1,j}$ equals
\[
\init_t(\tilde{c}_{1,j})=\begin{cases}
  \sum_{k=1}^n a_{k}k(-\AA_0)^{k-1}=h'(-\AA_0)& \text{ if } i=0,\\
\sum_{k=0}^n a_k(k+i)(-\AA_0)^{k+i-1}= (-\AA_0\,h'(-\AA_0)+ i h(-\AA_0))(-\AA_0)^{i-1} &  \text{ otherwise}.
\end{cases}
\]
\indent Since $\Delta_e$ is the discriminant of the univariate polynomial $h$
of degree $n>1$, we know that $h(\AA_0)$ and $h'(\AA_0)$ have no
common solution in $\CC^*$. Thus, we can pick $\AA_0$ to be any of the
$n$ simple roots of $h(x)$ in $\CC^*$.  This choice ensures that the
valuation of $\tilde{c}_{0,j}$ is strictly positive while
$\val(\tilde{c}_{1,j})=0$. Therefore, the edge with endpoints $(1,j)$
and $(i+n,j)$ lies in the Newton subdivision of $\tilde{g}$ and at
height 0.  By construction, we have
$\val(\tilde{c}_{k,l})=\val({c}_{k,l})$ and
$\val(\tilde{c}_{k',l'})=\val({c}_{k',l'})$. Combining these facts
with the previous arguments ensures that $v$ and $v'$ are vertices of
$\Trop(\tilde{g})$.

The feeding process described above can also be applied to determine
the valuations of the coefficients $\tilde{c}_{0,l}$ and
$\tilde{c}_{0,l'}$. We argue for the point $(0,l)$.  By construction,
the highest point of the Newton subdivision of $g$ along the
horizontal line $Y=l$ is $(k,l)$.  As we illustrate with the red
dashed lines in Figure~\ref{fig:twoTriangles}, our choice of $\AA$
above and a direct calculation ensure that the height of the point
$(0,l)$ in the Newton subdivision of $\tilde{g}$ equals
$-\val(c_{k,l})$, as expected.  Similarly, the points $(0,j+1),\ldots,
(0,l-1)$ have expected height induced by the height function of the
dual cell $v^{\vee}$. Their actual heights can be lower. This ensures
that the points $(1,j), (i+n,j), (k,l)$, and $(0,l)$ are part of the
vertex set of a polygon in the Newton subdivision of $\tilde{g}$ dual
to $v$. This polygon contains at most one extra vertex $(0,s)$, with
$j+1\leq s \leq l$.

By symmetry between the endpoints $v$ and $v'$, we can also find
$l'\leq s'\leq j-1$ with the property that the polygon with vertices
$(0,s'),(0,l'), (k',l'), (i+n,j), (1,j)$ lies in the Newton
subdivision of $\tilde{g}$ and is dual to $v'\in \Trop(\tilde{g})$.
In addition, our choice of $s,s'$ ensures that the triangle with
vertices $(0,s),(0,s'),(1,j)$ is a polygon in the Newton subdivision
of $\tilde{g}$. It is dual to a vertex $\ww$ in $\Trop(\tilde{g})$. y
The points $\ww, v$ and $v'$ induce a cycle in $\Trop(\tilde{g})$, as
we see in the left of Figure~\ref{fig:twoTriangles}.

We now use Lemma~\ref{lem:ModifViaProjections} to deduce that the
linear re-embedding $\Trop(I_{g,x+\AA})$ unfolds edges mapping to $e$
by the projection $\pi_{XY}$, as in
Figure~\ref{fig:unfoldingDoubleEdge}.  By construction, we know that
$v=(0,B)$ and $v'=(0,B')$ are vertices of $\Trop(I_{g,x+\AA})$ and lie
in the line $\{X=Z=0\}$.  Furthermore, we can find $\varepsilon >0$ so
that both open segments $\{0\}\times (B,B+ \varepsilon)\times \{0\}$
and $\{0\}\times (B'-\varepsilon, B')\times \{0\}$ inside $\{X=Z=0\}$
do not meet $\Trop(I_{g,x+\AA})$. If this were not the case, the
intersection points would be part of the projection
$\pi_{XY}(\Trop(I_{g,x+\AA}))=\Trop(g)$.

We claim that the open segment $\{0\}\times (B',B)\times \{0\}$ lies
in $\Trop(I_{g,x+\AA})\cap \{X=Z=0\}$ and has multiplicity
$n-1>0$. From the projection
$\pi_{ZY}(\Trop(I_{g,x+\AA}))=\Trop(\tilde{g})$, we know that this
segment contains no vertex of $\Trop(I_{g,x+\AA})$. We certify our
claim by computing the star of $v$ in $\Trop(I_{g,x+\AA})$ on each
open cell $\sigmaint_i$, $i=1,2,3$ from Figure~\ref{fig:twoTriangles},
and using the balancing condition. On the cells $\sigma_1^{\circ} \cup
\sigma_2^{\circ}$, the star contains only two edges (with
multiplicity), with directions $(j-l,k-i,0)$ (in $\sigmaint_1$) and
$(l-j,i+n-k,l-j)$ (in $\sigmaint_2$). Similarly, the cell
$\sigma_3^{\circ}$ contains two edges of the star: their directions
are $(0,0,s-l)$ and $(0,-1,j-s)$ and. The last one is nothing but the
lifting of the edge $wv$ from $\Trop(\tilde{g})$. The union of these
four edges with multiplicities is not balanced at $v$. An edge with
direction $(0,-1,0)$ and multiplicity $n-1$ solves this issue.

By diagram~\eqref{eq:contr/fold} and
Lemma~\ref{lem:ModifViaProjections} we see that $ \trop\colon
\widehat{\Sigma}(I_{g,,x+\AA})\smallsetminus
D_{I_{g,,x+\AA}}\rightarrow \Trop(g)$ maps the bounded edges $wv$ and
$wv'$ in $\Trop(I_{g,x+\AA})\cap \sigma_3$ to the edge $e$ in
$\Trop(g)$ with relative multiplicity 1. The map keeps their images
disjoint away from the vertex $\ww$. Hence, the linear re-embedding
$I_{g,x+\AA}$ unfolds the corresponding edges, as desired.
\end{proof}

\begin{remark}\label{rem:fatEdgeConverse} Notice that the arguments in the proof
  of Theorem~\ref{thm:fatEdge} cannot be reversed. Indeed, pick any
  point $p$ in the relative interior of $e$. Then
  $\init_p(g)=\init_e(g)$ is supported on an edge of length $n>1$, so
  it is  a zero-dimensional scheme of length $n$. If
  $\Delta_e$ vanishes at $\init_w(g)$, then this scheme is non-reduced, and 
\eqref{eq:tropMultVsRelMult} provides no information to determine
  the  value of the
  relative multiplicities.
\end{remark}

\begin{remark}\label{rem:unfoldingAndTrivalentOutcome}
  From the proof of Theorem~\ref{thm:fatEdge} we see that the
  unfolding procedure improves the situation: the multiplicity of the
  vertical edge $e$ has decreased by 1 in $\Trop(I_{g,x+\AA})$.  In
  particular, if $n=2$, our method produces a cycle in
  $\Trop(I_{g,x+\AA})$ with vertices $v,v'$ and $w$ and multiplicity 1
  on all its edges.  Example~\ref{ex:unfoldingDoubleEdge} illustrates
  this phenomenon.

  When $n>2$ we would like to iterate this process and unfolds further
  the edge $e$ in $\Trop(I_{g,x+\AA})$ when
  $\Delta_e(\init_{e}(I_{g,x+\AA}))\neq 0$. For this, we require $v$
  and $v'$ to remain trivalent vertices in $\Trop(I_{g,x+\AA})$. This
  will indeed be the case when $s=l$, $s'=l'$ and $k=k'=0$ (see
  Figure~\ref{fig:twoTriangles}). This trivalent condition on $v$ and
  $v'$ need to be essential in concrete examples: the method will
  carry through whenever the special linear re-embedding induced by
  the iterated tropical modification returns a Newton subdivision as
  in the left of Figure~\ref{fig:twoTriangles}.
\end{remark}

\begin{example}\label{ex:unfoldingDoubleEdge} 
  We consider a plane elliptic cubic curve $C$ whose tropicalization
  contains a vertical double edge $e$ with trivalent endpoints in
  place of a cycle. It is given by the equation
\[
g(x,y)=t^3\,x^3+x^2y+t^3\,xy^2+t\,y^3+t^4\,x^2+(1+t^2)\,xy+t^2\,y^2+t^5\,x+(1+t)\, y+t.
\]
The tropical curve is depicted on the right of
Figure~\ref{fig:unfoldingDoubleEdge}. The dual edge to $e$ in the
Newton subdivision of $g$ has lattice length 2 and contains the
lattice points $(1,0)$, $(1,1)$ and $(1,2)$.  Its discriminant equals
$\Delta_e= c_{1,1}^2-4c_{1,2}c_{1,0}$, so $\Delta_e(\init_e(g))\neq
0$. We use Theorem~\ref{thm:fatEdge} and the vertical line $X=0$ to
unfold the double edge $e$ and produce a cycle in the re-embedded
tropical curve.  Notice that
$\val(c_{0,1})=\val(c_{1,1})=\val(c_{2,1})=0$. We pick a special
lifting $f=x-\AA$ with $\val(\AA)=0$ and $\AA_0:=\init_{t}(\AA)$,
satisfying
$\init_t(c_{0,1})-\AA_0\init_t(c_{1,1})+\AA_0^2\init_t(c_{2,1})=1-\AA_0+\AA_0^2=0$.
The function $f(x)=x-\frac{1+\sqrt{-3}}{2}$ produces the desired
unfolding as we see on the left side of Figure~\ref{fig:unfoldingDoubleEdge}. Remark~\ref{rem:UnfoldDoubleEdgeElliptic} will show that this
re-embedding induces an isometry between the cycle in 
 $\Trop(I_{g,f})$ and the circle corresponding to the minimal skeleton of the complete curve $\widehat{C}^{\an}$. 
 \begin{figure}[tb]
 \hspace{-2ex}   \begin{minipage}[l]{0.37
\linewidth}
     \includegraphics[scale=0.2]{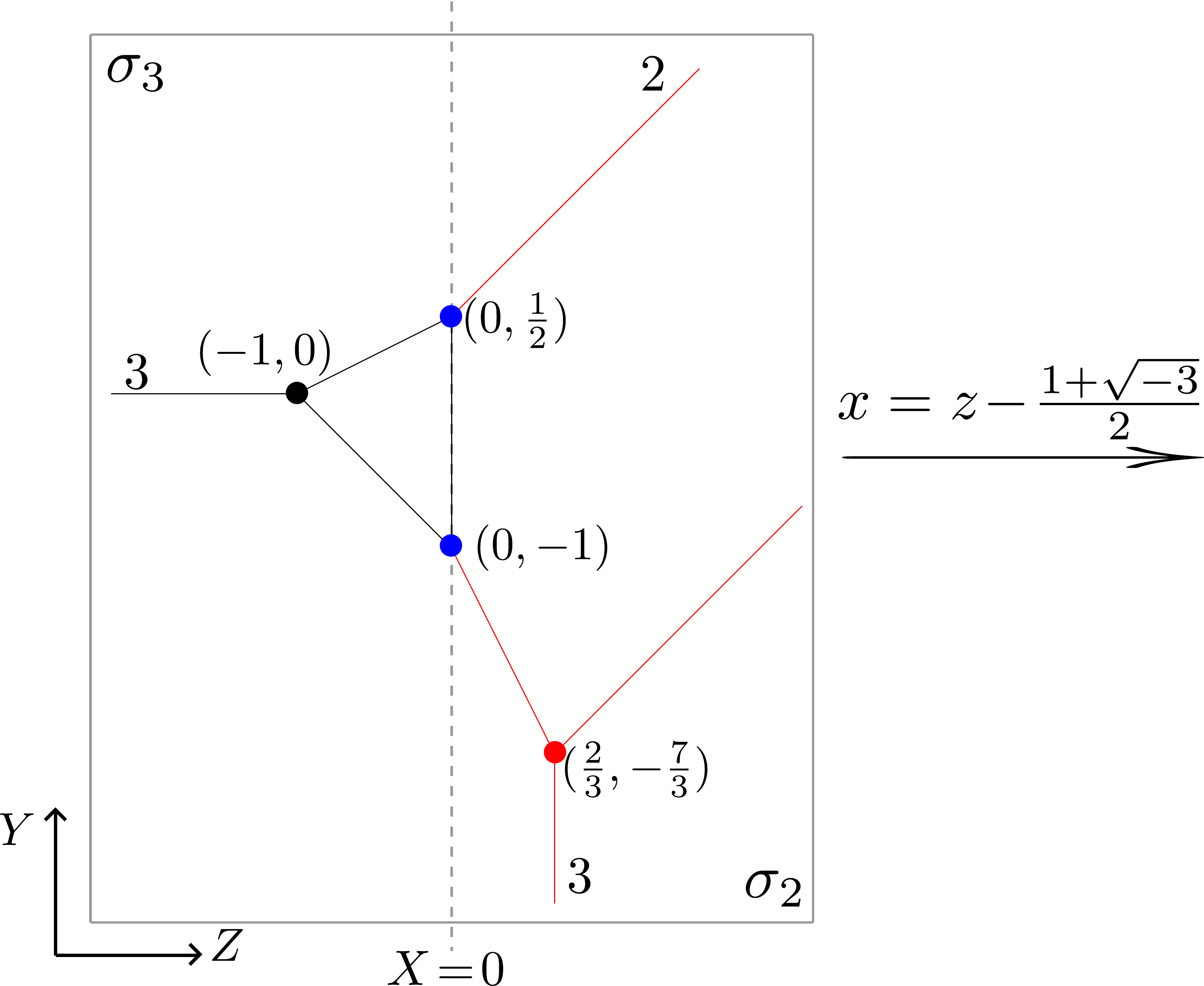}
   \end{minipage}
   \begin{minipage}[r]{0.36\linewidth}
   \includegraphics[scale=0.2]{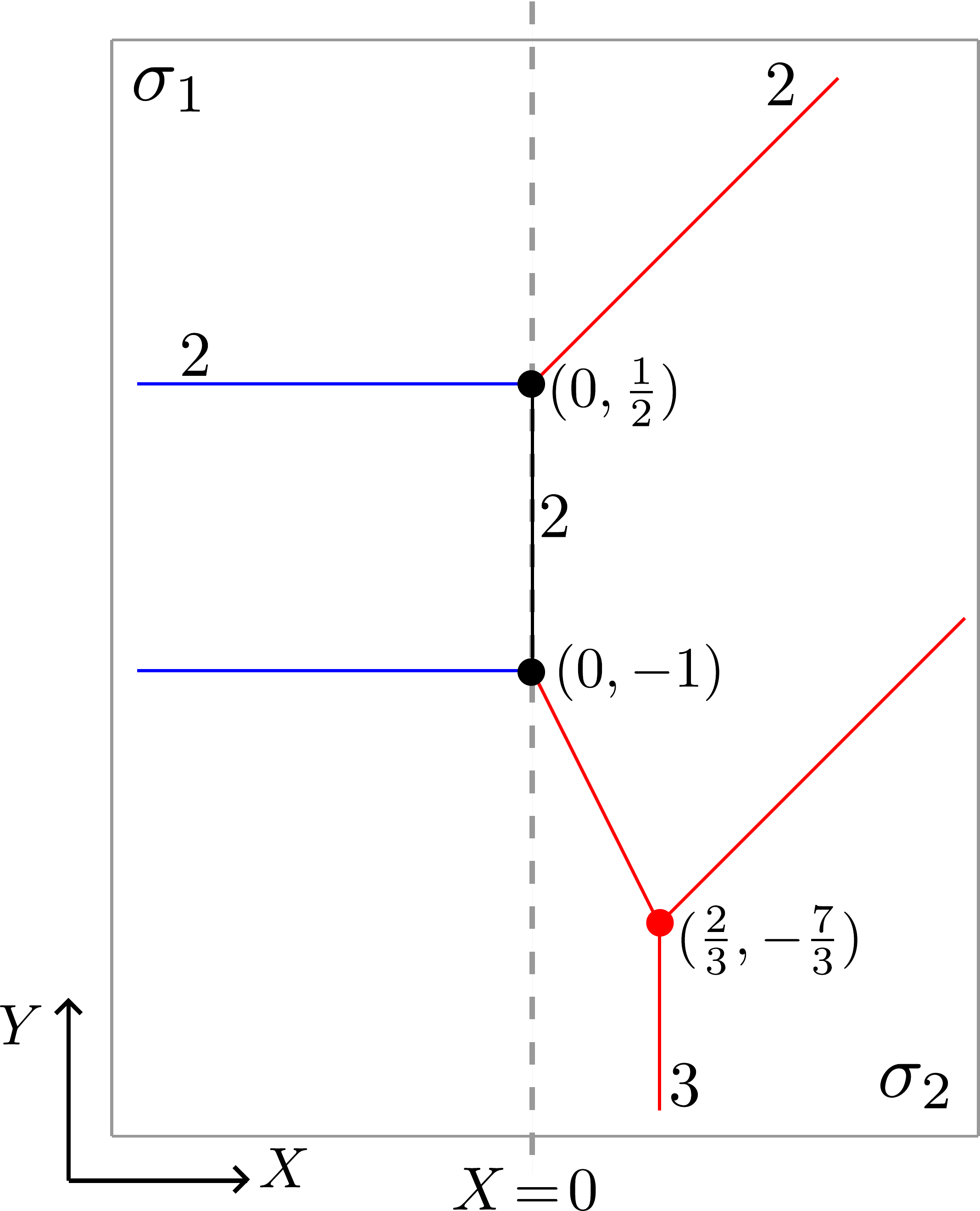}
 \end{minipage}
 \caption{From right to left: Using a linear re-embedding we unfold a
   double vertical edge on the right to produce a cycle on the left, as predicted by Theorem~\ref{thm:fatEdge}.}
\label{fig:unfoldingDoubleEdge}
 \end{figure}
\end{example}

\begin{example}\label{ex:64a} 
We consider the cubic curve  with defining equation
\[
g(x,y)=t^3\,x^3+t^5\,x^2y+t^3\,xy^2+t\,y^3+x^2+3\,xy+t^2\,y^2+(2+3t/2)x+(3+t^2)y+1.
\]
Its tropicalization is depicted in the right of
Figure~\ref{fig:64aTropical}. The 4-valent vertex $(0,0)$ lies in the cycle  in
$\Trop(g)$ and it is locally reducible. It is dual to a height 1 trapezoid and $\Delta_{(0,0)}$ 
vanishes at $\init_{(0,0)}(g)$.

Using Theorem~\ref{thm:redVertex} we unfold edges mapping to the
straight line through this reducible vertex.  We view the re-embedded
curve $\Trop(I_{g,f})$ using the projections $\pi_{XY}$ and $\pi_{ZY}$
in Figure~\ref{fig:64aTropical}, the leftmost being the
tropicalization of the plane curve $
g(z-1,y)$. 
\begin{figure}[htb]
  \centering
\includegraphics[scale=0.17]{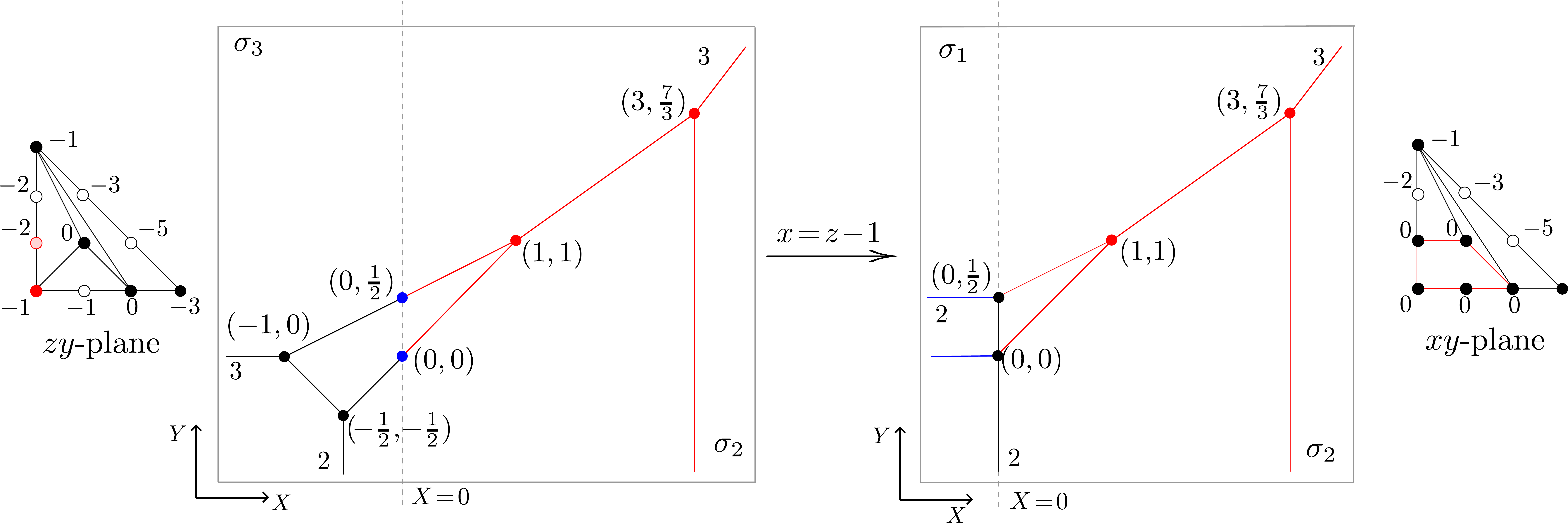}
\caption{From right to left: We use a linear re-embedding to
  unfold edges of a plane tropical curve with a
  $4$-valent reducible vertex traversed by a vertical line.  The
  corresponding image under $\pi_{ZY}$ has equation $\tilde{g}(z,y)=
  g(z-1,y)$. The endmost pictures correspond to the Newton
  subdivisions of $\tilde{g}$ and $g$, respectively.  The
    integer numbers adjacent to each lattice point indicate the
    negative valuation of the corresponding coefficients. The linear re-embedding lowers
    the red points.}
  \label{fig:64aTropical}
\end{figure}
\end{example}
\begin{example}\label{ex:non-ellipticDecontraction}
We consider a smooth degree $4$ plane curve of genus $3$  with defining equation:
\[
g=t^{13}x^4+(1+3t^4)\,x^2y^2+(1-2t^5)\,xy^3+t^{12}y^4+(1+2t^2)x^3+(1-t^3)\,x^2y+t\,xy^2+t^6\,x^2+t^3\,xy+t^{10}\,y+t^{14}.
\]
The induced Newton subdivision is depicted in the right of
Figure~\ref{fig:Mumford} (the missing height values can be taken to be
$-\infty$). The point $(0,0)$ is the unique locally reducible vertex
of the tropical curve $\Trop(g)$ in the right of
Figure~\ref{fig:Mumfordtrop}. Its star is the union of two skew lines:
$X-Y=0$ and $X-2Y=0$.  The local discriminant $\Delta_{(0,0)}$
vanishes at $\init_{(0,0)}(g)$. By Remark~\ref{rem:unimodular}, the
line $X-Y=0$ gives us two ways to locally repair the tropical curve
around $(0,0)$ by a linear re-embedding. We obtain two distinct yet
isometric tropical curves in $\RR^3$ that map to $\Trop(g)$ under the
projection $\pi_{XY}$.

First, we aim to apply the technique described in
Theorem~\ref{thm:redVertex}. In order to do so, we must perform a
unimodular transformation on the tropical curve to fall into our
standard trapezoid from Figure~\ref{fig:trapezoid}. Via a monomial
change of coordinates $\alpha$ and a translation, we make the vertex
$(0,0)$ dual to a unit square as in the right of
Figure~\ref{fig:Mumford}. The skew line $X-Y=0$ maps to the vertical
line $U=0$ in the $UV$-plane.  The corresponding linear tropical
modification induces a linear re-embedding by $z=f(u,v)=u+1$. In order
to see the effect of this change in the original coordinates $x$ and
$y$, we apply the inverse monomial map $\alpha^{-1}$ as in the left of
Figure~\ref{fig:Mumford}.  This procedure yields the desired
re-embedding by the ideal $I_{g,f\circ \alpha^{-1}}=\langle g,
z-(xy^{-1}-1)\rangle$. This result cannot solely be obtained with
linear tropical modifications.

\begin{figure}[htb]
\begin{minipage}[l]{0.5\linewidth}
   \includegraphics[scale=0.24]{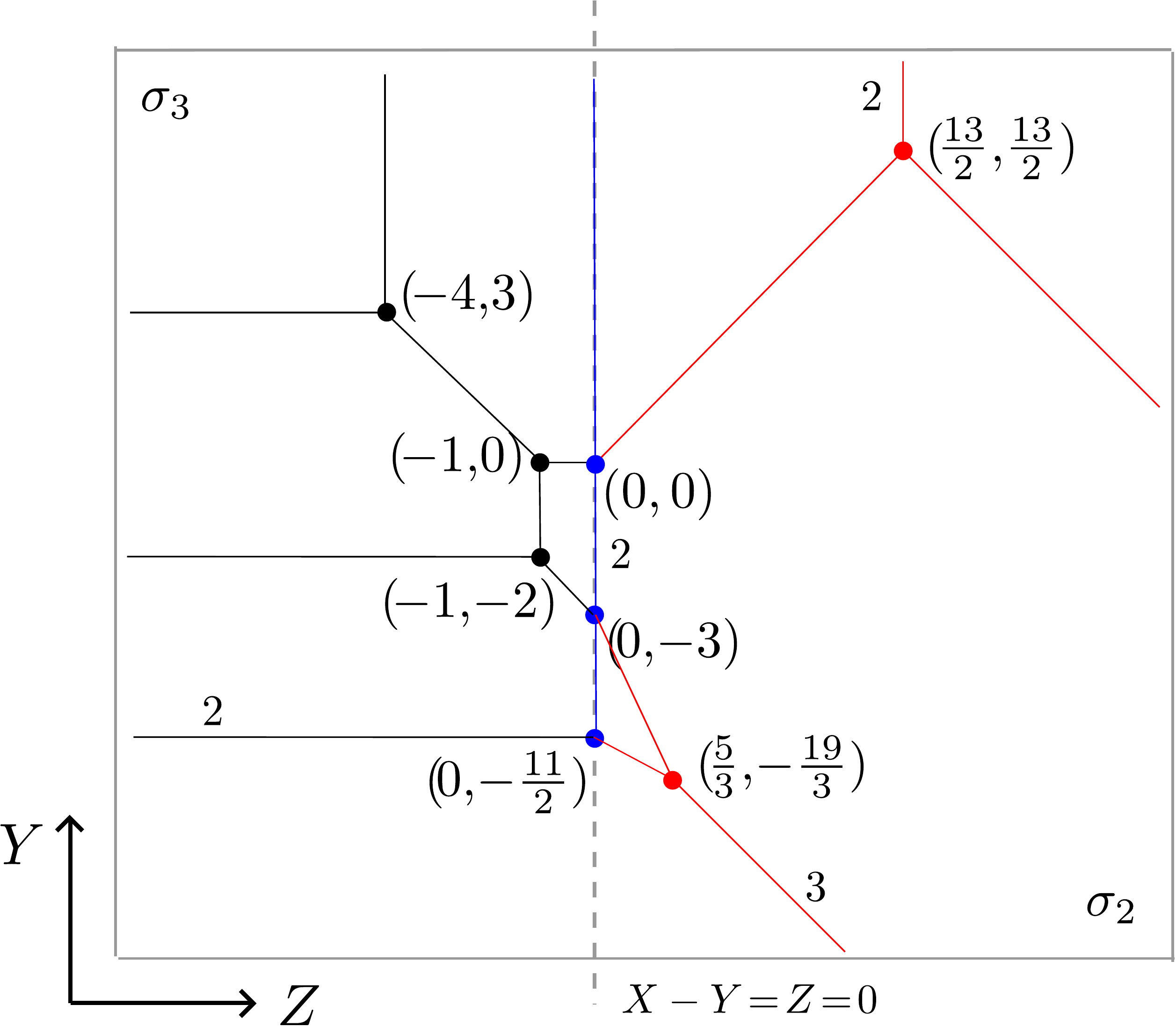}
\hfill
 \end{minipage}
 \begin{minipage}[l]{0.49\linewidth}
  \includegraphics[scale=0.24]{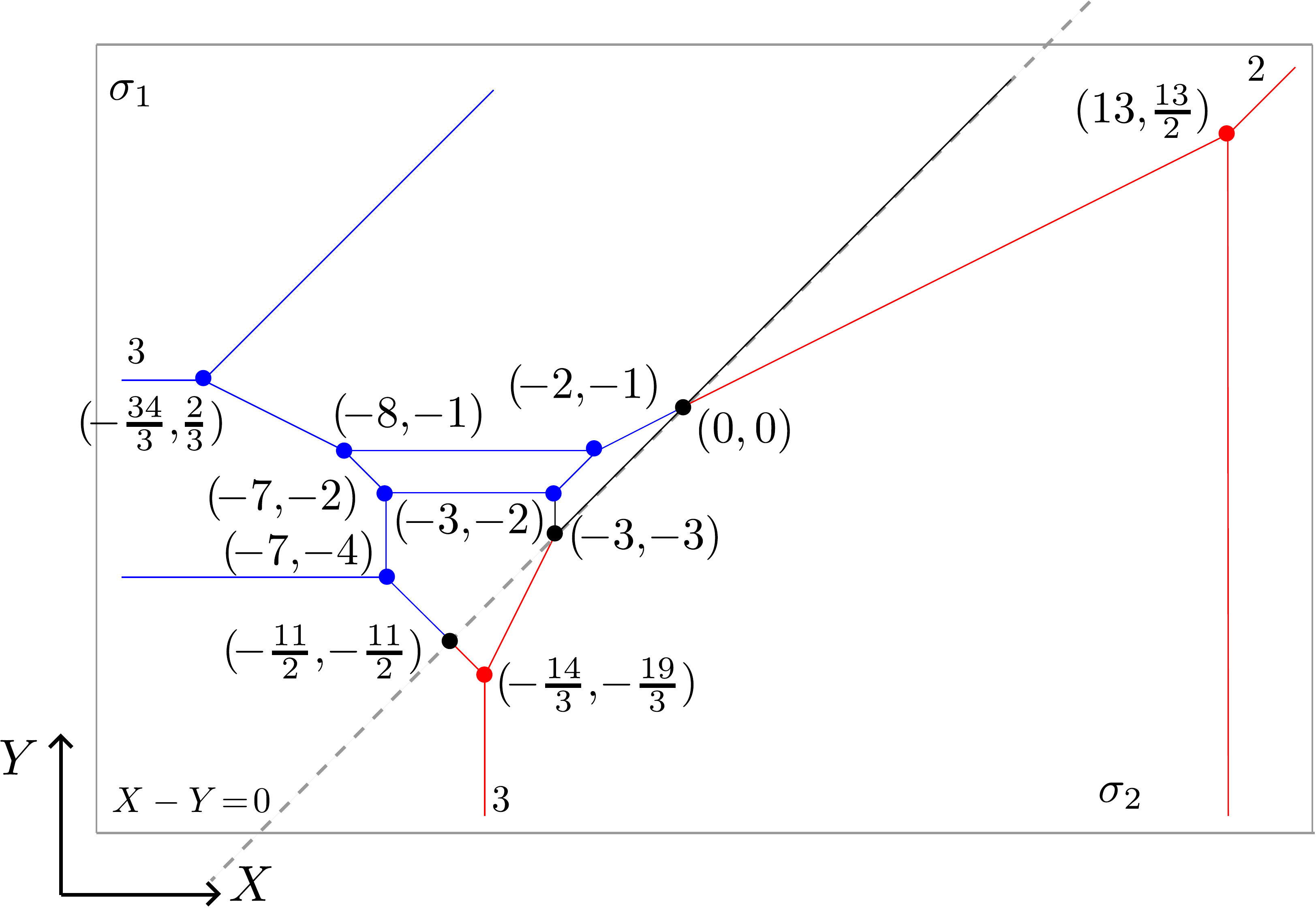}
\end{minipage}
\caption{From right to left: We conjugate a linear tropical
  modification by an affine transformation of $\RR^2$ to decontract an
  edge of a degree 4 plane tropical curve of genus 3 with a $4$-valent
  reducible vertex traversed by a skew line. The corresponding image
  under $\pi_{YZ}$ has equation $\tilde{g}(z,y)=g(y(z-1),y)$.}
  \label{fig:Mumfordtrop}
\end{figure}

The projection of $\Trop(I_{g,f\circ \alpha^{-1}})$ under $\pi_{ZY}$
is the curve $\Trop(\tilde{g})$ in the left of
Figure~\ref{fig:Mumfordtrop}, where $\tilde{g}(z,y)=g(y(z-1),y)$. We
can recover $\Trop(I_{g,f\circ \alpha^{-1}})$ by gluing
$\Trop(\tilde{g})$ and $\Trop(g)$ together with the following rules:
points on $\sigma_3$ and $\sigma_1$ are of the form $(Y,Y,Z)$ and
$(X,Y,0)$, respectively, and all points in $\sigma_2$ satisfy
$Z=X-Y$. The projection $\pi_{XY}$ contracts the edge from $(0,0,-1)$
to $(0,0,0)$ in the re-embedded tropical curve.

Notice that the tropical curve $\Trop(I_{g,f\circ \alpha^{-1}})$ has
three cycles containing only trivalent vertices and multiplicity one
edges.  By~\cite[Theorem 6.23]{BPR11}, the tropicalization map is
faithful on these three cycles. Since the original curve is smooth and
non-rational, we know that its completion admits a unique minimal
skeleton with three cycles (see
Section~\ref{sec:berk-skel-curv}). This skeleton is the complete graph
$K_4$ on 4 nodes and we can see an isometric copy of it in the
re-embedded tropical curve $\Trop(I_{g,f\circ\alpha^{-1}})$ via the
map $\trop$.

  \begin{figure}[htb]
    \centering
    \includegraphics[scale=0.28]{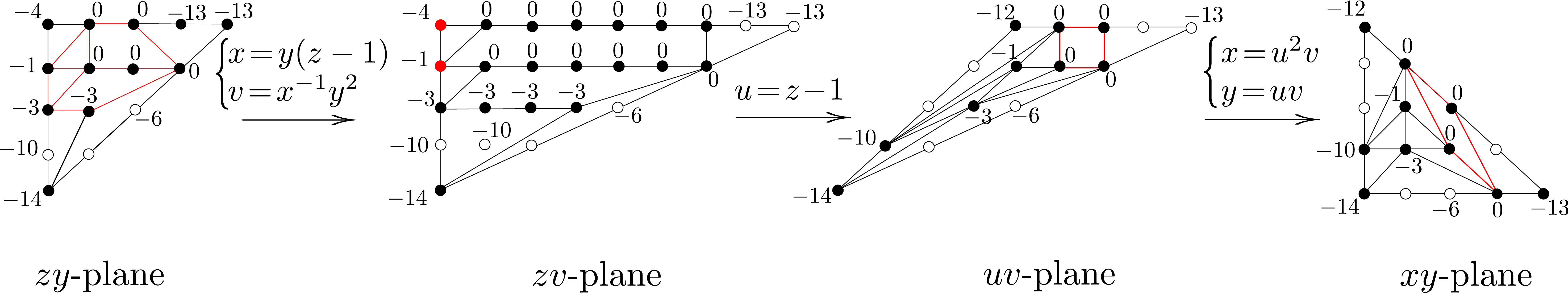}
    \caption{Decontraction of a locally reducible vertex in $\Trop(g)$
      contained in a skew line by conjugating  a linear tropical
      modification of $\RR^2$ (middle map) by monomial change of
      coordinates. The corresponding dual plane tropical curve to the
      endmost Newton subdivisions are depicted in
      Figure~\ref{fig:Mumfordtrop}.}
\label{fig:Mumford}
  \end{figure}
  Our second possibility for locally repairing $\Trop(g)$ around the
  vertex $(0,0)$ is by employing a single a linear tropical
  modification along the skew line $X=Y$ and using the special lifting
  $f(x,y)=x+y$. Figure~\ref{fig:NoMonomial} shows the image of
  $I_{g,x+y}$ under $\pi_{ZY}$ and the impact of this linear change of
  coordinates on the Newton subdivisions of $g(x,y)$ and
  $\tilde{g}(z,y)=g(z-y,y)$. By construction, the projection
  $\pi_{XY}$ sends a point $(x,y,z)$ in $\Trop(I_{g,x+y})\cap
  \sigmaint_3$ to $(y,y)$, so it contracts the edge with endpoints
  $(0,0,-1)$ and $(0,0,0)$, as well as the three ends with directions
  $(0,0,-1)$.
\begin{figure}
    \centering
    \begin{minipage}[l]{0.49\linewidth}
   \includegraphics[scale=0.24]{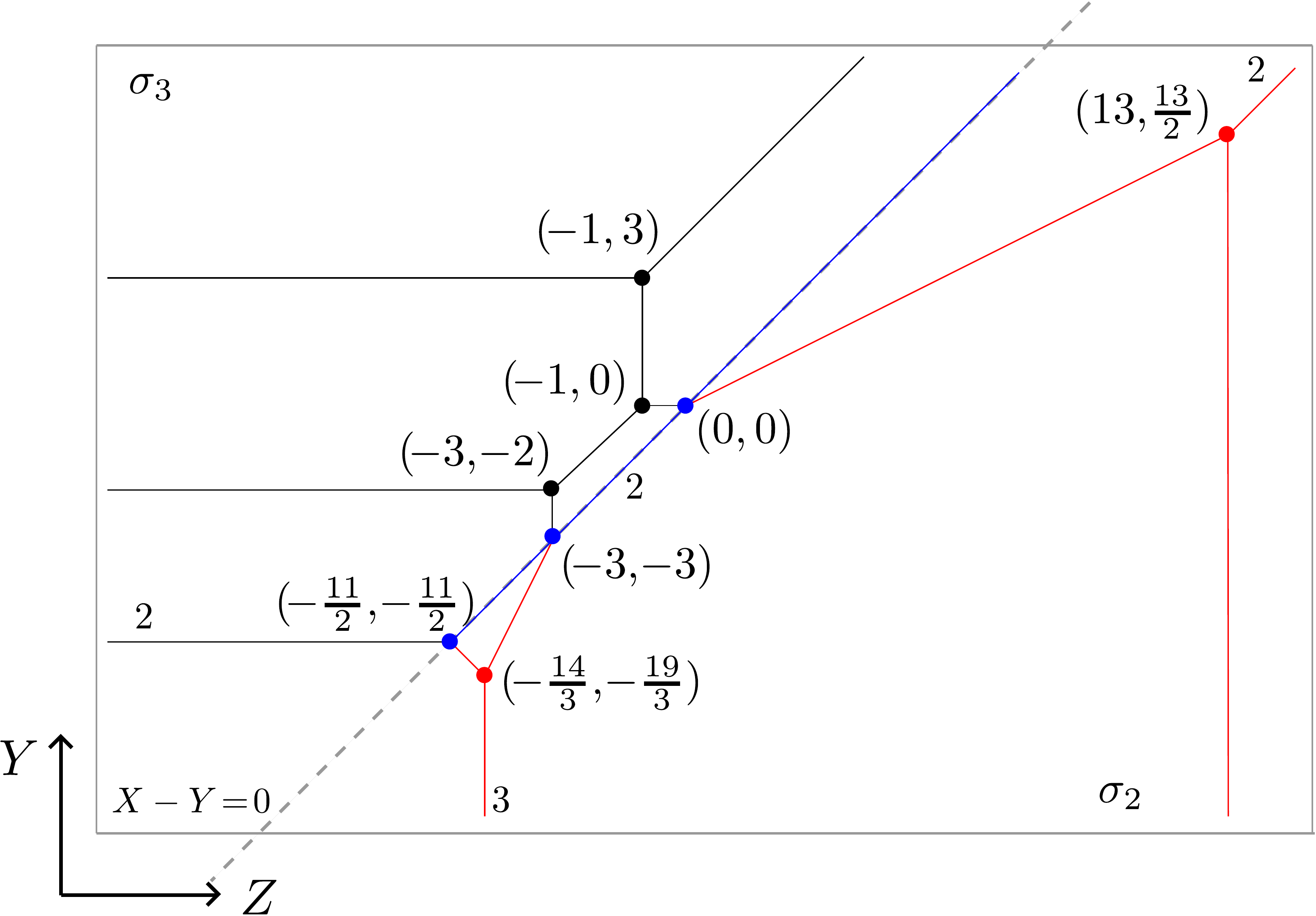}
\hfill
 \end{minipage}
    \begin{minipage}[l]{0.5\linewidth}
   \includegraphics[scale=0.32]{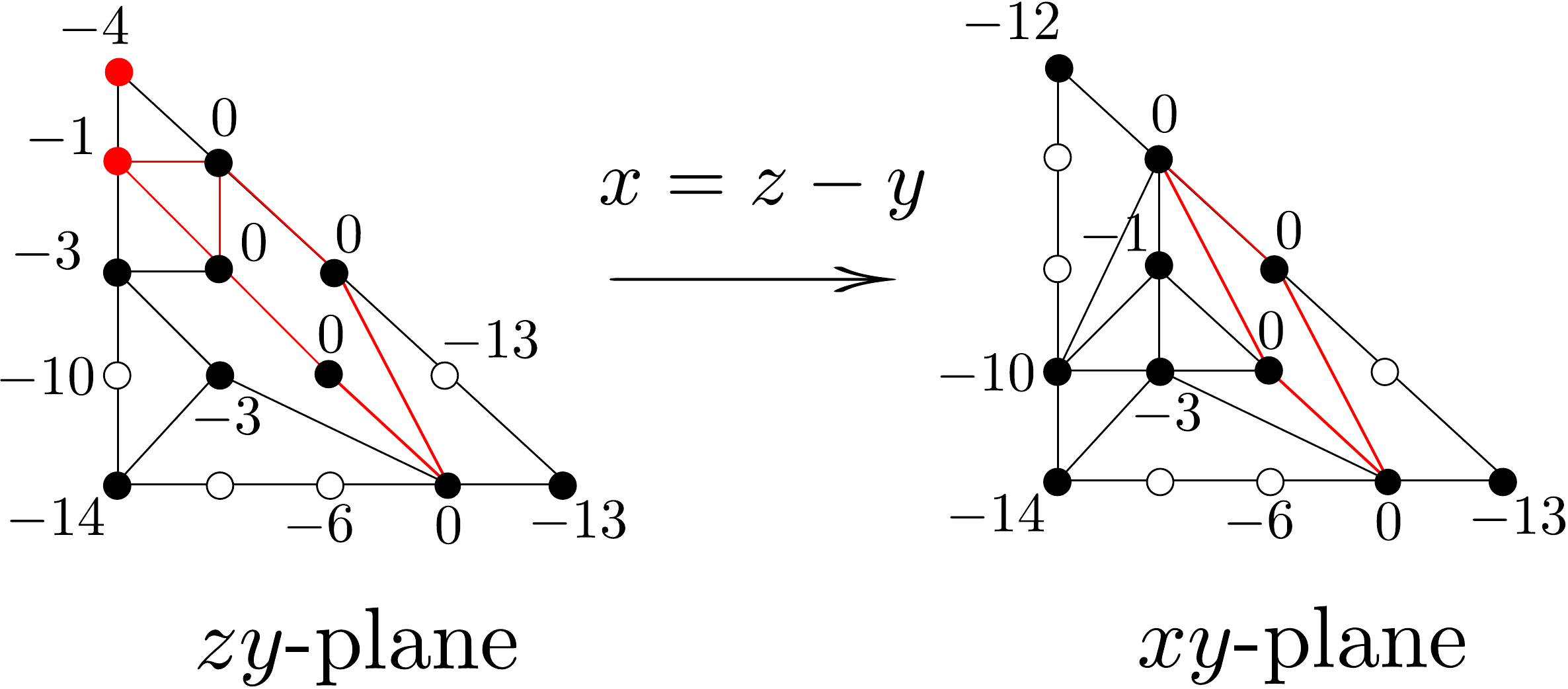}
    \end{minipage}
    \caption{From left to right: The output of a linear tropical
      modification along the skew line $X=Y$ and its effect on the Newton subdivisions of $g$ and $\tilde{g}$. The left most picture is the tropical curve $\Trop(g(z-y,y))$ and it should be compared with the leftmost tropical curve in Figure~\ref{fig:Mumford}.}
    \label{fig:NoMonomial}
  \end{figure}
\end{example}

\section{Tropical elliptic curves}\label{sec:trop-ellipt-curv}
\subsection{The $j$-invariant of a tropical elliptic curve}\label{ssec:known-results-about}
Throughout this section, we fix an equation $g\in \PS[x,y]$ defining
a plane elliptic  cubic, and we fix the valuation of its
coefficients. This data determines a unique tropical elliptic cubic
$\Trop(g)$.  This tropical curve contains a cycle only when the
interior point $(1,1)$ is visible in the Newton subdivision of $g$. It
is this case that interests us the most.  Our starting point is the
well-known formula to compute the $j$-invariant  from $g$:
\begin{equation}
  \label{eq:4}
  j(g)=\frac{A}{\Delta}\in \PS.
\end{equation}

We view $j(g)$ as a degree zero rational function in the coefficients of
$g$, defined over $\QQ$. The denominator
$\Delta$ is the discriminant of the cubic polynomial $g$.
The $j$-invariant has expected valuation
\[
\val_j(g)=\val(A)-\val(\Delta),
\] called the \emph{generic valuation} of the 
$j$-invariant.

In this situation, \cite[Theorem 11]{KMM07} ensures that $-\val_j(g)$
gives the cycle length of the tropical curve $\Trop(g)$.  Furthermore,
\cite[Lemma 23]{KMM07} shows that in this case, failure to have the
expected valuation of $j(g)$ is caused exclusively by an increment in
the valuation of $\Delta$. This means two things: first, the length of
the cycle in $\Trop(g)$ is bounded above by $-\val (j(g))$ and second,
the initial form in the $t$-expansion of $\Delta$ vanishes at $g$. In
the remainder of this section, we use these two facts to repair the
cycle of tropical plane elliptic cubic using linear re-embeddings.

\subsection{How to repair the cycle of a tropical plane elliptic  cubic}\label{sec:repair-elliptic}

The goal of this section is to proof the following theorem.  Its proof will yield a symbolic algorithm for repairing the
cycle of a tropical plane elliptic cubic in dimension 4 (see Algorithm~\ref{alg:repairElliptic}).
\begin{theorem}\label{thm:repairEll} Consider a  plane elliptic 
  cubic in $(\PS^*)^2$ defined by $g\in \PS[x_1,x_2]$. Assume
  $\val(j(g))<0$ and that $\Trop(g)$ contains a cycle whose length
  does not reflect the $j$-invariant. Then, we can recursively repair
  it with linear tropical modifications of the ambient space. The
  resulting ambient space has dimension at most 4.
\end{theorem}

As mentioned in the Introduction, our main tool to prove this result
will be a mild adaptation of Theorem~\ref{thm:redVertex}. A series of
lemmas simplifies the exposition. The heart of these technical
statements will allow us to bound the ambient dimension of the linear
re-embedding after suitable projections.

Our goal is to recursively unfold and decontract edges of the tropical
curve until we obtain the correct cycle length by a linear
re-embedding.  Since we aim at a recursive procedure, we abstain from
applying a monomial coordinate change to $g$ as in
Remark~\ref{rem:unimodular} to put any locally reducible vertex into
our standard trapezoid dual cell from
Figure~\ref{fig:trapezoid}. Instead, we use modification along
horizontal, vertical or slope 1 ``skew'' lines, thus ensuring that
each step of the linear re-embedding gives us back a plane cubic
equation $\tilde{g}$. We let $f$ be an algebraic lift of the tropical
polynomial defining the vertical, horizontal or skew line we use to
modify the plane. Notice that all trapezoids in
Figure~\ref{fig:trapezoid} have a basis of length one, hence the
$t$-initial coefficient of $\AA$ in the lifting $f$ producing unfold
or decontraction of edges (as in
Lemma~\ref{lm:VanishingDecontractsVertex}) will be unique.  When
considering the linearly re-embedded ideal $I_{g,f}$ and the
projections to different charts, we obtain coordinate changes of the
form $x\mapsto z-\AA$ for vertical lines, $y\mapsto z-\AA$ for
horizontal lines and $x\mapsto z-\AA y$ for skew lines, as indicated
in Figure~\ref{fig:shapesandfeeding}.

For the remaining of this section, we let $N$ be the common
denominator of all Puiseux series coefficients of $g$. Then the
coefficients of $g$ are Laurent series in $t^{1/N}$, and the
coordinates of vertices and edge lengths in $\Trop(g)$ are in
$\frac{1}{N}\ZZ$.  The proof of Theorem~\ref{thm:redVertex} shows that
we can always pick $\AA\in \CC(\!(t^{1/N})\!)$, thus the same holds
true for $\Trop(\tilde{g})$ and $\Trop(I_{g,f_1,\ldots, f_r})$.  Each
step of our recursion will increase the cycle length by a positive
number in $\frac{1}{N}\ZZ$.

\begin{figure}[tb]
\includegraphics[scale=0.1]{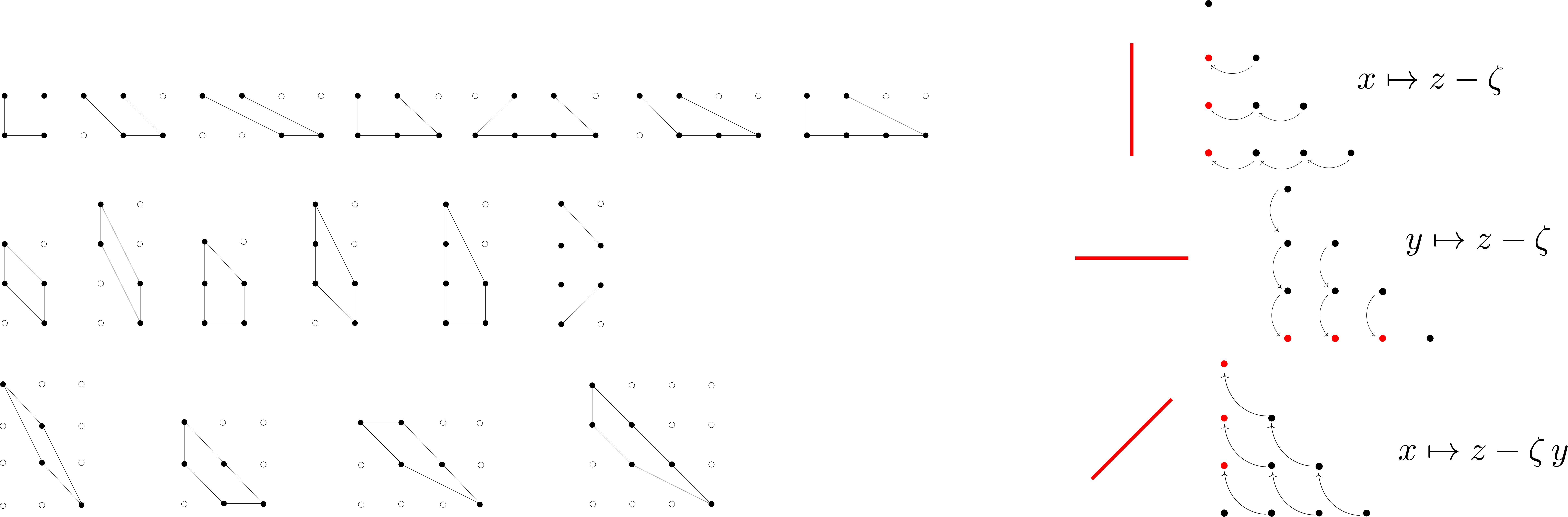}
\caption{From top to bottom: We depict all possible shapes (up to
  reflections) of dual cells to locally reducible vertices in the
  cycle of a tropical plane elliptic cubic. They are ordered according
  to the slope of the line passing through each vertex: the first row
  corresponds to the vertical lines, the second to the horizontal
  lines and the third to the skew lines $X=Y+l$ for $l\in \RR$. The
  remaining cases can be obtained by reflections along each modifying
  lines determined by each row, i.e. $x\mapsto -x$, $y\mapsto -y$ and
  $x\leftrightarrow y$, respectively.  The coordinate change needed to
  describe the linear re-embedding via projections as in Lemma
  ~\ref{lem:ModifViaProjections} and the feeding process in the
  Newton subdivision of $\tilde{g}$ are indicated on the right-hand
  side of the picture.}\label{fig:shapesandfeeding}
\end{figure}

In order to iteratively apply Theorem~\ref{thm:redVertex}, we need to
make sure that the requirements are satisfied at each step. We claim
that the cycle contains a locally reducible vertex with vanishing
local discriminant.  Indeed, since the cycle is shorter than
expected,~\cite[Theorem 11]{KMM07} ensures that the Newton subdivision
of $g$ cannot be a triangulation.  Figure~\ref{fig:shapesandfeeding}
shows all possible non-simplicial 2-cells in the Newton subdivision of
$g$, up to reflection.  All theses cells are equivalent to the
trapezoid in Figure~\ref{fig:trapezoid} by a unimodular
transformation, and $s=1$.  The corresponding dual vertices in
$\Trop(g)$ are locally reducible.  Corollaries~\ref{cor:nonHomeoCycle}
and~\ref{cor:highMultAndDiscrim} guarantee that the discriminant of
one of these locally reducible vertices vanishes. By
Remark~\ref{rem:UnimodularTransformationsForEllipticCase} the feeding
process for any of these locally reducible vertices is verbatim to the
one we considered in detail in
Theorem~\ref{thm:redVertex}. Furthermore, in the notation of
Lemma~\ref{lm:VanishingDecontractsVertex}, we conclude that $k'=1$ for
all such vertices.

The second hypothesis deals with the genericity
convention~\ref{genericityassumption}. In Section~\ref{sec:repa-trop}
we formulated a strong condition to simplify the notation and
arguments in the proof of Theorem~\ref{thm:redVertex}. However, when
restricted to the case of elliptic cubics with a visible cycle in
their tropicalizations, the arguments still carry along even when some
cancellations occur among the coefficients of $\tilde{g}$.

\begin{lemma}\label{lm:TooManyCancellations}
  Assume too many cancellations occur to prevent a decontraction
  or unfolding in the cycle of $\Trop(g)$. Then the tropical
  curve $\Trop(I_{g,f})$ breaks the visible cycle of $\Trop(g)$ and we
  conclude that $g$ defines a rational curve.
\end{lemma}
\begin{example}\label{ex:unfoldcycle}
Consider the plane  cubic with defining equation:
\[
g=t^4\, x^2 y+5t^3 \,xy^2 +t^9\, y^3 +x^2 +3xy+t^2\, y^2 +2x+(3-t^4) y+1.
\] 
Its Newton subdivision is depicted on the right of
Figure~\ref{fig:unfoldcycle}.  We let $\cP$ be the trapezoid with
vertices $(0,0), (0,1), (1,1)$ and $(2,0)$.  We modify $\RR^2$ along a
vertical line passing through the locally reducible vertex dual to
$\cP$. By Lemma~\ref{lm:expectedHeights}, the coefficients
$\tilde{c}_{0,0}$ and $\tilde{c}_{0,1}$ in $\tilde{g}(z,y)=g(z-\AA,y)$
will have strictly positive (unexpected) valuation only when the initial term of
$\AA$ is $-1$.  A simple calculation reveals
\[
\tilde{g}=t^4\, z^2 y+5t^3 \,zy^2 +t^9\, y^3 +z^2 +(3-2t^4)zy+(t^2-5t^3) y^2 .
\] 
In particular, we obtain $\tilde{c}_{0,0}=\tilde{c}_{1,0}=0$.

\begin{figure}
\includegraphics[scale=0.18]{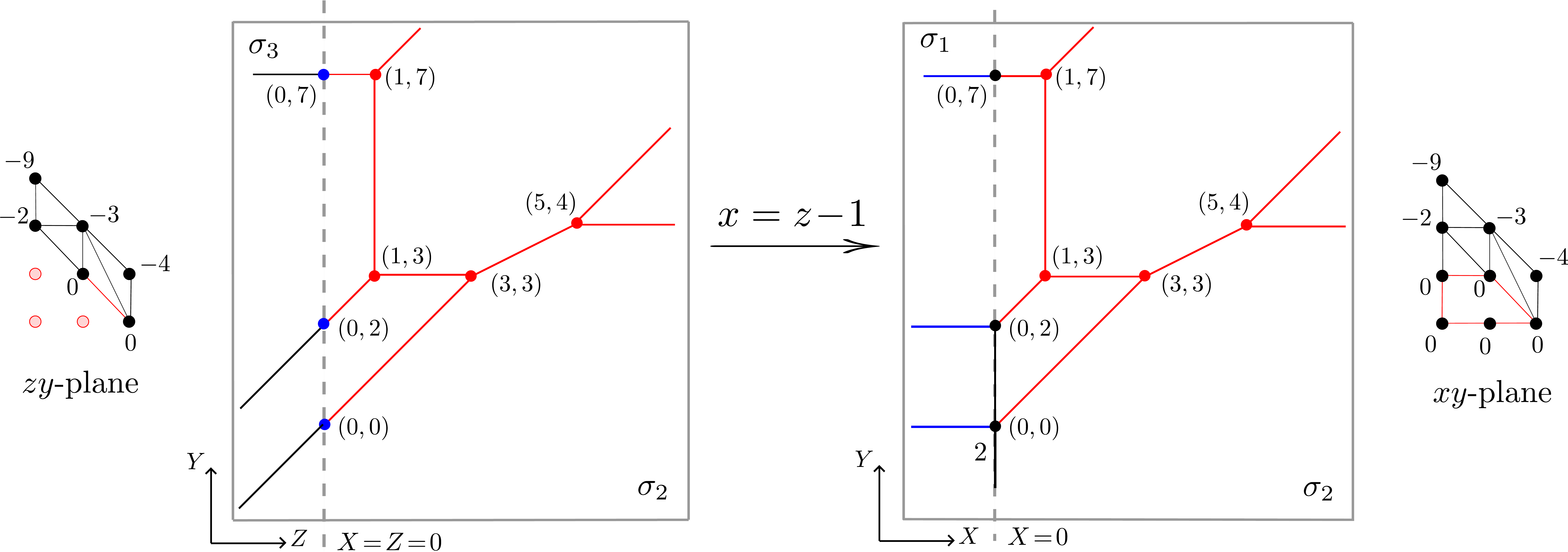}
 \caption{Unfolding a cycle of a rational cubic with a linear re-embedding whose corresponding coordinate change cancels too many monomials.}\label{fig:unfoldcycle}
\end{figure}

The projection $\pi_{ZY}(\Trop(I_{g,x-1}))=\Trop(\tilde{g})$ is shown
in the left of Figure~\ref{fig:unfoldcycle}. Using
Lemma~\ref{lem:ModifViaProjections} we conclude that
$\Trop(I_{g,x-1})\cap \{X=Z=0\}$ consists of three points: $(0,7,0),
(0,2,0)$ and $(0,0,0)$. Furthermore, $\Trop(I_{g,x-1})\cap
\sigmaint_3$ is the union of three ends, and thus $\Trop(I_{g,x-1})$
contains no cycle. The linear re-embedded curve is then rational and
this procedure unfolds the cycle in $\Trop(g)$.
 \end{example}

 \begin{proof}[Proof of Lemma~\ref{lm:TooManyCancellations}]
   We describe which cancellations can prevent a decontraction or
   unfolding of the cycle from $\Trop(g)$. Since the feeding process
   is symmetric for the three families of straight lines along which we
   modify, it is sufficient to prove our claim for a locally reducible
   vertex $v$ dual to the trapezoid $\cP$ from
   Figure~\ref{fig:trapezoid} and for tropical modifications along
   vertical lines.  We stick to the notation used in the proof of
   Lemma~\ref{lm:VanishingDecontractsVertex}. Henceforth, the interior
   point $(1,1)$ is a vertex of the Newton subdivision of $\tilde{g}$
   and it corresponds to the one of the two points $(1,j)$ or
   $(1,j+1)$ in the previous notation.

   Let us first assume it equals $(1,j+1)$. In this case, $\cP$ has
   top vertices $(0,1)$ and $(1,1)$. In the proof of
   Lemma~\ref{lm:VanishingDecontractsVertex}, we argued that the
   Newton subdivision of $\tilde{g}$ contains either the edge with
   vertices $(1,1)$ and $(1,0)$ (responsible for a decontraction of a
   bounded edge), or that the subdivision contains an edge connecting
   $(1,1)$ to $(0,k)$ for $k\leq 1$ (responsible for an
   unfolding). Notice that if certain monomials from $\tilde{g}$
   vanish, the previous two cells need not be in the subdivision.  A
   case-by-case analysis when $g$ is a cubic shows that no unfolding
   or decontraction of edges occurs if and only if none the three
   monomials $1$, $z$ and $y$ are present in $\tilde{g}$. But in this
   case, $\Trop(I_{g,f})\cap \sigma_3$ would not be connected, and we
   would unfold the cycle of $\Trop(g)$ in the linear re-embedding
   $\Trop(I_{g,f})$. By diagram~\eqref{eq:contr/fold}, this
   contradicts our hypothesis that $g$ is an elliptic cubic with bad
   reduction.

   Similarly, if $(1,j)=(1,1)$ we know that $\cP$ has top vertices
   $(1,1)$ and $(2,1)$. The only situation in which a violation of the
   genericity assumption prevents an unfolding or decontraction of
   bounded edges is when all monomials $z^iy^j$ in $\tilde{g}$ with
   $0\leq i+j\leq 2$ cancel completely. This would again lead to an
   unfolding of the cycle, so $g$ cannot be an elliptic cubic.
\end{proof}
\smallskip The following definition allows us to simplify our
repairing techniques:
\begin{definition}\label{def:visibleSide}
  Assume the tropical plane elliptic cubic $\Trop(g)$ contains a cycle
  and a locally reducible vertex $v$ with vanishing discriminant, and
  let $L$ be a straight line in $\Star_{\Trop(g)}(v)$. If $L$ is
  vertical (resp., horizontal) with equation $\{X=l\}$ (resp., $\{Y=l\}$), we say that
  the cycle of $\Trop(g)$ is \emph{on the visible side} of $L$ if it
  lies on the halfspace $L^+:=\{X\geq l\}$ (resp., $L^+:=\{Y\geq l\}$).
\end{definition}
By convexity, we know that the cycle lies entirely on one of the
halfspaces induced by $L$. The definition is motivated by
Lemma~\ref{lem:ModifViaProjections}: under the given conditions, the
cycle of $\Trop(I_{g,f})$ is visible in $\Trop(\tilde{g})$, and its
length is strictly larger than the length of the cycle of
$\Trop(g)$. The equation $\tilde{g}$ gives a planar linear
re-embedding of the elliptic cubic that improves the embedding
induced by $g$.

\begin{remark}\label{rem:skewsAreSpecial}
  Notice that if $L$ is the skew line $L:=\{Y=X+l\}$, we can always
  exchange $Y$ and $X$ so that the cycle of $\Trop(g)$ lies in
  $L^+:=\{Y\geq X+l\}$. In that case, the cycle of $\Trop(I_{g,f})$ is
  visible in $\Trop(\tilde{g}(x,z))$ and we improve the embedding of
  the curve by replacing $g$ with $\tilde{g}(x,z)$. In this sense,
  skew lines are special: the cycle of $\Trop(g)$ is always on the
  visible side of $L$.
\end{remark}

Our next result shows that the combinatorics of the
cycle in $\Trop(g)$ can be simplified by means of affine changes of coordinates 
constructed from tropical modifications of $\RR^2$ along
straight lines.
\begin{lemma} \label{lem:atMostTwo} Consider a plane elliptic cubic in
  $(\PS^*)^2$ with bad reduction defined by $g\in
  \PSN[x_1,x_2]$. Assume that $\Trop(g)$ contains a cycle. Then, there
  exists an affine change of coordinates $A\colon \PSN[z_1,z_2]\to
  \PSN[x_1,x_2]$ such that $\Trop(g\circ A)$ contains a cycle and the
  following conditions hold:
  \begin{enumerate}[(i)]
\item the length of the cycle of $\Trop(g\circ A)$ is bounded below by the length of the cycle of $\Trop(g)$;
\item if $v$ is a locally reducible vertex on the cycle of
  $\Trop(g\circ A)$ with vanishing discriminant, then no skew lines
  traverses $v$ and no horizontal or vertical line through $v$
  contains the cycle of $\Trop(g\circ A)$ on its    {visible side}.
  \end{enumerate}
The affine map  $A$ is constructed by means of tropical modifications
of $\RR^2$ along straight lines.
\end{lemma}
\begin{proof}
  We proceed by induction on $q=-\val(j(g))-\ell(g)\in
  \frac{1}{N}\ZZ_{\geq 0}$, where $\ell(g)$ denotes the length of the
  cycle in $\Trop(g)$.  If $q=0$, the cycle in $\Trop(g)$ contains no
  locally reducible vertex with vanishing discriminant by
  Corollary~\ref{cor:highMultAndDiscrim}. We take
  $A:=\id_{\PSN[x_1,x_2]}$.

  Assume $q>0$ and that the result is true for all $r<q$ with $r\in
  \frac{1}{N}\ZZ_{\geq 0}$.  Since $q>0$ we know that $\Trop(g)$ does
  not reflect the $j$-invariant of $g$, and so the cycle of $g$
  contains a locally reducible vertex. If $g$ satisfies \emph{(ii)},
  we take $A=\id_{\PSN[x_1,x_2]}$. If condition \emph{(ii)} fails,
  there are two reasons for this. We analyze each case separately.

  First, assume the cycle of $\Trop(g)$ contains a locally reducible
  vertex $v$ traversed by a skew line $L:=\{Y=X+A\}$. By symmetry we
  may assume that the cycle of $\Trop(g)$ lies in $L^+$. Let $f=y+\AA
  t^{-A} x$ be as in Lemma~\ref{lm:VanishingDecontractsVertex} and
  Remark~\ref{rem:UnimodularTransformationsForEllipticCase} with
  $\val(\AA)=0$, and write $\tilde{g}(z,y):=g(\AA^{-1}t^{A}(z-y),
  y)$. Remark~\ref{rem:skewsAreSpecial} ensures that the cycle of
  $\Trop(I_{g,f})$ is visible in $\Trop(\tilde{g})$. By construction,
  $\val(j(g))=\val(j(\tilde{g}))$ and $\ell(g)<\ell(\tilde{g})\in
  \frac{1}{N}\ZZ$ thus $r:=-\val(j(\tilde{g}))-\ell(\tilde{g})<q$ and
  $r\in \frac{1}{N}\ZZ$.  By the inductive hypothesis, we can find an
  affine transformation $A'\colon \PS[z_1,z_2]\to \PS[z,y]$ satisfying
  conditions \emph{(i)} and \emph{(ii)} for $\tilde{g}$. The map $A=
  (\AA^{-1}t^{A}(z-y), y) \circ A'
  $ verifies the desired requirements.

  Finally, suppose no skew line traverses any locally reducible vertex
  in $\Trop(g)$ with vanishing discriminant, but the cycle of
  $\Trop(g)$ is on the visible side of a horizontal or vertical line
  $L$ through one of these vertices. By symmetry, we can assume
  $L:=\{X=A\}$.  Let $f=x+\AA t^{-A}$ be as in
  Lemma~\ref{lm:VanishingDecontractsVertex} with $\AA\in
  \CC(\!(t^{1/N})\!)$ and $\val(\AA)=0$. Write $\tilde{g}(z,y)=g(z-\AA
  t^{-A},y)$. As in the previous case, our choice ensures that
  $\val(j(g))=\val(j(\tilde{g}))$ and
  $r:=-\val(j(\tilde{g}))-\ell(\tilde{g})<q$ with $r\in
  \frac{1}{N}\ZZ_{\geq 0}$. We use the inductive hypothesis to
  construct an affine map $A'$ for $\tilde{g}$. The map
  $A=(z-\AA\,t^{-A},y)\circ A'$ satisfies both conditions in the
  statement.
\end{proof}

\begin{remark}\label{rem-v0v1v2}
  Assume the cubic polynomial $g$ satisfies the conditions of
  Lemma~\ref{lem:atMostTwo}. A simple calculation shows that the cycle
  in $\Trop(g)$ admits only seven possible shapes for locally
  reducible vertices and their dual cells in the Newton subdivision of
  $g$. They are depicted in Figure~\ref{fig:shapesandfeeding}. We
  conclude that at most two of these vertices will have vanishing
  discriminants. We call them $v_1$ and $v_2$. Furthermore, we assume
  that $v_1$ is traversed by the horizontal line $L_1:=\{X=0\}$, and
  $v_2$ (if it exists) is traversed by the vertical line
  $L_2:=\{Y=a\}$ with $a>0$. The cycle of $\Trop(g)$ is not on the
  visible side of $L_1$ nor $L_2$.

  \begin{figure}[tb]
    \centering
    \includegraphics[scale=0.1]{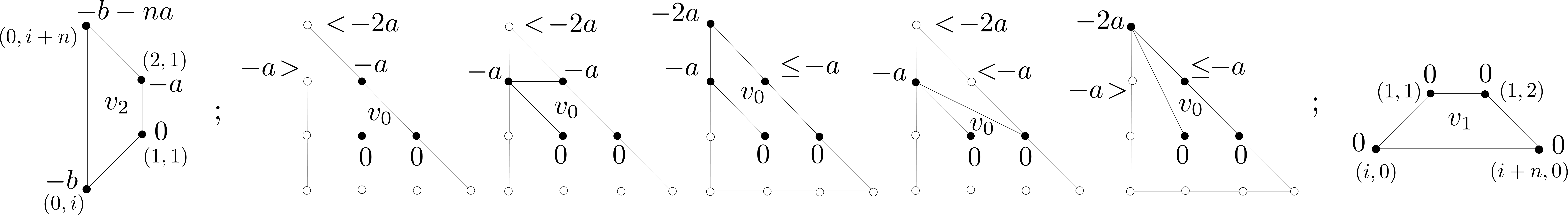}
    \caption{Combinatorics and heights of the dual cells to the
      distinguished vertices $v_2,v_0$ and $v_1$. Here, $i=0,1$ and we
      disallow the combination $i=n=1$.}
    \label{fig:figuresAlgorithm}
  \end{figure}

  Let $v_0$ be the vertex in the cycle of $\Trop(g)$ adjacent to $v_1$
  and lying on $L_1$. Figure~\ref{fig:figuresAlgorithm} shows the
  combinatorics of the distinguished dual cells
  $v_0^{\vee},v_1^{\vee}$ and $v_2^{\vee}$ in the Newton subdivision
  of $g$ and the corresponding heights, where $a,b>0$.  Notice that
  condition \emph{(ii)} from Lemma~\ref{lem:atMostTwo} ensures that
  $i=0,1$ and for both $v_1^{\vee}$ and $v_2^{\vee}$. We disallow the
  combination $i=n=1$. Furthermore, the cycle of $\Trop(g)$ lies on
  the halfspace $\{Y \geq 0\}$.

\end{remark}

Our proof strategy for Theorem~\ref{thm:repairEll} will consist on
repairing the vertices $v_1$ and $v_2$ separately. The following
result allows us to repair the cycle of the curve locally around $v_1$
by a linear re-embedding in dimension 3. We will use a variant of this
result when dealing with $v_2$.
\begin{lemma}\label{lm:repairV_1InDim3}
  Let $g\in \PSN[x,y]$ define a plane elliptic cubic with bad
  reduction, satisfying the conditions of Lemma~\ref{lem:atMostTwo}
  for $\varphi=\id_{\PSN[x,y]}$. Assume $\Trop$ contains a locally
  reducible vertex $v_1$ with vanishing discriminant, traversed by the
  line $L:=\{X=0\}$. Then, there exists an affine map $\psi\colon
  \PSN[x,y]\to \PSN[x,y]$ with $\psi(x)=x$, $\psi(y)=y+\alpha$ and
  $\val(\alpha)\geq 0$, and a polynomial $f_{r+1}:=x+\AA_1 + \AA_{11}
  t^{-A_1}+\ldots + \AA_{1r}t^{-A_r}\in \PSN[x]$ with $\val(\AA_1)=\val(\AA_{11})=\ldots =\val(A_{1r})=0$ 
  and  $0>A_1>\ldots>A_r$
  satisfying the following conditions:
\begin{enumerate}[(1)]
\item $I_{g\circ \psi, f_{r+1}}\subset \PSN[x,y,z_{1r}]$ is a
  linear re-embedding of $g\circ \psi$ constructed from a linear
  tropical modification of $\RR^2$ adapted to $g\circ \psi$ as in
  Lemma~\ref{lm:VanishingDecontractsVertex};
\item the weighted sets $\Trop(g)\cap \sigma_1^{\circ}\cap\{Y\geq 0\}$ and $\Trop(g\circ \psi)\cap \sigma_1^{\circ}\cap \{Y\geq 0\}$ agree;
\item $I_{g\circ \psi,f_{r+1}}$ contains no locally reducible
  vertex with vanishing discriminant in $\sigma_3$.
\end{enumerate}
Furthermore, assume $g$ contains a second locally reducible vertex
$v_2$ with vanishing discriminant, traversed by the $L_2:=\{Y=a\}$
with $a>0$. Then, the cycle of $\Trop(I_{g\circ \psi,f_{r+1}})$ lies
on the halfspace $L_2^-:=\{Y\leq a\}$ and it meets the hyperplane
$(Y=a)$ only along the edge joining $v_0$ and $v_2$ in $\sigma_1$.
\end{lemma}
\begin{proof}   
  We pick a suitable function $f_1:=x+\AA_1$ where $\val(\AA_1)=0$ 
as
  in Lemma~\ref{lm:VanishingDecontractsVertex}. We set
  $\tilde{g}(z_1,y)=g(z_1-\AA,y)\in \PSN$.  The induced linear
  re-embedding $I_{g,f_1}$ decontract/unfolds edges of
  $\widehat{\Sigma}(g)\smallsetminus D_g$ mapping to $L_1$ in
  $\Trop(g)$. In particular, $0<\ell(I_{g,f_1})-\ell(g)\in
  \frac{1}{N}\ZZ$.  By Lemma~\ref{lem:ModifViaProjections}, we know
  that the cycle of $I_{g,f_1}$ lies in the cells $\sigma_1\cup
  \sigma_3$. The point $v_0$ is a vertex of $\Trop(\tilde{g})$ and its
  dual cell has four possible shapes, namely all polygons in the
  center of Figure~\ref{fig:figuresAlgorithm} except for the triangle
  with vertices $(1,1), (1,2)$ and $(2,1)$. 

  Notice that $v_1$ is also a vertex of $\Trop(\tilde{g})$. It can be
  trivalent or locally reducible, with or without vanishing
  discriminant. Lemma~\ref{lm:VanishingDecontractsVertex} ensures that
  when viewed in $\Trop(I_{g,f_1})$, the multiplicity of $v_1\times
  \{0\}$ is 1, even though its multiplicity on $\Trop(\tilde{g})$ can
  be 2. Similarly, the vertex $v_0\times \{0\}$ in
  $\Trop(I_{g,f_{r+1}})$ has multiplicity 1.

  By construction, the curve $\Trop(\tilde{g})$ contains at most one
  locally reducible vertex $v_{11}$ with vanishing discriminant in the
  cell $\sigma_3^{\circ}$. Furthermore, such vertex is traversed by a
  horizontal or vertical line $L_{11}$, with equations $Z_1=A_1$ or
  $Y=B_1$, and $A_1,B_1 \in \frac{1}{N}\ZZ$ satisfy $A_1<0$ and
  $B_1\leq 0$. In both cases, the cycle of $\Trop(I_{g,f})$ lies on
  the hyperplane $L_{11}^+$.

  We construct both elements $\psi$ and $f_{r+1}$ in a recursive
  fashion, by performing tropical modifications of $\RR^2_{Z_1,Y}$
  along $r$ horizontal and $s$ vertical lines of the form $Z_1=A_i$ or
  $Y=B_i$ for suitable $A_1>A_2>\ldots>A_r$ in $\frac{1}{N}\ZZ$ and
  $B_1>B_2>\ldots>B_s$.  We proceed by induction on
  $q:=-\val(j(g))-\ell(I_{g,f_1}) \in \frac{1}{N}\ZZ_{\geq 0}$.
  Notice that $0\leq q<-\val(j(g))$.  If $q=0$, \cite[Theorem
  11]{KMM07} ensures that the vertex $v_{11}\in \sigma_3^{\circ}$ does
  not exist. Our statement is satisfied by the choice
  $\psi=\id_{\PSN[x,y]}$ and $r=0$.

  Next, we assume the statement holds for all $0\leq k<q$ in
  $\frac{1}{N}\ZZ$ and that the problematic vertex $v_{11}$ does
  indeed exist. We use $L_{11}$ and
  Lemma~\ref{lm:VanishingDecontractsVertex} to repair the embedding
  around $v_{11}$, by working with the plane with coordinates
  $(Z_1,Y)$. In principle, this produces a linear re-embedding
  $I_{\tilde{g}, f_{11}}\subset \PSN[z_1,y,z_{11}]$ of the curve
  $\tilde{g}$, and in turn, a new linear re-embedding
  $I_{g,f_1,f_{11}}:=\langle g,z_1-f_1,z_{11}-f_{11}\rangle \subset
  \PSN[x,y,z_1,z_{11}]$ of the input curve $g$. We claim that we can
  simplify the situation by a linear projection, and produce a linear
  re-embedding of $g$ in dimension $3$ that preserves the cycle of
  $\Trop(I_{g,f_1,f_{11}})$. The projection depends on the nature of
  $L_{11}$.

  If $L_{11}:=\{Z_1=A_1\}$, our lifting function has the form
  $f_{11}:=z_1+\AA_{11}t^{-A_1}$ with $\val(\AA_{11})=0$.
  We collect $f_1$ and $f_{11}$, and define
  \[f_2:=x+\AA_1+\AA_{11}t^{-A_1}\quad\text{ and }\quad
  \tilde{g}'(z_{11},y):=\tilde{g}(z_{11}-f_{11},y)= g(z_{11}-\AA_1-
  \AA_{11}t^{-A_1},y).\] The function $z_{11}-f_2$ belongs to the ideal
  $J:=\langle z_1-f_1,z_{11}-f_{11}\rangle$. 

  The given two generators of $J$ form a tropical basis. The tropical
  plane $\Trop(J)$ contains $2$ edges and 5 two-dimensional cells
  $\sigma_{12},\sigma_{22}, \sigma_{32}, \sigma_{31}$ and $\sigma_{33}$.
  Each $\sigma_{ij}$ is obtained by intersecting the 2-cell $\sigma_i$
  corresponding to $\Trop(z_1-f_1)$ and the 2-cell $\sigma_j$ of
  $\Trop(z_{11}-f_{11})$ (each viewed in $\RR^4$). For example,
  $\sigma_{12}:=\{X\leq 0,Z_{11}=Z_{1}=0\}$.

  By Proposition~\ref{pr:projectionsJ}, we recover
  $\Trop(I_{g,f_1,f_{11}})\subset \Trop(J)$ from the three projections
  $\Trop(g)$, $\Trop(\tilde{g}(z_1,y))$ and
  $\Trop(\tilde{g}'(z_{11},y))$.  Notice that $\Trop(\tilde{g})$ lies
  in the visible side of $L_{11}\subset \RR^2_{Z_1,Y}$, so the cycle
  of $\Trop(I_{\tilde{g},f_{11}})$ is visible on $\Trop(\tilde{g}')$.
  Thus, the cycle of $\Trop(I_{g,f_1,f_{11}})$ lies on the union
  $\sigma_{32}\cup \sigma_{33}\cup\sigma_{12}$. These cells are
  parametrized by the pairs $(Z_{11},Y)$, $(Z_{11},Y)$ and $(X,Y)$,
  respectively. The linear re-embedding $I_{g,f_{2}}\subset
  \PSN[x,y,z_{11}]$ contains a cycle isometric to the cycle in
  $I_{g,f_1,f_{11}}$. It is induced by the lifting $z_{11}-f_2$, i.e.\
  by the linear tropical modification of $\RR^2$ (with coordinates
  $(X,Y)$) along $X=0$. We take $\psi_1:=\id_{\PSN[x,y]}$.

  On the other hand, assume $L_{11}:=\{Y=B_1\}$. We proceed in a
  similar fashion to the previous case. We work with the curve
  $\Trop(\tilde{g})$, a linear tropical modification of
  $\RR^2_{Z_1,Y}$ along $L_{11}$ and a lifting
  $f_{11}:=y+\AA_{11}t^{-B_1}$ with $\AA_{11}$ as in
  Lemma~\ref{lm:VanishingDecontractsVertex} that prolongs the cycle of
  $\Trop(\tilde{g})$ in $\Trop(I_{\tilde{g},f_{11}})$.  As before, the
  plane $J:=\langle z_1-f_1,z_{11}-f_{11}\rangle$ is generated by a
  tropical basis, and the cycle of $\Trop(I_{g,f_1,f_{11}})$ lies in
  the union of the cells $\sigma_{12}$, $\sigma_{32}$ and
  $\sigma_{33}$. Furthermore, it lies on the hyperplanes
  $L_2^-:=\{Y\leq a\}$ and $L_{11}^+:=\{Y\geq B_1\}$.  We define
  $f_2:=f_1\in \PSN[y]$.  The lifting $f_{11}$ defines an affine map
  $\psi_1\colon \PSN[x,y]\to \PSN[x,y]$ with $\psi_1(x)=x$ and
  $\psi_1(y)=f_{11}$.

  In both cases, the points $v_0,v_2$ are vertices on the cycle of
  $\Trop(I_{g\circ \psi_1,f_2})$. Furthermore, our hypotheses ensure
  that the cycles of $\Trop(g\circ \psi_1)$ and $\Trop(g)$ agree, and
  $I_{g\circ \psi_1, f_2}\subset \PSN[x,y,z_{11}]$ is obtained by a
  tropical modification of $\RR^2$ along $X=0$ adapted to the curve
  $g\circ \psi_1$ as in Lemma~\ref{lm:VanishingDecontractsVertex}. In
  particular, $\ell(I_{g,f_1})\leq \ell(I_{g\circ \psi_1,f_2}) \in
  \frac{1}{N}\ZZ$.  In addition to the potential vertex $v_2$, the
  cycle of $\Trop(I_{g,f_2})$ contains at most one other locally
  reducible vertex $v_{12}$ with vanishing discriminant. This vertex
  lies in $\sigma_3^{\circ}$ and it is traversed by a straight line of
  the form $\{Z_{11}=A_2\}$ or $\{Y=B_2\}$ where $A_2,B_2\in
  \frac{1}{N}\ZZ$ satisfy $A_1>A_2$ and $B_1>B_2$, whenever
  applicable.

  If no such vertex $v_{12}$ exists, the functions $\psi_1$ and $f_2$
  satisfy the requirements of the statement. In the presence of the
  problematic vertex $v_{12}$, we define
  $k:=-\val(j(g\circ\psi_1))-\ell(I_{g\circ \psi_1,f_2})$. By the
  inductive hypothesis, there exists $r>0$, an affine map
  $\psi_2\colon \PSN[x,y]\to \PSN[x,y]$ with $\psi_2(x)=x$ and
  $\psi_2(y)=y+\alpha$ with $\val(\alpha)\geq 0$, and a polynomial
  $f_{r+1}$ of the form
  $f_{r+1}=x+\AA_1+\AA_{11}t^{-A_1}+A_{12}t^{-A_2}+\ldots +
  A_{1r}t^{-A_r}$ satisfying the conditions of the statement for the
  curve $g\circ \psi_1$. The function $\psi:=\psi_1\circ \psi_2$ and
  the polynomial $f_{r+1}$ verify the result.

  Finally, assume $\Trop(g)$ contains the problematic vertex
  $v_2$. Write $\widetilde{g\circ \psi}(z_{1r},y):=
  I_{g\circ\psi,f_{r+1}}\cap \PSN[z_{1r},y]$.  An easy convexity
  argument shows that the dual cell to $v_0$ in the Newton subdivision
  of $\widetilde{g\circ \psi}$ is the parallelogram or the trapezoid
  in the center of Figure~\ref{fig:figuresAlgorithm}.  It follows that
  $\{Y=a\}\cap \Trop(\widetilde{g \circ \psi}) = \{v_0\}$. The last
  claim in the statement follows from
  Lemma~\ref{lem:ModifViaProjections} and condition \emph{(2)}.  This
  concludes our proof.
\end{proof}

\begin{proof} [Proof of Theorem~\ref{thm:repairEll}] Fix $N$ such that
  $g\in \PSN[x,y]$.  Since the cycle of $\Trop(g)$ does not reflect
  the $j$-invariant of $g$, by Theorem~\ref{thm:redVertex} we know
  that the cycle of $\Trop(g)$ contains a locally reducible vertex
  with vanishing discriminant. After applying an affine change of
  coordinates in $\PSN^2$ as in Lemma~\ref{lem:atMostTwo}, we may
  assume that $g$ satisfies the conditions of the lemma for
  $A=\Idd_{\PSN[x,y]}$. Furthermore, by Remark~\ref{rem-v0v1v2}, we
  know that $\Trop(g)$ contains a vertex $v_1$ with vanishing
  discriminant traversed by the vertical line $\{X=0\}$ and a
  potential vertex $v_2$ with vanishing discriminant traversed by the
  horizontal line $\{Y=a\}$, where $a>0$.

  Following earlier notation, we let $\ell(I)\in \frac{1}{N}\ZZ$ be
  the length of the cycle of $\Trop(I)$ induced by a linear
  re-embedding $I\subset \PSN[x,y,z_1,\ldots,z_r]$ of the input plane
  elliptic cubic $g$. The ideal $I$ will be constructed by iterative
  applications of Theorem~\ref{thm:redVertex} and  Lemma~\ref{lm:VanishingDecontractsVertex}.

  Using Lemma~\ref{lm:repairV_1InDim3}, we repair the cycle locally
  around $v_1$ via the linear re-embedding $I_{g\circ
    \psi,f_{1}}\subset \PSN[x,y,z_{1}]$ in dimension 3.
  Lemma~\ref{lem:ModifViaProjections} ensures that the points
  $v_0\times \{0\}$ and $v_1\times \{0\}$ are vertices of
  $\Trop(I_{g\circ \psi,f_1})$ and their multiplicity is 1.

  First, suppose that the potential problematic vertex $v_2$ does not
  exist. Proposition~\ref{pr:projectionsJ} then ensures that all
  remaining vertices in the cycle of $I_{g\circ \psi,f_{1}}$ have
  multiplicity 1. Indeed, the vertices of the cycle in in the relative
  interior of each chart $\sigma_i$ ($i=1,2,3$) are either irreducible
  or locally reducible and with non-vanishing discriminant. All edges
  in the cycle have multiplicity 1. The map $\trop$ is faithful on the
  cycle of $\Trop(I_{g\circ \psi, f_{1}})$. Thus, the cycle has the
  expected length and witnesses the prolongation of the cycle in
  $\Trop(g)$.  \smallskip

  Finally, assume the vertex $v_2$ does exist. Write
  $\widetilde{g\circ \psi}(z_{1},y):= I_{g\circ\psi,f_{1}}\cap
  \PSN[z_{1},y]$. The vertex $v_2$ is contained in the cycle of $\Trop(g)$,
  thus also in $\Trop(g\circ \psi)$ and in $\Trop(I_{g\circ \psi,
    f_{1}})$. Figure~\ref{fig:figuresAlgorithm} ensures that the cycle
  of $\Trop(g)$ lies on $\{a\geq Y\geq 0\}$.  In order to repair the
  embedding locally around $v_2$ we must work with two charts of
  $\Trop(I_{g\circ \psi, f_{1}})$: the one containing the cycle in $\Trop(g\circ
  \psi(x,y))$ and the one including the cycle in $\Trop(\widetilde{g\circ
    \psi})$. We claim that we can disregard the latter.
  Indeed, by Lemma~\ref{lm:repairV_1InDim3}, the
 vertex $v_0$ in $\Trop( \widetilde{g\circ \psi})$ is not
 traversed by the line $\{Y=a\}$. Suppose we perform a tropical
 modification of $\RR^2_{Z_{1},Y}$ along this line. Any choice
 of lifting function $h_2:=z_2-\AA_2 t^{-a}$ with $\val(\AA_2)=0$  
has  the same effect on the curve $\Trop(\widetilde{g\circ
   \psi})$:  it induces a tropical modification of the curve. The cycle  remains unchanged in
 $\Trop(I_{\widetilde{g\circ
     \psi},h_2})$, 
we see in  the right most picture in Figure~\ref{fig:DualExampleAlgorithm}.
 In conclusion, we can disregard the charts
 with coordinates $(Z_1,Y)$  when repairing the curve
 $\Trop(I_{g\circ \psi,f_{1}})$ locally around $v_2$. 

 Using the notation of Figure~\ref{fig:figuresAlgorithm}, we write
 $v_2=(-b+i a ,a)$, and $-b+ia<0$. We work with the line
 $L_2:=\{Y=a\}$, the vertex $v_2$ and the input curve $\Trop(g\circ
 \psi)$.  By Lemma~\ref{lm:repairV_1InDim3} we can find $f_2:=y+\AA_2
 t^{-a}$ and a map $\psi'\colon \PSN[x,y]\to \PSN[x,y]$ with
 $\psi'(x)=x+\beta t^{b-ia}$ and $\psi'(y)=y$, where $\AA_2,\beta\in
 \PSN$, $\val(\AA_2)=0$ and $\val(\beta)\geq 0$ such that the ideal
 $I_{g\circ \psi\circ \psi', f_2}$ satisfies the conditions of
 Lemma~\ref{lm:repairV_1InDim3}.

 If $\psi'=\Idd_{\PSN[x,y]}$, we conclude that the ideal
 $I_{g(x,y+\alpha), f_{1}, f_2}\subset \PSN[x,y,z_{1},z_2]$ induces a
 faithful linear re-embedding of the input curve in dimension 4 and
 its tropicalization has the expected cycle length.  On the contrary,
 assume $\psi'\neq \Idd_{\PSN[x,y]}$.  The proof of
 Lemma~\ref{lm:repairV_1InDim3} shows that $\psi'$ is constructed from
 a vertical modification along a line $\{X=B\}$ with $B\leq -b+ia$.
 We write $f_3:=x+\beta t^{b-ia}$ and consider the ideal
 $I:=I_{g(x,y+\alpha), f_{1}, f_2,f_3}\subset
 \PSN[x,y,z_{1},z_2,z_3]$. The variables $x,z_{1}$ and $z_3$ are
 related by the linear forms $z_1=f_{1}(x)$ and $z_3=x+\beta
 t^{b-ia}$. They are liftings of linear tropical modifications along
 two parallel hyperplanes: $\{X=0\}$ and $\{X=-b+ia\}$. We think of
 them as two vertical modifications that can be merged together. More
 precisely, the cycle of $\Trop(I)$ is contained in the cells
 $\sigma_{133}, \sigma_{132}, \sigma_{111}, \sigma_{312}$, which can
 be parametrized without using the coordinate $X$. Therefore, we can
 project the ideal to the variables $\{z_3,y,z_{1},z_2\}$ and obtain a
 new linear re-embedding of the curve in dimension 4 by the ideal
 $\tilde{I}:=I_{g(z_3-\beta t^{b-ia},y+\alpha), {f}_{1}(z_3-\beta
   t^{b-ia}),
   f_2}
 \subset \PSN[z_3,y,z_{1},z_2]$. Example~\ref{ex:Dim4Example} shows an
 instance of such projection and the resulting linear re-embedding of
 $g$.

The corresponding tropical curve in $\RR^4$ contains a cycle which is
isometric to the cycle of $\Trop(I)\subset \RR^5$, and all its
vertices and edges have multiplicity 1 by
Proposition~\ref{pr:projectionsJ}. Therefore, it prolongs the cycle in
$\Trop(g)$ and it has the expected cycle length. This concludes our
proof.

\end{proof}

\begin{algorithm}[htb]
  \KwIn{A polynomial $g(x_1,x_2)\in \PSN[x_1,x_2]$ defining a plane elliptic
    cubic with bad reduction, and $\Trop(g)$ contains a cycle of length $\ell<-\val{j(g)}$}
  \KwOut{A linear re-embedding $I_{g\circ \psi,f_3,f_4}\subset \PSN[x_1,x_2,x_3,x_4]$ of the curve defined
    by $g$ where $\psi$ is an affine map defined on $\PSN[x_1,x_2]$ and 
    the cycle of $\Trop(I_{g\circ \psi,f_3,f_4})$ has length $-\val(j(g))$.}  
  \smallskip  
$I \leftarrow \langle g \rangle\subset \PSN[x_1,x_2]$; \hspace{1ex} $J \leftarrow 0\subset\PSN[x_1,x_2]$ ; \hspace{1ex} $f_3\leftarrow x_1$\;
 $\psi\leftarrow \Idd$ on $\PSN[x_1,x_2]$;  \hspace{1ex} 
$\ell \leftarrow $ length of the cycle in $\Trop(I)$\;
	\While{there exists a locally reducible vertex $v$ in the cycle of $\Trop(g)$  with vanishing local discriminant  and a vertical, horizontal or skew line $L$ through $v$ containing the cycle of $\Trop(g)$ on the visible side of $L$ (as in Lemma~\ref{lem:atMostTwo})}{
 $f \leftarrow $ lifting of the linear tropical modification of $\RR^2$ along $L$ that repairs $v$ as in Lemma~\ref{lm:VanishingDecontractsVertex}\; 
 $\psi\leftarrow$ affine transformation of $\PSN[x_1,x_2]$ induced by $f$ that keeps the cycle of $\Trop(g)$ visible on $\Trop(g\circ \psi)$\;
$g\leftarrow g\circ \psi$; \quad
$I\leftarrow \langle g \rangle \subset \PSN[x_1,x_2]$.
}
\If{$\Trop(g)$ has a locally reducible vertex $v_0=(A_1,A_2)$ with vanishing discriminant and a vertical line $L_1=\{X_1=A_1\}$ through it}
{
$f_3 \leftarrow $lifting of the  tropical modification of $\RR^2$ along $L_1$ adapted to $g$\;
$\tilde{g}\leftarrow (I + \langle x_3-f_3\rangle) \cap \PSN[x_2,x_3]$\;
\If{$\tilde{g}$ contains a locally reducible vertex $\ww$ with vanishing discriminant in $(L_1^+)^{\circ}:=\{X_3<A_1\}$
}
{Construct an affine map $\psi$ on $\CC[x_3,x_2]$ with $\psi(x_3)=x_3$ and a polynomial $f_{1r}(x_3)$ adapted to $\tilde{g}$ using Lemma~\ref{lm:repairV_1InDim3}. We can construct it from a slight variant of the first subroutine using only horizontal and vertical lines.\\
  $f_3\leftarrow f_{1r}(f_3(x_1))$; \hspace{1ex} $g\leftarrow g\circ \psi$\;}}
  $J \leftarrow \langle x_3-f_3\rangle \subset \PSN[x_1,x_2,x_3]$;\hspace{1ex} $I\leftarrow \langle g \rangle +J \subset \PSN[x_1,x_2,x_3]$.\\
\If{$\Trop(g)$ contains a locally reducible vertex $v_2=(B_1,B_2)$ with vanishing discriminant and a horizontal line $L_2=\{X_2=B_2\}$ through it such that the cycle of $\Trop(g)$ is not on the visible side of $L_2$}
{$f_4\leftarrow $lifting of the  tropical modification of $\RR^2$ along $L_2$ adapted to $g$\;
$J\leftarrow J + \langle x_4-f_4\rangle $; $I\leftarrow \langle g \rangle +J \subset \PSN[x_1,x_2,x_3,x_4]$;
$\tilde{g}\leftarrow I \cap \PSN[x_1,x_4]$\; 
\If{$\tilde{g}$ contains a locally reducible vertex $\ww$ with vanishing discriminant in $(L_2^+)^{\circ}:=\{X_4<B_2\}$
}
{Construct an affine map $\psi$ on $\CC[x_1,x_4]$ with $\psi(x_4)=x_4$ and a polynomial $h_{2r}(x_4)$ adapted to $\tilde{g}$ using Lemma~\ref{lm:repairV_1InDim3}.\\
  $f_4\leftarrow h_{2r}(f_4(x_2))$; \hspace{1ex} $g\leftarrow g\circ \psi$;\hspace{1ex}
$f_3 \leftarrow f_3(x_1-\psi(0))$\;
$J \leftarrow \langle x_3-f_3, x_4-f_4\rangle \subset \PSN[x_1,x_2,x_3,x_4]$;\hspace{1ex} $I\leftarrow \langle g \rangle +J \subset \PSN[x_1,x_2,x_3,x_4]$.\\}
}\Return $I$.
  \caption{Repairing the cycle of a  tropical plane elliptic cubic using
   linear tropical modifications and special linear re-embeddings.}    \label{alg:repairElliptic}
\end{algorithm}

The following two examples illustrate the two key steps involved in the proof of Theorem~\ref{thm:repairEll}:
\begin{example}~\label{ex:twoStepsExample} 
Consider the plane elliptic cubic with defining equation:
\[
g=-t^3\,x^3+(t^4+t^5)x^2y+(-t^5+t^6)xy^2+t^3\,y^3+(t^2-t^3)x^2+4xy+(2t^2+3t^3)y^2+2x+(2+2t)y+(1+t).
\] 
Its $j$-invariant has valuation -8. The tropical curve is depicted on
the right of Figure~\ref{fig:TwoStepPart1}. It contains a cycle of
length 6 that needs to be prolonged. We do so in two steps.  The
bottom picture shows the Newton subdivision of the input curve and the
output linear re-embedding on each iteration.

In the first step, we modify the plane $\RR^2$ along the vertical line
$X=0$, corresponding to the tropical function $\max\{X,0\}$. The
vertex $(0,0)$ has valency four and lies on this vertical line. Its
discriminant equals $\Delta_{(0,0)}=c_{1,1}c_{0,0}-c_{1,0}c_{0,1}$ and
it vanishes at $\init_{(0,0)}(g)$. By choosing the special lifting
$f_1=x+1/2$ and the curve $g_1(z,y)=g(z-1/2,y)$, we prolong our cycle
by decontracting a bounded edge in $\sigmaint_3$, as we see in the
center of Figure~\ref{fig:TwoStepPart1}.  Viewed in the projection
$\Trop(g_1)= \pi_{YZ}(\Trop(I_{g,x+1/2}))$, the new cycle has length
seven, so it is still too short.

\begin{figure}[htb]
   \begin{minipage}[l]{0.35\linewidth}
     \includegraphics[scale=0.2]{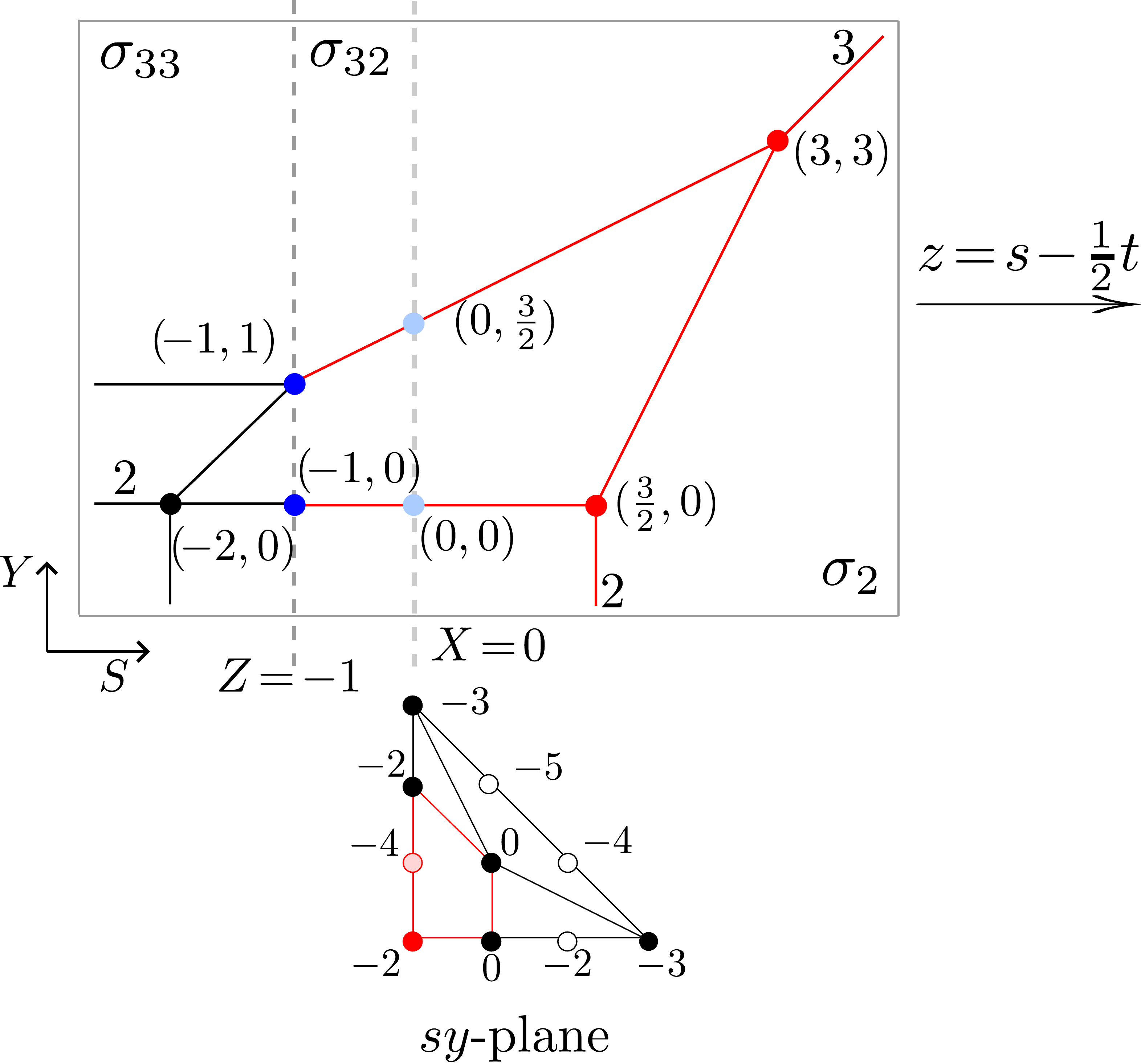}
   \end{minipage}
   \begin{minipage}[r]{0.35\linewidth}
\hspace{0.5ex}
   \includegraphics[scale=0.2]{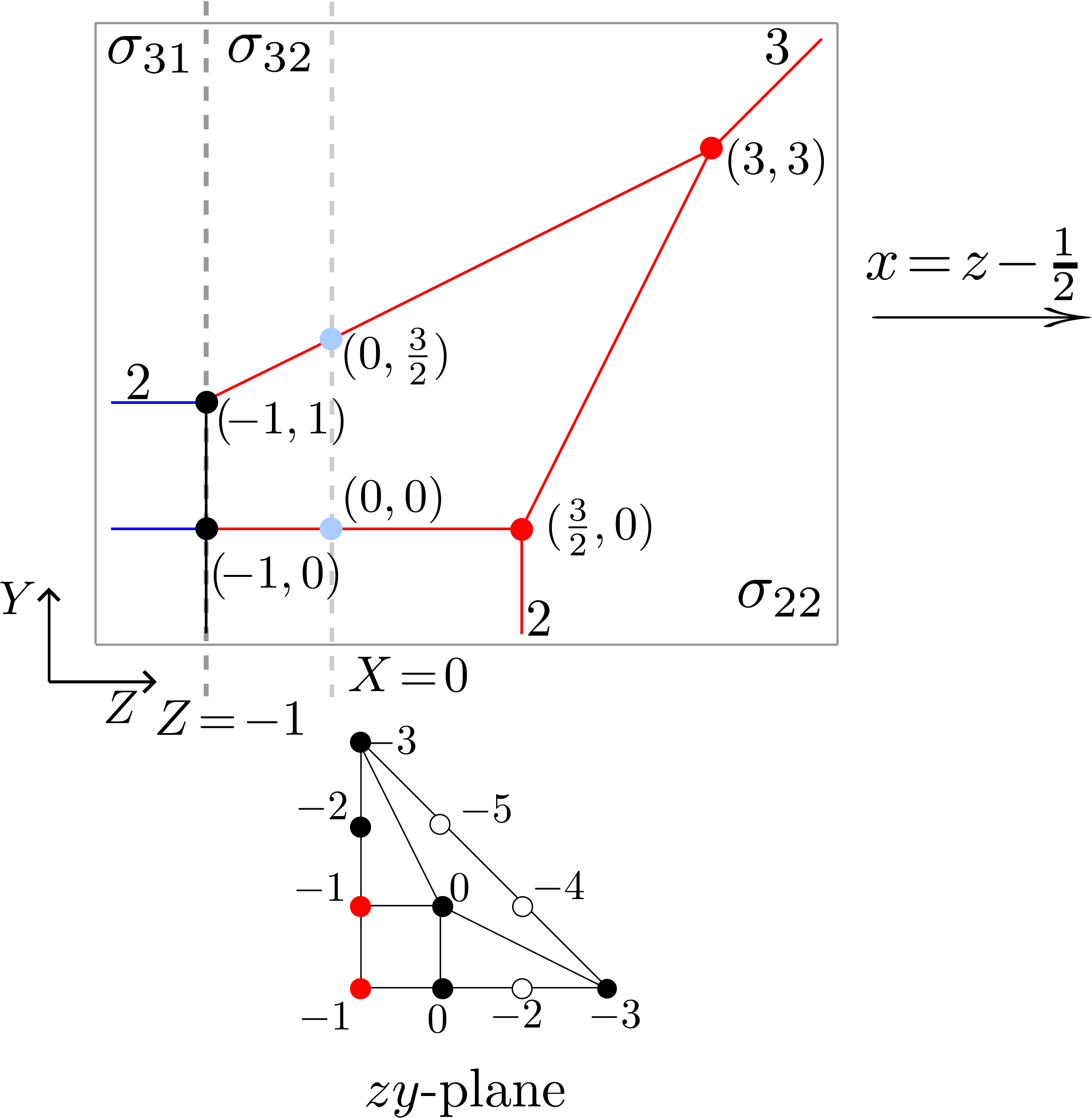}
 \end{minipage}
   \begin{minipage}[c]{0.28\linewidth}
     \includegraphics[scale=0.2]{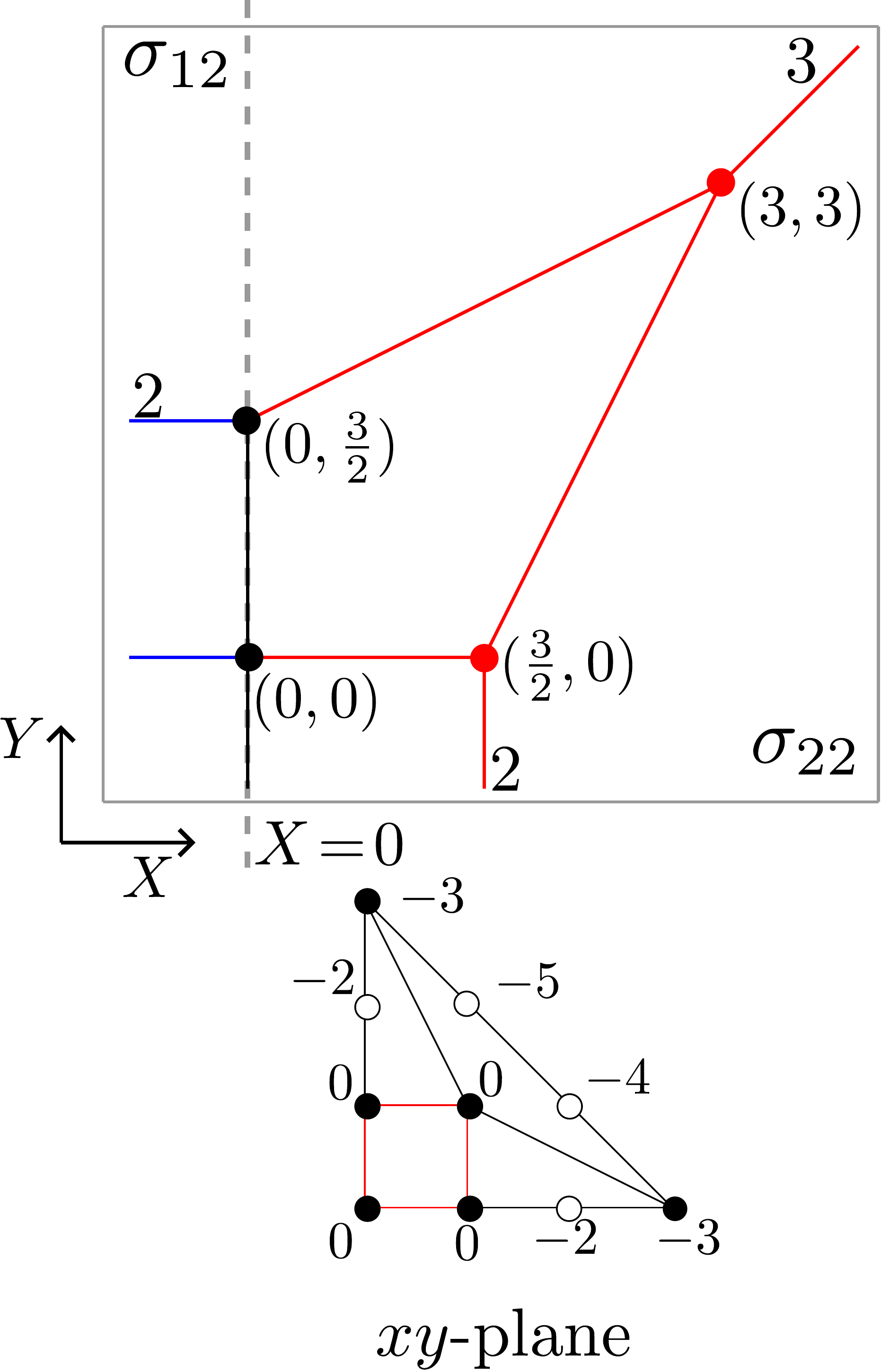}
 \end{minipage}
  \caption{Two iterations of the repairing algorithm for a plane elliptic cubic. The bottom row depicts the dual subdivisions, induced by
    the indicated heights. The heights of the red points change by the two
    linear tropical modifications.}
  \label{fig:TwoStepPart1}
\end{figure}

For our second iteration, we start with the tropical curve
$\Trop(g_1)$ in the $ZY$-plane and the locally reducible vertex
$(-1,0)$, whose discriminant vanishes on $\init_{(-1,0)}(g_1)$. The
$YZ$-plane is given as the union of the cells $\sigma_2\cup
\sigma_3$. We modify it along the vertical line $Z=-1$ corresponding
to the tropical polynomial $\max\{Z,-1\}$, depicted on the center of
Figure~\ref{fig:TwoStepPart1}. We choose the lifting $f_2=z+t/2$, a
new variable $s$ and tropicalize the ideal $I_{g_1,f_2}=\langle
g_1,s-f_2\rangle$. Write $g_2(s,y)=g_1(s-t/2,y)=I_{g,f_1,f_2}\cap
\PS[s,y]$. When projected to the $SY$-plane we observe the curve
$\Trop(g_2)$, living in the union of the cells $\sigma_{32},
\sigma_{33}$ and $\sigma_{22}$ in the modified plane, as in the left
of Figure~\ref{fig:TwoStepPart1}. The re-embedded and projected curve
$\Trop(g_1)$ from our first iteration lies in the cells $\sigma_{31},
\sigma_{32}$ and $\sigma_{22}$, as in the center of
Figure~\ref{fig:TwoStepPart1}.

We combine the two modifications in one step and make the
corresponding affine coordinate change to give the desired special
linear re-embedding of $C$ in dimension 2, namely $\langle
g(s-(1/2+t/2),y)\rangle$. The cycle in the new tropical curve has
length 8, as desired.
\end{example}

The next example illustrates the behavior of
Algorithm~\ref{alg:repairElliptic} in the presence of two locally
reducible vertices with vanishing discriminants on the top left and
bottom right of the cycle of $\Trop(g)$, and when more than two linear
tropical modifications are required to achieve the desired cycle
length.
\begin{example}\label{ex:Dim4Example}
 Consider the plane elliptic cubic with bad reduction defined by the equation
\begin{align*}
  g(x,y)=&
  x^3+(1-9t^2)x^2y+2t^4xy^2+t^{20}y^3+(1-24t^9-t^{40})x^2+(1+5t-16t^9+144t^{11})xy\\ & +8t^{67}y^2+(1-16t^9+t^{15}+192t^{18})x+(2t^4+64t^{18}-576t^{20})y+(1-8t^9+64t^{18}-8t^{24}).
\end{align*}
Its $j$-invariant has valuation -15.  The corresponding tropical curve
has a cycle of lattice length $12$ and is depicted at the curve of
Figure~\ref{fig:AlgorithmExample}. It contains two locally reducible
vertices with vanishing discriminant, namely $(0,0)$ and $(-4,4)$. As
in the proof of Theorem~\ref{thm:repairEll}, we repair the embeddings
by treating these two vertices independently. We will need three
linear tropical modifications, namely along the lines $X=0$, $Y=4$ and
$X=-4$. Their 2-dimensional charts are depicted in
Figure~\ref{fig:AlgorithmExample}. We choose the special liftings
induced by the functions $f_1=x+1$, $f_2=y+t^{-4}/2$ and
$f_3=x-2t^4$. Figure~\ref{fig:DualExampleAlgorithm} shows the Newton
subdivision of five planar projections. Notice that $(1,1)$ is not a
vertex of the leftmost subdivision. Even though the vertex $(-5,3)$ in
the projection of the curve to the $(Z_3,Z_2)$-coordinates is locally
reducible, its local discriminant does not vanish, so its multiplicity
in $\sigma_{133}^{\circ}$ equals 1. By
Proposition~\ref{pr:projectionsJ}, its multiplicity in
$\Trop(I_{g,f_1,f_2,f_3})$ is also 1.

We reconstruct the tropical curve $\Trop(I_{g,f_1,f_2,f_3})$ by
looking at the tropical plane curves $\Trop(g)$,
$\Trop(\tilde{g}(z_1,y))$ and $\Trop(\tilde{g}(x,z_2))$ and using
Lemma~\ref{lem:ModifViaProjections}. The cycle of this curve in
$\RR^5$ is on the visible side of the hyperplane $X=-4$. Thus, we can
project the curve and the ambient plane to the space corresponding to
the variables $y,z_1,z_2,z_3$, and still repair the embedding of the
original plane elliptic curve.  The resulting ideal is $\langle
g(z_3-2t^4, y), z_2-y-t^{-4}/2, z_3-z_1+1+2t^4\rangle$. Its
tropicalization in $\RR^4$ is shown in
Figure~\ref{fig:MakeDualGraphAlgorithm}. Its cycle reflects the
$j$-invariant of the input curve.
\end{example}

\begin{figure}[tb]
  \centering
 \includegraphics[scale=0.175]{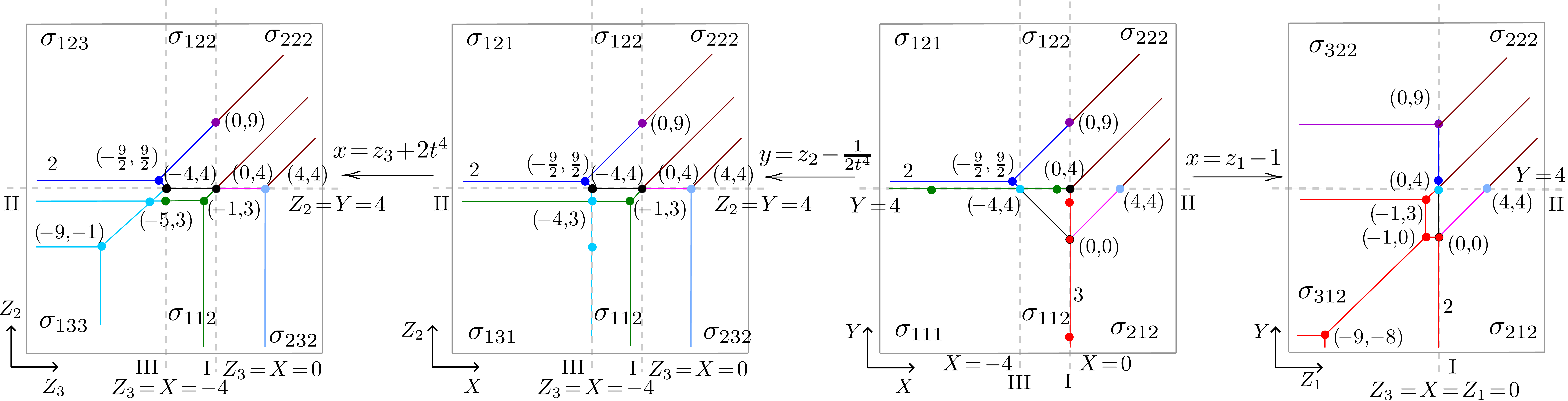}
   \caption{Three iterations of
     Algorithm~\ref{alg:repairElliptic} repair the middle tropical
     elliptic cubic. We draw the relevant four projections in
     $\Trop(I_{g,f_1,f_2,f_3})$, indicating the corresponding images of
     all the vertices of $\Trop(I_{g,f_1,f_2,f_3}\cap \PSN[y,z_1,z_2,z_3])$.  
   }
  \label{fig:AlgorithmExample}
\end{figure}
\begin{figure}[tb]
  \centering
  \includegraphics[scale=0.3]{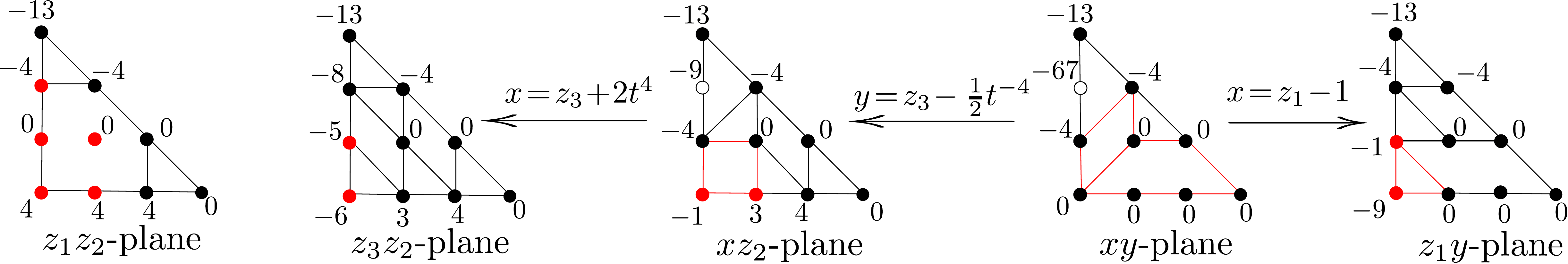}
  \caption{Dual subdivisions (and heights) to the plane tropical
    curves obtained from five coordinate projections of
    $I_{g,f_1,f_2,f_3}$.}
  \label{fig:DualExampleAlgorithm}
\end{figure}
\begin{figure}[tb]
  \centering
  \includegraphics[scale=0.1]{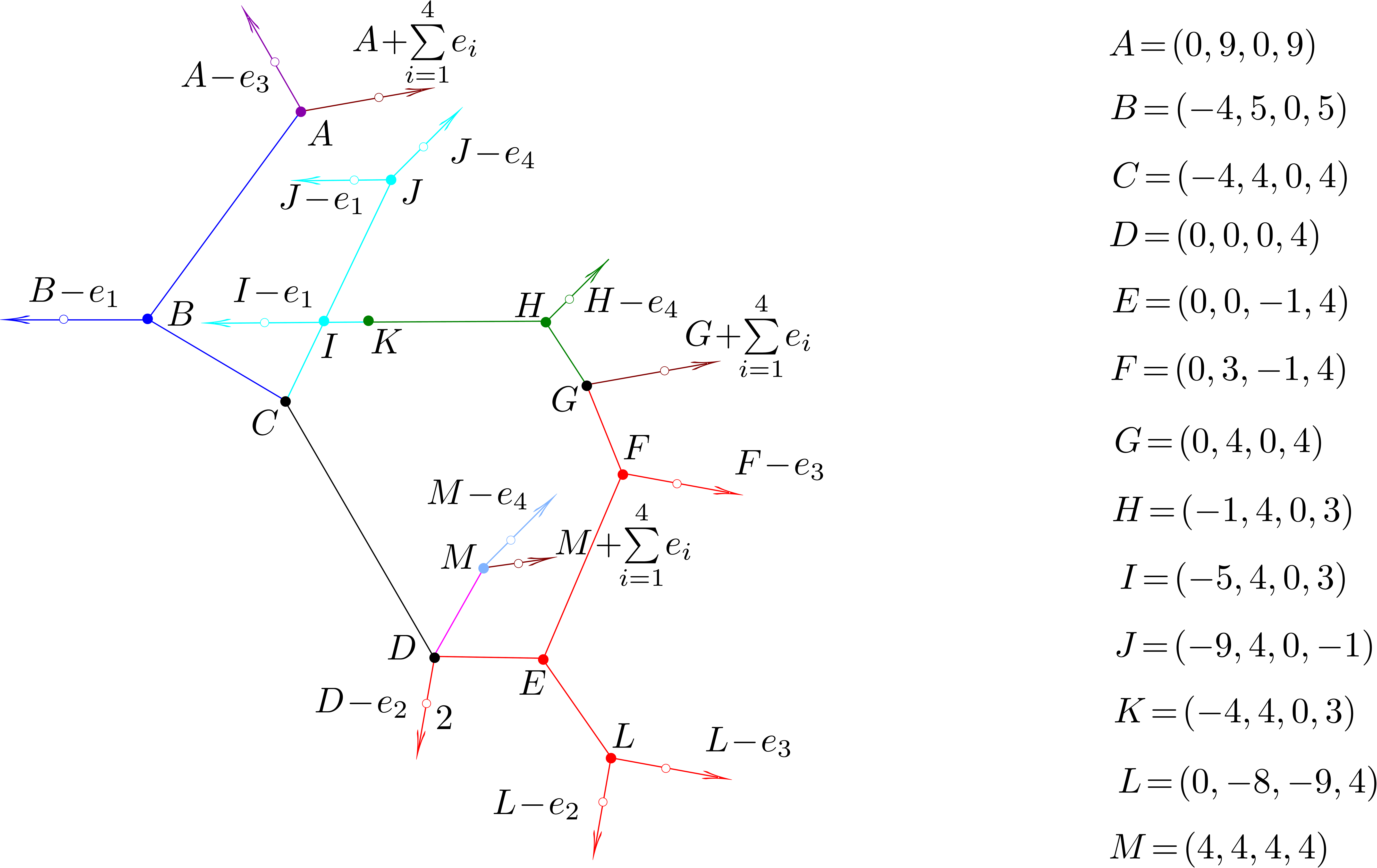}
  \caption{Tropical elliptic cubic repaired in $\RR^4$ using
    Algorithm~\ref{alg:repairElliptic}. The coordinates are labeled
    $(Z_3,Y,Z_1,Z_2)$ and the color coding agrees with
    Figure~\ref{fig:AlgorithmExample}.}
  \label{fig:MakeDualGraphAlgorithm}
\end{figure}

\begin{remark} \label{rem:UnfoldDoubleEdgeElliptic} Assume $C$ is a
  plane elliptic cubic $C$ with bad reduction defined by $g$. Suppose
  that its tropicalization $\Trop(g)$ contains no cycle but one of its
  bounded edges $e$ has tropical multiplicity $m_{\Trop}(e)>1$ and
  non-vanishing discriminant. Then the combinatorics of the Newton
  polygon of $g$ ensure that
  $m_{\Trop}(e)=2$. Theorem~\ref{thm:fatEdge} guarantees that we can
  unfold this double edge and produce a cycle by a special linear
  re-embedding.

  For simplicity, assume that the double edge $e$ of $\Trop(g)$ is
  vertical. Then the push-forward formula for
  multiplicities~\cite[Theorem 8.4]{BPR11} applied to the projection
  $\pi_{ZY}$ ensures that the cycle in the re-embedded curve contains
  only multiplicity one edges, and so does its projection
  $\Trop(\tilde{g})$. If, in addition, the cycle in $\Trop(\tilde{g})$
  is trivalent, we conclude by~\cite[Theorems 6.23 and 6.24]{BPR11}
  that this cycle is isometric to the minimal skeleton of the complete
  analytic elliptic curve $\widehat{C}^{\an}$ under the
  tropicalization
  map~\eqref{eq:tropMap}. Example~\ref{ex:unfoldingDoubleEdge}
  illustrates this phenomenon.
\end{remark}

We end this section with an easy characterization of tropical faithfulness on cycles of tropical plane elliptic cubics whose $j$-invariant has negative valuation. As expected,  this result also follows from~\cite[Section 6]{BPR11}, but our approach makes it easier to verify in concrete examples.
\begin{corollary} Let $C$ be a plane elliptic cubic with defining equation $g$ whose $j$-invariant has negative valuation. Then its tropicalization faithfully represents the minimal skeleton of the complete curve $\widehat{C}^{\an}$ if and only if it contains a cycle  and when restricted to this cycle the following hold:
  \begin{enumerate}[(i)]
  \item all edges have multiplicity one,
\item all vertices are either locally irreducible or they are locally reducible and with non-vanishing discriminant.
  \end{enumerate}
Furthermore, the cycle in $\Trop(g)$ has length $-\val(j(g))$.
\end{corollary}
\begin{proof}
  Since the Newton polygon of $g$ equals 3 times the unit two-simplex,
  and $\Trop(g)$ contains a cycle, we know that the only locally
  reducible vertices in its cycle are dual to those in
  Figure~\ref{fig:shapesandfeeding} (up to reflections). By
  Corollary~\ref{cor:highMultAndDiscrim}, any such locally reducible
  vertex with $m_{\Trop}(v)=2$ must have a vanishing
  discriminant. The result follows from Theorem~\ref{thm:repairEll}
  and Lemma~\ref{lem:nonHomeo}.
\end{proof}

\subsection{Repairing smooth plane elliptic cubics with $A$-discriminants}\label{sec:repa-ellipt-cubics}
In this section, we show how we can repair bad embeddings of plane
elliptic cubics in an elementary way, without relying on their
Berkovich skeleta. Our goal is to reprove Theorem~\ref{thm:repairEll}
by combinatorial means, and in particular to show how we can use
linear re-embeddings to ensure that the length of the cycle of the
re-embedded tropical curve equals the negative valuation of the
$j$-invariant (see~\cite{KMM07}). Our main tool will be the theory of
$A$-discriminants, introduced in Section~\ref{sec:discriminants}. The
results presented here hold for arbitrary planar configurations and
might be of independent interest.

As a motivation, we start with a configuration in the dilated
2-simplex corresponding to a plane elliptic cubic. We keep the
notation from Section~\ref{ssec:known-results-about}, and fix a
defining equation $g$ for $C$ with prescribed valuations of
coefficients inducing a tropical plane elliptic curve $\Trop(g)$
containing a cycle. We assume the cycle is shorter than expected. As
we saw in Section~\ref{ssec:known-results-about}, this is caused by
the cancellation of an initial form of the discriminant polynomial
$\Delta$, when evaluated at the coefficients of $g$.
Algorithm~\ref{alg:repairElliptic} uses the vanishing of a local
discriminant to repair the
cycle. Corollary~\ref{cor:initialDiscriminant} explains the connection
between these two facts through a factorization formula to describe
initial forms of $A$-discriminants in terms of discriminants of an
induced marked coherent subdivision of the convex hull of $A$, i.e.\
the Newton subdivision of $g$.

Let us fix a full-dimensional point configuration $A$ in $\ZZ^k$, and
let $Q$ be the convex hull of $A$. Given any subset $B$ of $A$ we
denote by $\ZZ\tightcdot B$, respectively $\RR\tightcdot B$, the
linear span of $B$ over $\ZZ$, respectively over $\RR$. In addition,
we define the lattice index
\begin{equation}
i(B,A)=[ \ZZ\tightcdot A\cap \RR\tightcdot B:\ZZ\tightcdot B].\label{eq:latticeIndex}
\end{equation}
We write $i(B)=i(B,\ZZ^k)$.  We work with both 
discriminant cycles and principal determinants of the configuration $A$ as defined
in~\cite[Chapter 10]{GKZ} and discussed further in~\cite{CCDDRS11}. Rather than defining the principal
determinant $E_A\in \ZZ[c_a: a\in A]$ and the corresponding cycle
$\tilde{E}_A:=E_A^{i(A)}$ for a  configuration $A\subset \ZZ^k$ on its own,
we choose to present it in terms of the prime factorization formula
from~\cite[Theorem 1.2, $\S$~10.1.B]{GKZ} adapted to the case when the
condition $i(A)=1$ is not required (as in Esterov~\cite[Proposition 3.10]{Est10}): 
\begin{equation}\label{eq:ProdFormulaGKZ}
  \tilde{E}_A:=   \pm\, \Delta_A^{i(A)}\prod_{F\prec Q} {\Delta}_{F\cap A}^{i(F\cap A,\ZZ^2)u(F\cap A,A)}.
\end{equation}
Here, $F\prec Q$ denotes a marked proper face of the polytope $Q$ and
$g$ represents a bivariate Laurent polynomial supported on $A$. The
lattice indices $i(A)$ and $i(F\cap A,A)$ are defined as
in~\eqref{eq:latticeIndex}. The exponents $u(F\cap A,A)$ refer to the
\emph{generalized subdiagram volume} associated to $F$ and $A$ viewed
in the ambient lattice $\ZZ^k$. They are computed as follows. We fix
the linear projection $\pi\colon \RR\tightcdot A\to \RR\tightcdot
A/\RR\tightcdot F$ and let $\Omega$ be the normalized volume form on
$\RR\tightcdot A/\RR\tightcdot F$. The normalization is performed with
respect to the lattice $\pi(\ZZ^k)$, so that the fundamental domain $
\RR\tightcdot A/(\RR\tightcdot F +\ZZ^k)$ has volume
$(\dim(\RR\tightcdot A)-\dim(\RR\tightcdot F))!$. We set
\begin{equation}
  \label{eq:subdiagramVol}
  u(F\cap A,A):=\Omega\big(\conv(\pi(A))\smallsetminus \conv(\pi(A\smallsetminus F\cap A))\big).
\end{equation}

\begin{remark}\label{rem:GKZvsEsterov}
  The positive integers $u(F\cap A, A)$ are denoted by $c^{F\cap A,
    A}$ in~\cite[$\S 2.5$]{Est10}. When $i(A)=1$, the definition of
  $u(F\cap A, A)$ from~\eqref{eq:subdiagramVol} agrees with the
  subdiagram volume form $u(\ZZ\tightcdot A/(F\cap A))$ defined
  in~\cite[Chapter 5, Theorem~2.8]{GKZ}. For arbitrary $A$ we recover
  the latter by renormalizing the volume form in $\RR\tightcdot
  A/\RR\tightcdot F$ with respect to the lattice $(\ZZ\tightcdot
  A)/(\ZZ\tightcdot A\cap \RR\tightcdot F)$ and replacing $\Omega$
  in~\eqref{eq:subdiagramVol} by this new volume form.
\end{remark}

The subdiagram volume forms $u(F\cap A, A)$
from~\eqref{eq:subdiagramVol} and $u(\ZZ\tightcdot A/(F\cap A))$
from~\cite[Chapter 5]{GKZ} are related by the following identity over $\ZZ$.
\begin{lemma}\label{lm:SubdiagramVolIdentities} Let $A\subset \ZZ^k$
  be a full-dimensional point configuration and $F$ be a face of the
  polytope $Q=\conv(A)$. Then:
  \begin{equation*}
    \label{eq:VolIdentities}
    i(F\cap A, \ZZ^k)\, u(F\cap A, A) =i(A)\, i(F\cap A, \ZZ\tightcdot A) \, u(\ZZ\tightcdot A/ (F\cap A)).
  \end{equation*}
\end{lemma}
\begin{proof} Since all vertices of $Q$ lie in $A$, we know that
  $\RR\tightcdot F= \RR\tightcdot (F\cap A)$ for every face $F$ of
  $Q$.  The result follows from Remark~\ref{rem:GKZvsEsterov} and the
  identity \[ i(A)={i(F\cap A, \ZZ^k)} i(F\cap A, \ZZ\tightcdot
  A)^{-1} [\ZZ^k/(\RR\tightcdot F\cap \ZZ^k): \ZZ\tightcdot
  A/(\RR\tightcdot F\cap \ZZ\tightcdot A)].\qedhere\]
\end{proof}

From now on, we fix a full-dimensional planar configuration $A\subset
\ZZ^2$ with $m$ points.  We pick a tuple $\ww\in \RR^m$ giving a
height for every point in $A$. The tuple $\ww$ induces a marked
subdivision $\cP:=\{(Q_i, A_i): i\in I\}$ of the marked pair
$(Q,A)$. Here, $Q_i=\conv(A_i)$ for all $i\in I$ are the maximal cells
in the subdivision $\cP$.  We aim to compute the initial form of
$\Delta_A$ with respect to $\ww$.

We let $\mathcal{E}_i$ denote the set of edges of $Q_i$ for each $i\in
I$, and let $\mathcal{E}$ be the set of edges of $Q$. Let
$\mathcal{E}_i^{\outerEdge}$ be the set of edges of $Q_i$ that lie
entirely in the boundary of the polytope $Q$ and set
$\mathcal{E}_i^{\intEdge}:=\mathcal{E}_i\smallsetminus
\mathcal{E}_i^{\outerEdge}$ to be the complementary set of internal
edges in $\mathcal{E}_i$.  For $i\neq j$ we set
$\mathcal{E}^{\intEdge}_{i,j}=\mathcal{E}^{\intEdge}_{i} \cap
\mathcal{E}^{\intEdge}_{j}$.

The following result is reminiscent of the combinatorial formula
from~\cite[Theorem 4.1]{Stu94}. that computes initial forms of
resultants as products of powers of smaller resultants.
\begin{theorem}\label{thm:factorizationFormula} Let $A$ be a full-dimensional configuration in $\ZZ^2$. Then,
  \begin{equation*}
    \init_{\ww}(\Delta_A)
    \!=\!\lambda \,\mu\prod_{j\in I} \Delta_{A_j}^{[\ZZ\cdot\! A:\ZZ\cdot\! A_j]}\!\!\! 
    \prod_{\substack{j\in I\\e\in
        \mathcal{E}^{\outerEdge}_j}}\!\!\!\Delta_{e\cap A_j}^{i(e\cap
      A_j)
      (u(e\cap A_j,A_j)-u(\RR\cdot e\cap A, A))/i(A)}
    \!\!\!\prod_{\substack{j<l\\e\in
        \mathcal{E}^{\intEdge}_{j,l}\!}}\!\!\!\Delta_{e\cap A_j}^{i(e\cap
      A_j)
      (u(e\cap A_j, A_j)+u(e\cap A_l, A_l))/i(A)}.
\label{eq:factorization}
\end{equation*}
All exponents are nonnegative integers, $\mu$ is a Laurent monomial
and the constant $\lambda \in \ZZ$ and can be computed as 
\[
\lambda=\pm \prod_{j\in I}[\ZZ\tightcdot A:\ZZ\tightcdot
A_j]^{\vol_{\ZZ\cdot\! A}(Q_j)}\big(\prod_{\substack{e\in\mathcal{E}\\ \dim(e\cap Q_j)=1}} [\ZZ\tightcdot(e\cap A),
\ZZ\tightcdot(e\cap A_j)]^{u(e\cap A, A)\vol_{\ZZ^2\cap \RR 
    \cdot e}(e\cap Q_j)/i(A)}\big)^{-1},
\]
where the volume forms are normalized with respect to the indicated lattices.
\end{theorem}
Theorem~\ref{thm:factorizationFormula} is the main result in this
section. We shall derive it by means of the following characterization from~\cite[Chapter 10.1.E,
Theorem $12'$]{GKZ}
of initial forms of principal determinants of full-dimensional configurations $A\subset \ZZ^k$, adapted to the case when $i(A)=1$ is not
required. 
\begin{equation}
  \label{eq:initialE_A}
  \init_{\ww}(E_A)^{i(A)}=\pm \, \mu \prod_{j\in I}\;[\ZZ\tightcdot A: \ZZ\tightcdot A_j]^{\vol_{\ZZ^k}(Q_j)} E_{A_j}^{i(A_j)}.
\end{equation}
Here, $\mu$ is a Laurent monomial and the principal determinant
$E_{A_j}$ is evaluated at the restriction of $g$ to those monomials supported on $A_j$.
Combining~\eqref{eq:initialE_A} and the product formula~\eqref{eq:ProdFormulaGKZ} for each $A_j$ ($j\in I$) gives the identity
\begin{equation}
  \label{eq:StartingEqn}
  \begin{split}
    \init_{\ww}(\Delta_A)^{i(A)}\prod_{F\prec
      Q}\init_{\ww}(\Delta_{F\cap A})^{i(F\cap A, \ZZ^2)u(F\cap A,A)}=
    \pm \,\mu \prod_{j\in I} [\ZZ\tightcdot A: \ZZ\tightcdot A_j]^{\vol_{\ZZ^k}(Q_j)} \prod_{j\in I}\Delta_{A_j}^{i(A_j)}\\
    \prod_{j\in I}\prod_{F_j\prec Q_j} \Delta_{F_j\cap A_j}^{i(F_j\cap
      A_j, \ZZ^k)u(F_j\cap A_j,A_j)},
  \end{split}
\end{equation}
where the product on the right-hand side runs over all proper faces $F_j$ of $Q_j$. 

We shall prove Theorem~\ref{thm:factorizationFormula} by studying each
initial form $\init_{\ww}(\Delta_{F\cap A})$ on the left-hand side
of~\eqref{eq:StartingEqn} when $k=2$, one dimension at a time. By
definition, we know that $\Delta_{F}=c_F$ whenever $F$ is a vertex of
a polytope and $a_F$ is its corresponding coefficient. Thus, we can
incorporate all these Laurent monomials to $\mu$ and assume that $k=2$
and all faces $F$ and $F_j$ in~\eqref{eq:StartingEqn} are edges of $Q$
and $Q_j$, respectively. The following lemmas simplify the exposition.

\begin{lemma}\label{lm:valueEdges}
  Let $e
  $ be an edge of $Q$. Then
\[
\init_{\ww}(\Delta_{e\cap A})^{i(e\cap A,\ZZ^2)} = \pm \,\mu_e\!\!\!\!\! \prod_{\substack{j\in I\\\dim(e\cap Q_j)=1}}[\ZZ\tightcdot(e\cap A): \ZZ\tightcdot(e \cap A_j)]^{\vol_{\ZZ^2\cap \RR\cdot e}(e\cap Q_j)} \Delta_{ e\cap A_j}^{i(e \cap A_j , \ZZ^2)},
\]
where $\mu_e$ is a Laurent monomial with support contained in $e\cap A$.
\end{lemma}
\begin{proof}
  The result is an immediate consequence of~\eqref{eq:StartingEqn}
  where we replace the starting configuration $A$ by $e\cap A$ and the
  induced subdivision ${\cP}$ by $\cP_{e}:= \{(e,e\cap A_i): i\in I,
  \dim(e\cap Q_i)=1\}$.
\end{proof}

\begin{lemma}\label{lm:exponents}
  For every $j\in I$ and every $e\in \mathcal{E}_j$, the quantities
  $i(e\cap A_j,\ZZ^2)u(e\cap A_j,A_j)$ and $i(A_j)$ are integer
  multiples of $i(A)$. If $e\in \mathcal{E}_j^{\outerEdge}$, then
  $i(e\cap A_j,\ZZ^2)u(\RR\tightcdot e\cap A,A)$ also lies in
  $i(A)\tightcdot \ZZ$.
\end{lemma} 
\begin{proof}
  A simple calculation shows that $i(A_j)=i(A) [\ZZ\tightcdot
  A:\ZZ\tightcdot A_j]$ for all $j\in I$.  Our claim follows from
  Lemma~\ref{lm:SubdiagramVolIdentities}. Indeed, if $e\in
  \mathcal{E}_j$, then
\[
i(e\cap A_j,\ZZ^2)u(e\cap A_j,A_j)=i(e\cap A_j,A_j) \,i(A_j)\, u(\ZZ\tightcdot A_j/(A_j\cap e)).
\]
 Thus, $i(A_j)$ divides  $i(e\cap A_j,\ZZ^2)u(e\cap A_j,A_j)$ over $\ZZ$, and
 hence so
 does $i(A)$. 

 Finally, if $e\in \mathcal{E}_j^{\outerEdge}$, then $e=F\cap Q_j$ for
 a unique edge $F$ of $Q$. Notice that $F\cap A_j=e\cap
 A_j$. Lemma~\ref{lm:SubdiagramVolIdentities} implies that $i(A)$
 divides $i(F\cap A,\ZZ^2)u(F\cap A,A)$ over $\ZZ$. Since $i(F\cap
 A_j,\ZZ^2)=i(F\cap A, \ZZ^2)[\ZZ\tightcdot(F\cap A):
 \ZZ\tightcdot(F\cap A_j)]$, we conclude that $i(e\cap
 A_j,\ZZ^2)u(\RR\tightcdot e\cap A,A) \in i(A)\tightcdot \ZZ$.
\end{proof}

\begin{proof}[Proof of Theorem~\ref{thm:factorizationFormula}]
  By combining Lemma~\ref{lm:valueEdges} with
  expression~\eqref{eq:StartingEqn} we conclude that
  \begin{equation}
    \label{eq:StartingEqn2}
    \begin{split}
      \init_{\ww}(\Delta_A)^{i(A)}\!\!\!\!\!\!  \prod_{\substack{e\in
          \mathcal{E},j\in I\\ \dim(Q_j\cap e)=1}}
      \!\!\![\ZZ\tightcdot(e\cap A): \ZZ\tightcdot(e \cap
      A_j)]^{\vol_{\ZZ^2\cap \RR\tightcdot e}(Q_j\cap F)u(e\cap A,A)}
      \Delta_{e\cap A_j}^{i(e\cap A_j, \ZZ^2) u(e\cap A, A)}\\
      = \pm \,\mu'\,\prod_{j\in I}[\ZZ\tightcdot A: \ZZ\tightcdot
      A_j]^{\vol_{\ZZ^2}(Q_j)} \prod_{j\in I}\Delta_{A_j}^{i(A_j)}
      \prod_{j\in I} \prod_{e\in \mathcal{E}_j} \Delta_{e\cap
        A_j}^{i(e \cap A_j, \ZZ^2) u(e\cap A_j,A_j)},
    \end{split}
  \end{equation}
  where ${\mu'}=\mu(\prod_{e\in \mathcal{E}} \mu_e^{u(e\cap A,
    A)})^{-1}$ is a Laurent monomial.  All polynomials in
  \eqref{eq:StartingEqn2} are irreducible over $\ZZ$. Rearranging the
  terms, we obtain the desired factorization of
  $\init_{\ww}(\Delta_A)$. This follows by analyzing the edges $e\in
  \mathcal{E}_j$. If $e\in \mathcal{E}_j^{\intEdge}$, then $e$ lies in
  the boundary of exactly one other polytope, say $Q_l$, so $e\in
  \mathcal{E}^{\intEdge}_{j,l}$ and $e\cap A_j=e\cap A_l$. The
  polynomial $\Delta_{e\cap A_j}$ appears only on the right-hand side
  of~\eqref{eq:StartingEqn2} and its exponent equals $i(e\cap
  A_j,\ZZ^2)\big(u(e\cap A_j,A_j)+u(e\cap A_l,A_l)\big)$.  On the
  contrary, if $e\in \mathcal{E}_j^{\outerEdge}$, then $e\notin
  \mathcal{E}_l$ for any $l\neq j$. Rearranging terms
  in~\eqref{eq:StartingEqn2}, we conclude that the exponent of
  $\Delta_{e\cap A_j}$ in the factorization of
  $\init_{\ww}(\Delta_A)^{i(A)}$ equals $i(e\cap
  A_j,\ZZ^2)\big(u(e\cap A_j,A_j)-u(e\cap A,A)\big)$. Since
  $A_j\subset A$, this quantity is non-negative by
  \eqref{eq:subdiagramVol}.

Lemma~\ref{lm:exponents} ensures that all exponents
in~\eqref{eq:StartingEqn2} are non-negative integers divisible by
$i(A)$. Since all discriminants in~\eqref{eq:StartingEqn2} have
content one, we know that the rational number 
\[
 \prod_{j\in I}[\ZZ\tightcdot A: \ZZ\tightcdot
      A_j]^{\vol_{\ZZ^2}(Q_j)} \big( \prod_{\substack{e\in \mathcal{E},j\in I\\ \dim(Q_j\cap e)=1}}
      \!\!\![\ZZ\tightcdot(e\cap A): \ZZ\tightcdot(e \cap
      A_j)]^{\vol_{\ZZ^2\cap \RR\tightcdot e}(Q_j\cap F)u(e\cap A,A)}\big)^{-1}
\]
is in fact an integer number and its $i(A)$-th root also lies in
$\ZZ$. The latter is precisely the quantity $\lambda$ in the
statement.This concludes our proof.
\end{proof}

Theorem~\ref{thm:factorizationFormula} is particularly enlightening
when $A$ defines a cubic equation $g$, as in the case of plane elliptic
cubics. As usual, we write $\{c_a: a\in A\}$ for the coefficients of $g$.
 \begin{corollary}\label{cor:initialDiscriminant}
   Let $g$ be a cubic bivariate polynomial and let $A$ be the
   configuration of points that supports $g$. Let $\ww$ be the weight
   vector corresponding to the valuation of all coefficient of
   $g$. Let $\cP=\{(Q_j,A_j): j\in I\}$ be the maximal cells in the
   Newton subdivision of $g$. Assume $(1,1)$ is a vertex of this
   subdivision. Then
\begin{equation}
  \init_{\ww}(\Delta_A)=\lambda \,\underline{c_{a}}^{\alpha}\prod_{j\in I} \Delta_{A_j}^{i(A_j)},\label{eq:7}
\end{equation}
where $\lambda\in \ZZ$, $\underline{c_a}^{\alpha}$ is a Laurent
monomial, and the product runs over all non-defective $A_j$'s. The
discriminant $\Delta_{A_j}$ is evaluated on the restriction of $g$ to
those monomials supported on $A_j$.
 \end{corollary}
 \begin{proof}
   The result follows from Theorem~\ref{thm:factorizationFormula}
   after the following observations. First, note that any pyramid is a
   defective configuration. All internal edges in the Newton
   subdivision of $g$ have lattice length one, hence $\Delta_{e\cap
     A_j}=1$ for all $e\in \mathcal{E}_j^{\intEdge}$.  Finally, assume
   $e\in \mathcal{E}_j^{\outerEdge}$ is not a pyramid. An easy
   calculation shows that $u(e\cap A_j,A_j)=u(\RR\tightcdot e\cap
   A,A)=1$ because $(1,1)$ is an internal edge of $g$.
 \end{proof}
 As a consequence, we give an alternative elementary proof of
 Theorem~\ref{thm:repairEll}. Indeed, write $\Delta$ for the
 discriminant of the elliptic cubic equation evaluated at the
 coefficients of $g$. We view $\Delta$ as an element of $\PS$. By
 Corollary~\ref{cor:initialDiscriminant}, the initial form of $\Delta$
 factors as the product~\eqref{eq:7}. We write $g_{|_{A_j}}$ for the
 restriction of $g$ to those monomials supported on $A_j$.

 Notice that our first proof of Theorem~\ref{thm:repairEll} relied on
 an iterative usage Theorem~\ref{thm:redVertex}, where both
 implications are needed. Going back to Theorem~\ref{thm:redVertex},
 we see that one implication, namely the one proved in
 Lemma~\ref{lem:noDecontractionIfLocallySmooth}, relies on Berkovich
 theory. The other, Lemma~\ref{lm:VanishingDecontractsVertex}, is
 purely combinatorial, but requires a vanishing local discriminant as
 input, in order to repair a problematic vertex. In our first proof,
 we have deduced the local vanishing discriminant again using
 arguments from Berkovich theory, more precisely, Corollaries
 \ref{cor:nonHomeoCycle} and \ref{cor:highMultAndDiscrim}. In the
 elliptic cubic case, we can guarantee the vanishing of the local
 discriminant by means of Corollary~\ref{cor:initialDiscriminant},
 without relying on Lemma~\ref{lem:noDecontractionIfLocallySmooth}.

\begin{proof}[Proof of Theorem~\ref{thm:repairEll}]
  Assume $\Trop(g)$ has a visible cycle but the $j$-invariant of $g$
  does not have the expected valuation. Then, $(1,1)$ is a vertex in
  the Newton subdivision of $g$ and \cite[Lemma 23]{KMM07} ensures
  that the expected initial form of $\Delta$ vanishes at $g$. We
  conclude that $\Delta_{A_j}(g_{|_{A_j}})=0$ for some $j\in I$. In
  particular, $A_j$ is non-defective so $Q_j$ cannot be a pyramid. The
  cubic condition implies that $Q_j$ is a trapezoid of height 1 with a
  base of length 1. The corresponding vertex $v$ dual to $A_j$ lies in
  the cycle of $\Trop(g)$ and is locally reducible. Using
  Lemma~\ref{lm:expectedHeights} and the proof of
  Lemma~\ref{lm:VanishingDecontractsVertex} we can prolong the cycle
  of the tropical elliptic cubic by a linear re-embedding induced by a
  linear tropical modification of the ambient space.

Conversely, assume none of the local discriminants $\Delta_{A_j}$
vanish when evaluated at $g_{|_{A_j}}$.
Corollary~\ref{cor:initialDiscriminant} ensures that $\Delta$ has the
expected initial form. We conclude that the $j$-invariant of $g$ has
the expected valuation and so the cycle in the tropical elliptic cubic
has the expected length.
\end{proof}
\section{Experimental Results}\label{sec:experiments}
The polyhedral nature of tropical plane curves allows for many
experimentations to devise algorithms to locally repair non-faithful
tropicalizations by means of tropical modifications. In this section we
provide three examples that shed light on some of the open questions
discussed earlier in this paper. We view them as starting points for further
investigations in this area.

Our first example extends the conclusion of Theorem
\ref{thm:redVertex} to a local reducible vertex of valency 6, which locally is the union of a tropical line with a reflected tropical line.
\begin{example}\label{ex:otherRedVertex}
Consider the plane curve in $(\PS^*)^2$ with defining equation
\[
g(x,y)=(t+2t^2)x^3y^3+(1-t^3)x^2y^2+(1+t)xy^3+(1-t^4)x^3+(1+3t^2)x^2+(1+6t)y^2+(1+t)y.
\]
The corresponding tropical plane curve is depicted in the left of
Figure~\ref{fig:projectionsSpecialRedVertex}. The vertex $v:=(0,0)$ is
locally reducible. One of its components is the tropical line
$F=\max\{X,Y,0\}$. We perform a linear tropical modification of
$\RR^2$ along $F$. We pick a lifting $f=x+ay+bz$ determined by special
choices of $a,b\in \PS$ with valuation 0 adapted to the curve
$\Trop(g)$. This modification produces six two-dimensional cones in
$\RR^3$, spanned by the rays of $F$ and $-e_3$. We label them
$\sigma_1,\ldots, \sigma_6$ as in
Figure~\ref{fig:SpecialRedVertex}. The cells $\sigma_4,\sigma_5$ and
$\sigma_6$ are the ones attached to the tropical line $F$. For
example, $\sigma_1$ is defined by the system $X,Y\leq 0$ and $Z=0$,
whereas $\sigma_4$ is determined by $X=0$ and $Y,Z\leq 0$.

As Figure~\ref{fig:SpecialRedVertex} illustrates, we assign different
colors to the intersection of $\Trop(I_{g,f})$ with each cone
$\sigma_i$ for $i=1,\ldots, 6$. This helps us see the image of each
piece under the three projections $\pi_{XY}, \pi_{XZ}$ and $\pi_{ZY}$,
given in the top row of
Figure~\ref{fig:projectionsSpecialRedVertex}. From left to right,
these projections are defined by the polynomials $g(x,y)$,
$g_1:=g(z-ay-b,y)$ and $g_2:=g(x,(z-x-b)/a)$, respectively.  We need
all three projections in order to reconstruction $\Trop(I_{g,f})$. As
in the proof of Lemma~\ref{lem:ModifViaProjections}, the
multiplicities on the edges of $\Trop(I_{g,f})$ mapping to the
tropical line defined by $F$ on each of the three projections are
determined by the push-forward formula for multiplicities.
 \begin{figure}[tb]
   \centering
     \includegraphics[scale=0.13]{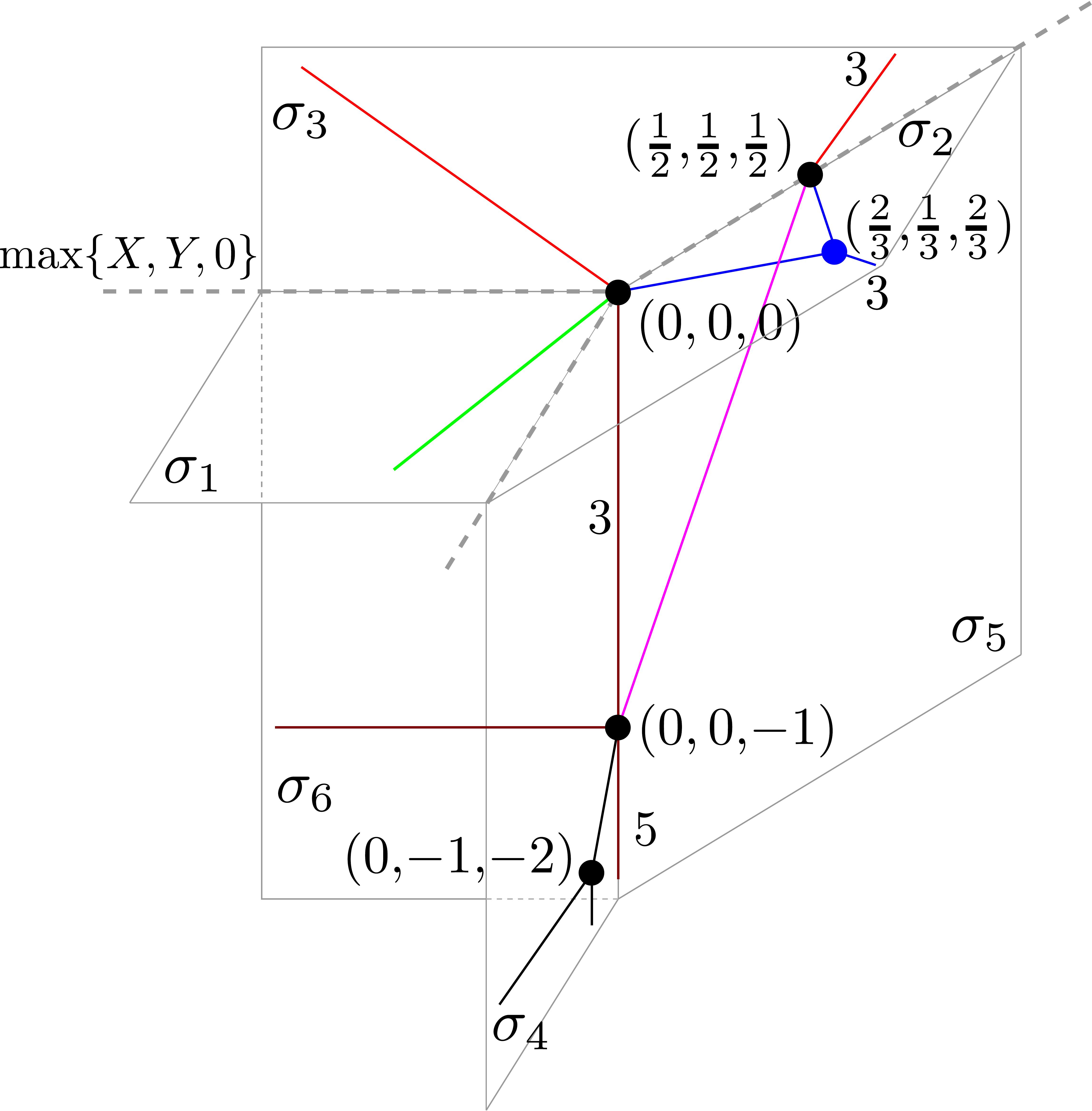}
     \caption{A special linear tropical modification of a locally
       reducible vertex along the tropical line defined by
       $F=\max\{X,Y,0\}$ (indicated with a dashed line), its lifting
       $f=x+y+1$ and its effect on a given tropical plane curve.}
   \label{fig:SpecialRedVertex}
 \end{figure}

 The choice of scalars $a,b$ is done to emulate the conclusions of
 Lemmas~\ref{lm:expectedHeights}
 and~\ref{lm:VanishingDecontractsVertex}.  We analyze the contribution
 of the terms $c_{ij}(z-ay-b)^iy^j$ coming from all marked points in
 the the cell dual to $v$, in all the monomials $y^k$ of $g_1$, for
 $k=j,\ldots, i+j$. We proceed analogously with the curve defined by
 $g_2$ and see the contribution of $c_{ij}x^i(y-x-b)^j/a^j$ to the
 monomials $z^k$ with $k=i,\ldots, i+j$. The coefficients of these
 monomials have expected valuation 0. Our choice of $a,b$ must be such
 that some of them have strictly positive valuation. To achieve this,
 we use the local discriminant $\Delta_{v^{\vee}}$.

 An easy calculation shows that this local discriminant vanishes at
 $\init_{v}(g)$, as in the situations covered in
 Section~\ref{sec:repa-trop}. Therefore, we have a chance of having a
 non-faithful tropicalization locally around $(0,0)$. Indeed, by
 choosing $a=b=1$ we see that this is the case: we manage to make the
 initial term of all monomials $y^k$ ($k=j,\ldots, i+j$) in $g_1$ and
 $z^k$ ($k=i,\ldots, j$) in $g_2$ drop together. These monomials are
 the red points in Figure~\ref{fig:projectionsSpecialRedVertex}.  As a
 consequence, a bounded edge in $\Trop(I_{g,f})$ with direction $-e_3$
 maps to $(0,0)$ under the projection $\pi_{XY}$. We see this
 phenomenon in Figure~\ref{fig:SpecialRedVertex}.
 \begin{figure}[htb]
 \begin{minipage}[r]{0.27\linewidth}
\includegraphics[scale=0.2]{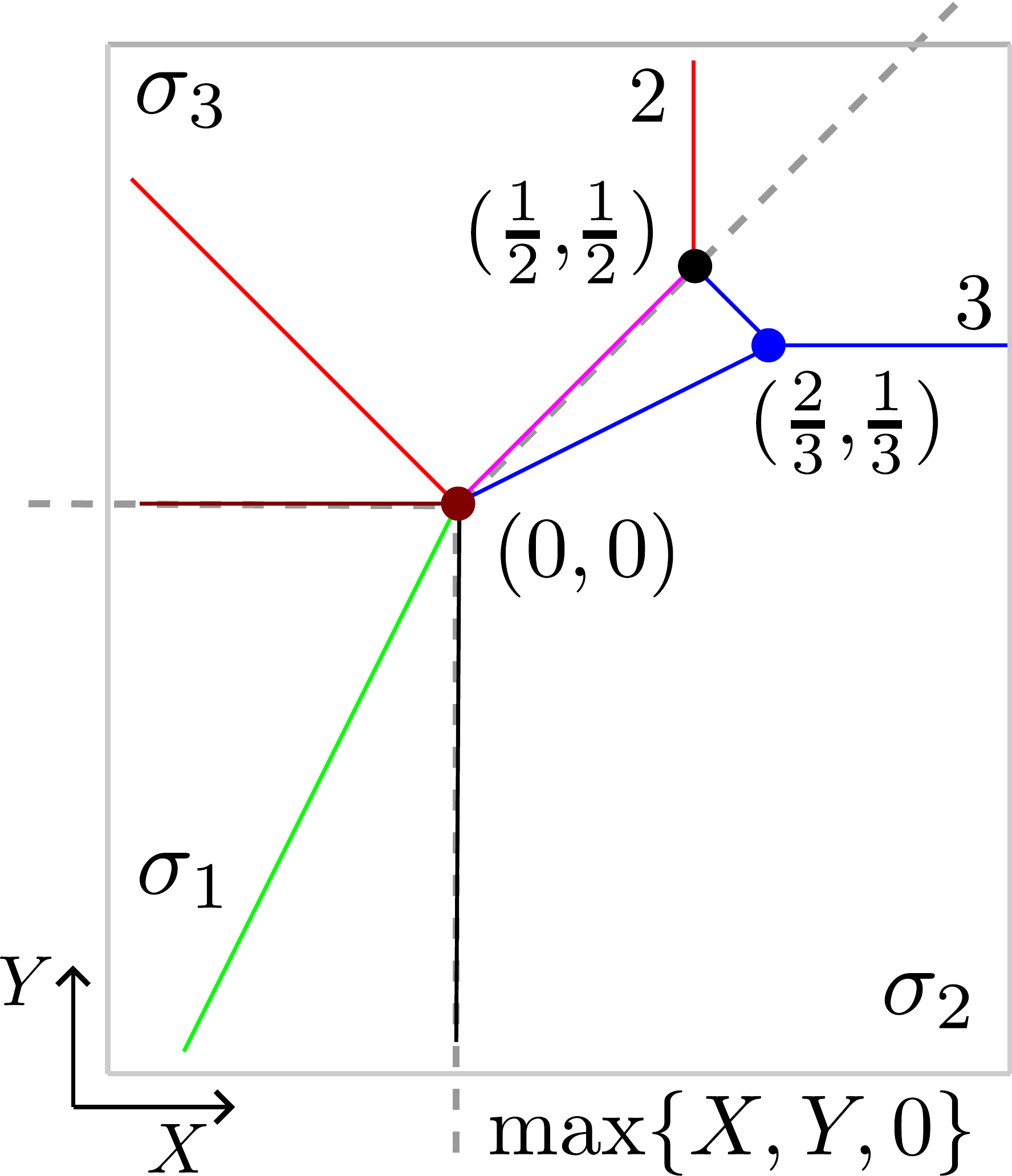}
\end{minipage}
  \begin{minipage}[c]{0.27\linewidth}
  \includegraphics[scale=0.2]{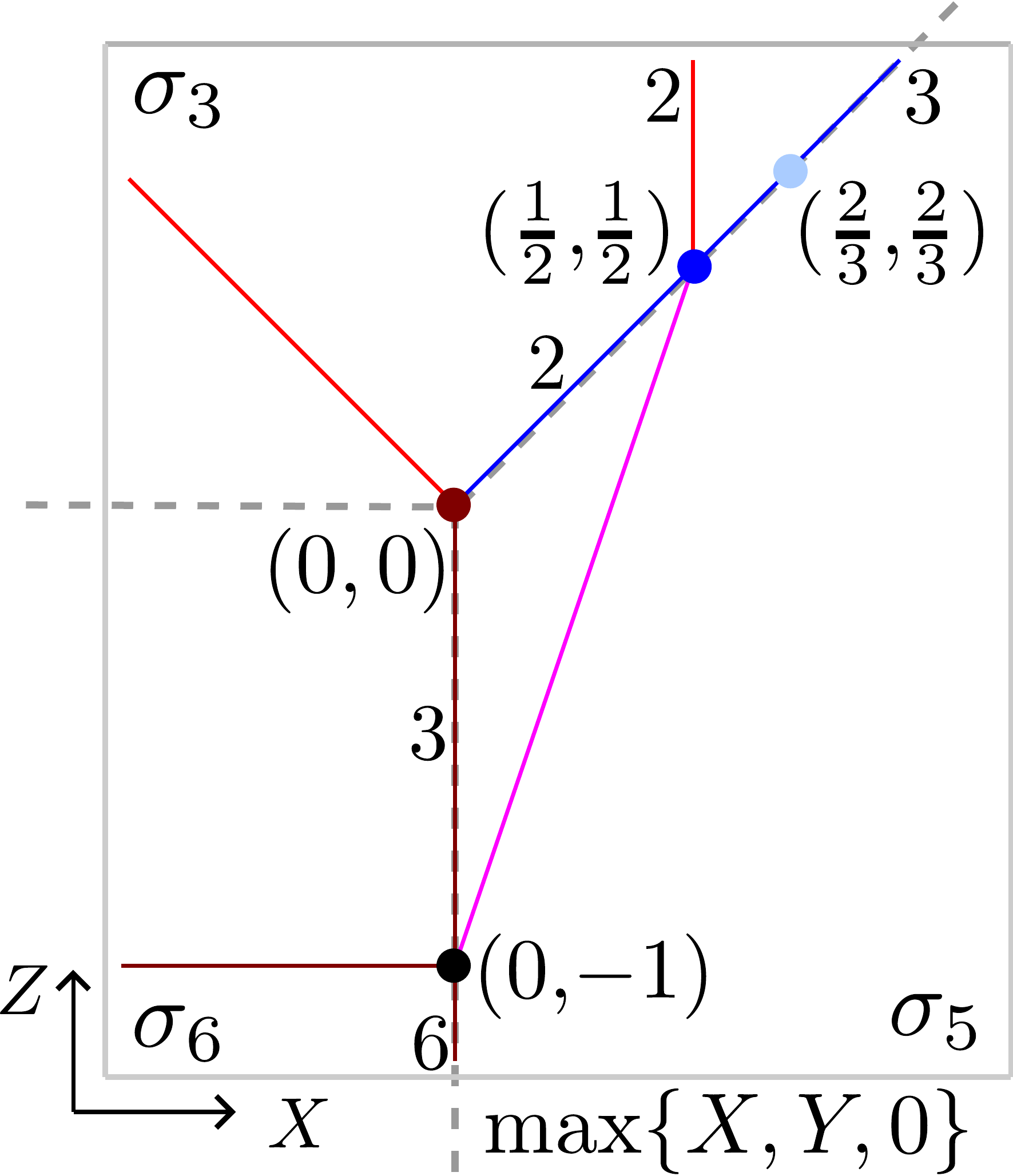}
\end{minipage}
  \begin{minipage}[r]{0.44\linewidth}
  \includegraphics[scale=0.2]{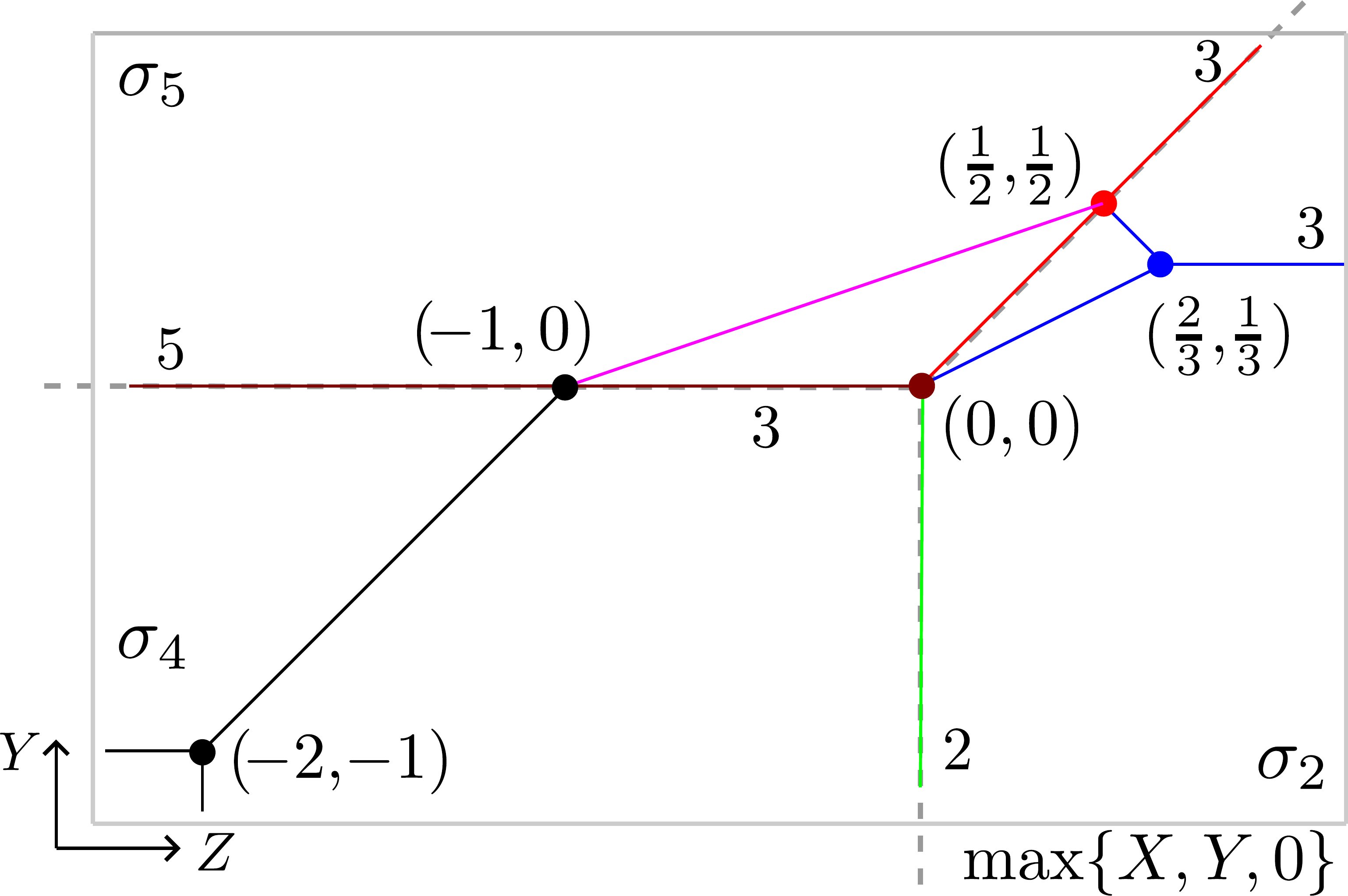}
\end{minipage}

\begin{minipage}[r]{0.27\linewidth}
\includegraphics[scale=0.25]{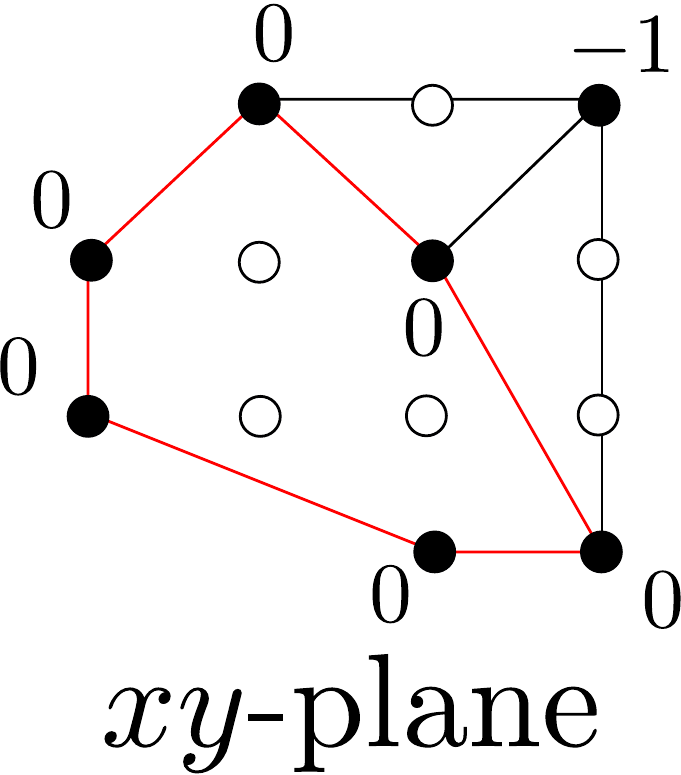}
\end{minipage}
  \begin{minipage}[c]{0.27\linewidth}
 \includegraphics[scale=0.25]{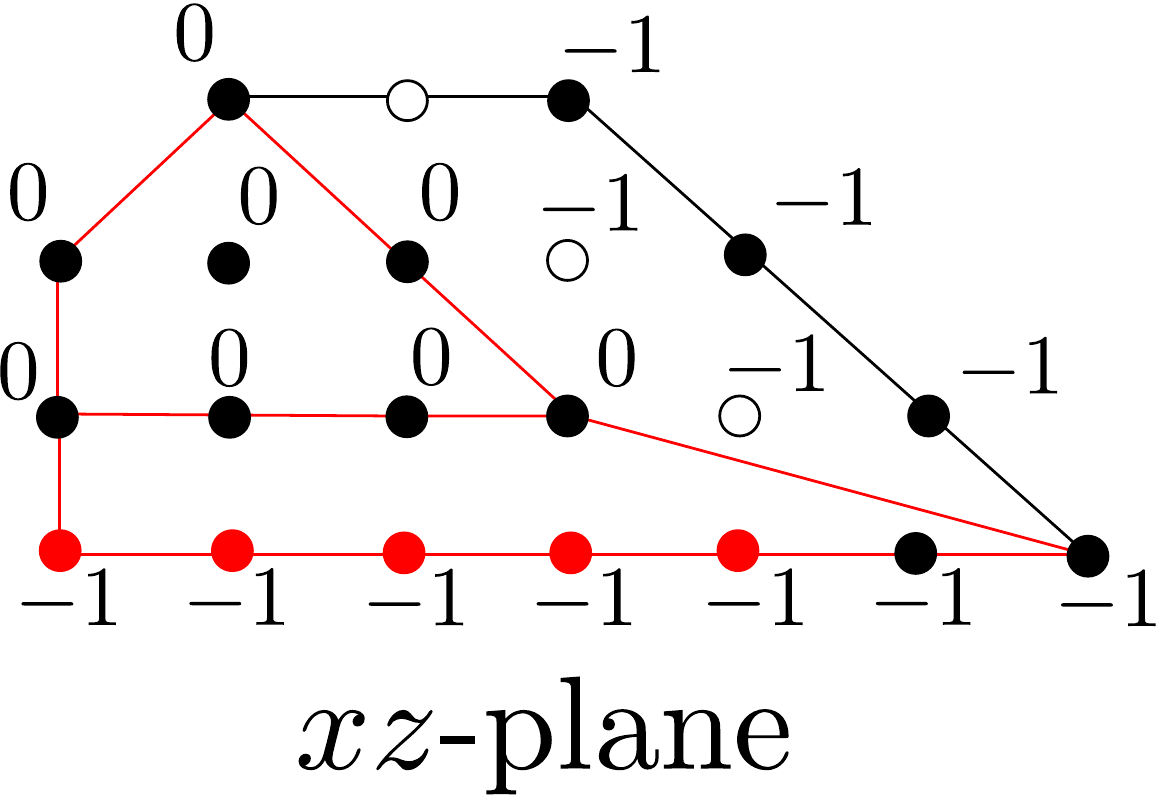}
\end{minipage}
  \begin{minipage}[c]{0.32\linewidth}
\hspace{10ex} \includegraphics[scale=0.25]{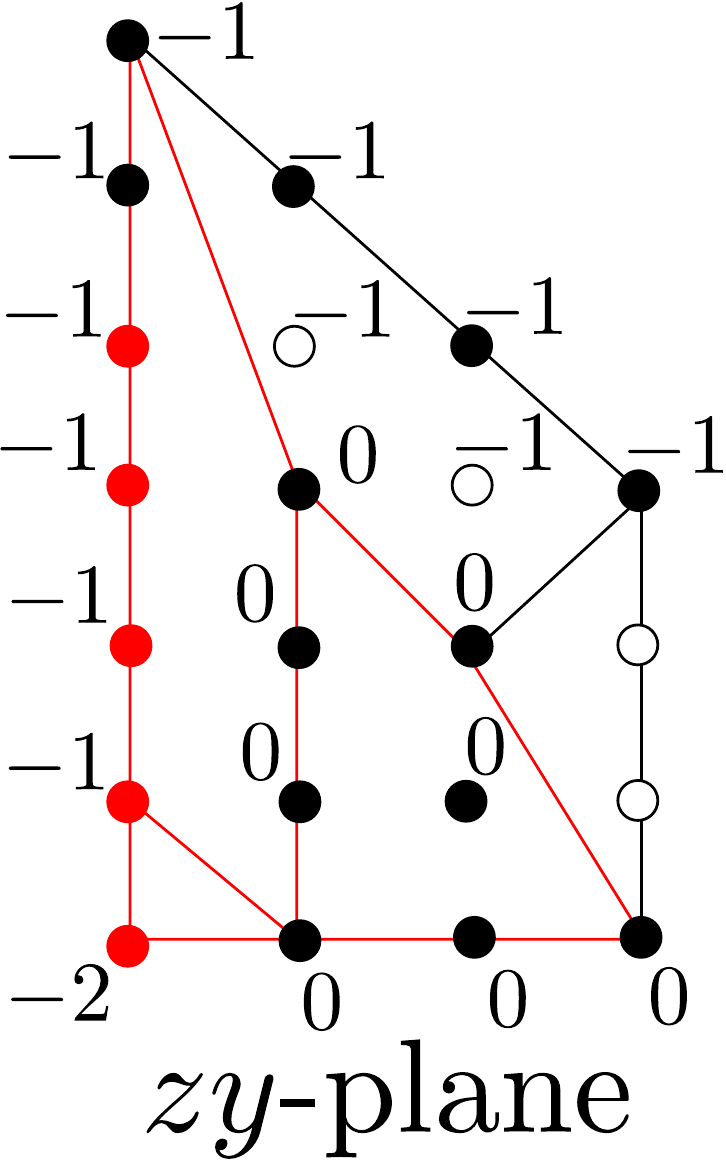}
\end{minipage}
\caption{From left to right: Projections of the tropical curve (with
  dual Newton subdivisions) from Figure~\ref{fig:SpecialRedVertex} to
  the planes $XY$, $XZ$ and $ZY$-respectively.}
   \label{fig:projectionsSpecialRedVertex}
 \end{figure}
\end{example}

A second natural question that arises from our results from
Section~\ref{sec:repa-trop} is the following. Consider a plane
elliptic cubic with bad reduction defined by $g$, but whose
tropicalization $\Trop(g)$ contains no cycle. Can we use linear
tropical modifications to make this cycle appear via a special linear
re-embedding? If so, does this method generalize to other Mumford
curves? A positive answer to this question would provide a
combinatorial effective way to test if a plane curve is Mumford.

As observed in Remark~\ref{rem:unfoldingAndTrivalentOutcome},
Theorem~\ref{thm:fatEdge} gives a positive answer to the above
question when $\Trop(g)$ contains a bounded edge of multiplicity $2$
with non-vanishing discriminant (as we saw in
Example~\ref{ex:unfoldingDoubleEdge}). Our next example produces a
cycle by modifying $\RR^2$ along a high-multiplicity
end. 
\begin{example}\label{ex:CyleAppears} 
We consider the smooth plane elliptic cubic $C$ with defining equation
\[
g(x,y)=t^{10}\,x^3+x^2y+xy^2+t^{11}\,y^3+3xy-1.
\]
Its $j$-invariant has valuation $\val_j(g)=-10<0$. The tropical curve
$\Trop(g)$ is depicted in the top-right corner of
Figure~\ref{fig:MakeVisible}. It contains no cycle and no bounded edge
of high multiplicity.

The induced Newton subdivision of $g$ appears on the top-right of
Figure~\ref{fig:MakeVisible}. The triangle containing $(1,1)$ in
its interior is dual to the vertex $(0,0)$ of $\Trop(g)$. Since the
discriminant $\Delta_{(0,0)^{\vee}}$ vanishes at $\init_{(0,0)}(g)$,
we know the tropicalization map may not be faithful locally around
$(0,0)$.  Indeed, using three consecutive linear tropical
modifications we can find a linear re-embedding of $C$ where 
the curve $\Trop(I_{g,f_1,f_2,f_3})$ in the bottom-left of
Figure~\ref{fig:SkeletonMakeCycleVisible} contains a  cycle of length $10$.

 The three projections used to reconstruct $\Trop(I_{g,f_1,f_2,f_3})$
 are described in Figure~\ref{fig:MakeVisible}, with the following
 convention. Each cell $\sigma_{ijk}$ with $i,j,k\in \{1,2,3\}$
 encodes the intersection of the cells $\sigma_i(1)$, $\sigma_j(2)$
 and $\sigma_k(3)$ corresponding to each one of the three linear
 tropical modifications.

 The three tropical modifications are chosen as follows. First, we
 attempt to add a downward end $e$ of multiplicity 2 to the vertex
 $(0,0)$ in the cell $\sigmaint_3$. We do so by means of a
 modification along $X=0$, picking a lifting function $f_1=x+\AA_1$
 with $\val(\AA_1)=0$. Our choice of $\AA_1$ must be such that the
 discriminant of the attached edge $e$ vanishes when evaluated at
 $\init_{\pi_{ZY}(e)}(g(z-\AA_1,y))$, to have a chance for the map
 $\trop$ to be non-faithful on $e$. Indeed, if this were not the case,
 the end $e$ will be the image of two ends in $\Sigma(\widehat{C})$
 by~\cite[Proposition 4.24]{BPR11} and the procedure will not yield a
 bounded edge of higher multiplicity.  We find
 $\AA_1':=\init_t(\AA_1)$ using the well-known quadratic formula
\begin{equation}
\AA_1'(\init_t(c_{1,1}) + \init_t(c_{2,1})\AA_1')^2-4 \init_t(c_{0,0})\init_t(c_{1,2}) = \AA_1'(3+\AA_1')^2+4=(\AA_1'+1)^2(\AA_1'+4)=0.\label{eq:5}
\end{equation}
We choose $\AA_1=\AA_1'=-1$.  As a result, the Newton subdivision of
$g_1(z,y):=g(z-1,y)$ contains a quadrilateral with $(1,1)$ as its
unique interior point. This quadrilateral is dual to the vertex
$(0,0)$ in $\Trop(g_1)$.  For this lifting $f_1$ (and no other), both
discriminants $\Delta_{(0,0)^{\vee}}$ and $\Delta_{e^{\vee}}$ vanish
at $\init_{(0,0)}(g_1)$ and $\init_{e}(g_1)$, respectively. The curve
$\Trop(g_1)$ is depicted in the top-left of
Figure~\ref{fig:MakeVisible}.

Next, we aim to transform our horizontal multiplicity 2 end $e$ in
$\Trop(g_1)$ into a bounded edge of multiplicity 2, and use
Theorem~\ref{thm:fatEdge} to unfold it.  We modify the $ZY$-plane
along the horizontal line $Y=0$, picking a new variable $v$ and a
lifting $f_2= y +{\AA}_2$, with $\val({\AA}_2)=0$. Our choice of
${\AA}_2$ is subject to the constraint that the coefficients of $1$,
$z$ and $v$ in $g_2(z,v):=g_1(z,v-{\AA}_2)$ have higher valuation than
expected.

Our choice of $\AA_1$ imposed by condition~\eqref{eq:5} allows us to
find the desired value for ${\AA_2}$. In this particular example, the
value $\AA_2=-1$ satisfies all three requirements. Its initial term is completely
determined by the initial term of the constant coefficient of $g_2$,
namely $(\init_t(\AA_2)+1)^2$.  The tropical curve $\Trop(g_2)$ is depicted
on the lower-right of Figure~\ref{fig:MakeVisible}. It contains a
bounded, multiplicity 2 edge $e'$ inside the line $Z=V$.

In our final step, we look at the discriminant associated to the edge
$e'$. By construction, $\Delta_{e'^{\vee}}$ does not vanish at
$\init_{e'}(g_2)$. Even though the vertex $(0,0)$ has degree 4 in
$g_2$, we can mimic the proof of Theorem~\ref{thm:fatEdge} to conclude
that we can unfold $e'$ via a tropical modification. Indeed, the
linear tropical modification along the line $Z=V$ with lifting
$f_3=z+(1+\sqrt{3})/2\,v$ unfolds the edge $e'$ and yields a faithful
embedding on the visible cycle of $\Trop(g,f_1,f_2,f_3)$. The
projection to the $UV$-plane contains this cycle and is depicted in
the bottom-left of Figure~\ref{fig:MakeVisible}.


 \begin{figure}[tb]
    \centering
   \includegraphics[scale=0.2]{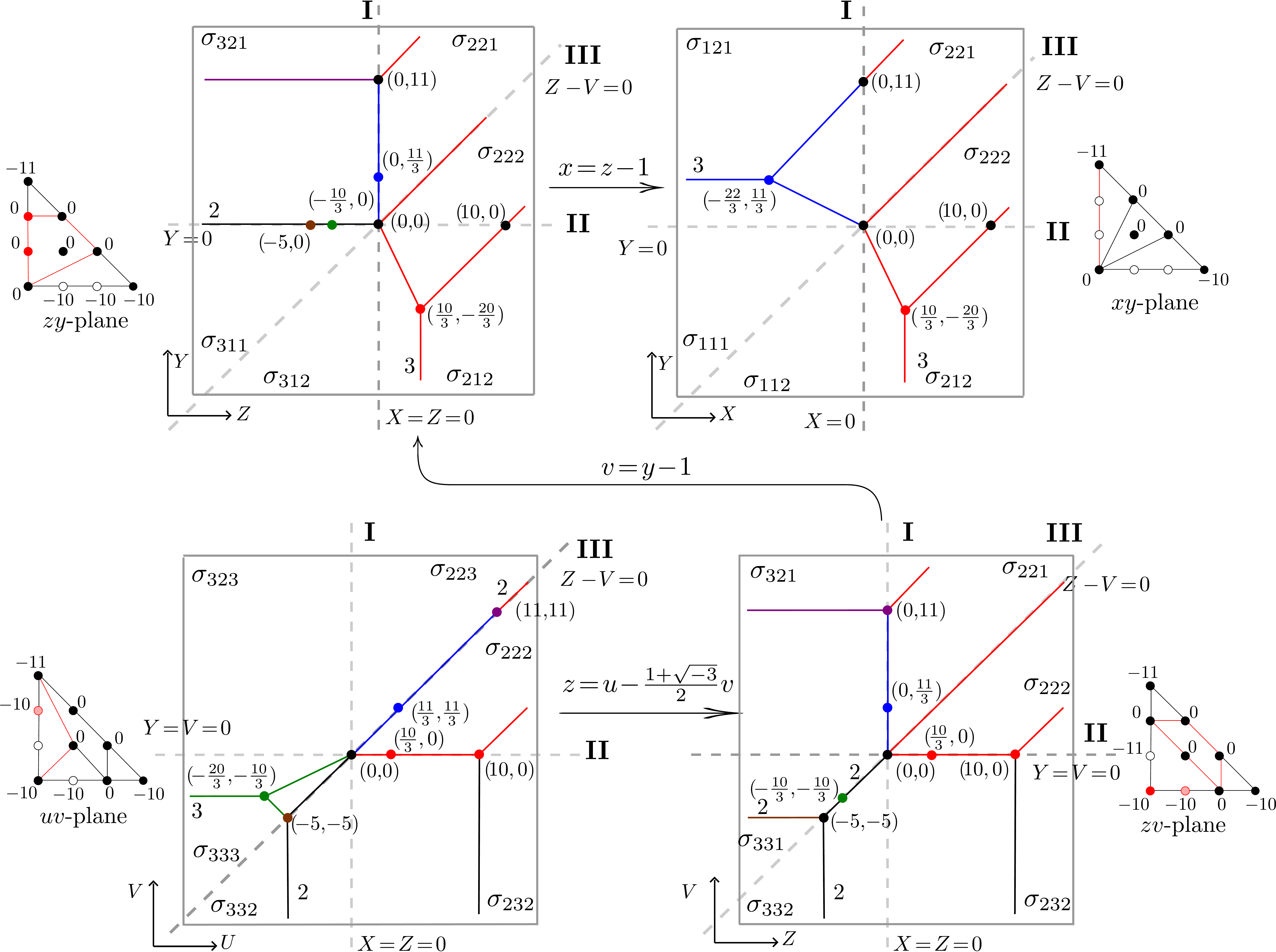}
    \caption{From right to left and top to bottom: Using 3 linear tropical modifications we make the cycle of a tropical plane elliptic cubic visible. 
 The red lines and dots on the side pictures indicate the cells of the  Newton subdivisions of 
 $g,g_1,g_2$ and $g_3$, and the monomials affected by each coordinate change.}
    \label{fig:MakeVisible}
  \end{figure}

  \begin{figure}[tb]
    \centering
    \includegraphics[scale=0.11]{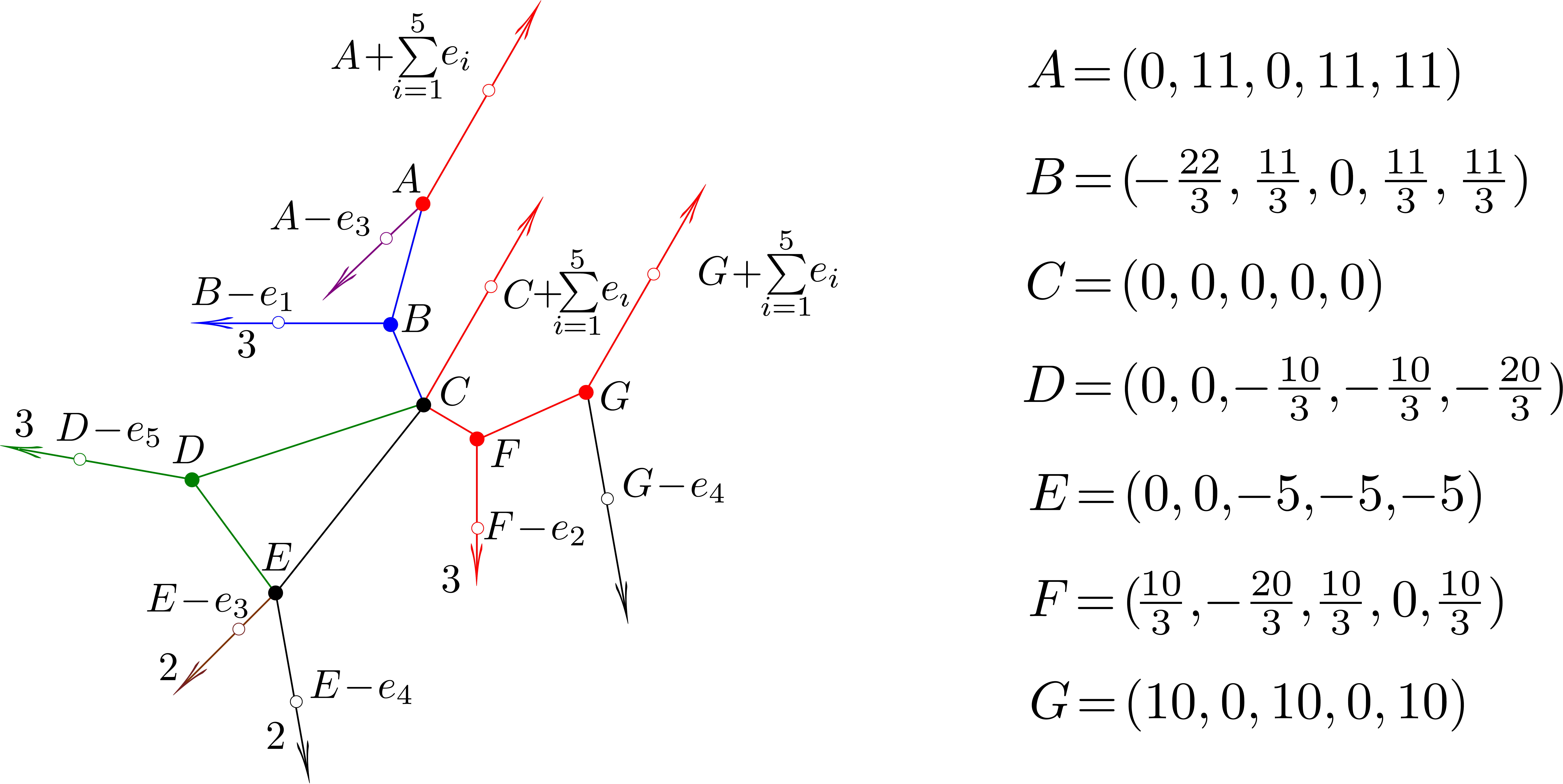}
    \caption{The tropicalization of $I_{g,f_1,f_2,f_3}$ repairs the
      tropical plane elliptic cubic by three linear tropical
      modifications of $\RR^2$ (adding variables $z,v,u$,
      respectively). Here, $f_1=x+1$, $f_2=y+1$ and
      $f_3=z+\frac{1+\sqrt{-3}}{2}v$. We label the coordinates in
      $\RR^5$ by $(X,Y,Z,V,U)$. The color coding matches that of
      Figure~\ref{fig:MakeVisible}.}
\label{fig:SkeletonMakeCycleVisible}
  \end{figure}

  It is worth noticing that no other choice of $\AA_1'$ in our lifting
  function $f_1$ for our first modification allows us to produce a
  bounded multiplicity 2 edge in the tropical curve $\Trop(g_1)$ with
  a second linear tropical modification. As we mentioned earlier, if
  we choose $\AA_1'=-4$, any subsequent linear tropical modification
  will split the multiplicity 2 end on the top-left of
  Figure~\ref{fig:MakeVisible} into two multiplicity 1 ends after
  shifting the endpoint $(0,0)$ in the southwest direction.
\end{example}

As we saw in the previous example, linear tropical modifications can
help us draw some conclusions about the tropicalization
map~\eqref{eq:tropMap}. In Figure~\ref{fig:SkeletonMakeCycleVisible},
the extended skeleton
$\widehat{\Sigma}(I_{g,f_1,f_2,f_3})\smallsetminus
D_{I_{g,f_1,f_2,f_3}}$ contains a subgraph homeomorphic to
$\Trop(I_{g,f_1,f_2,f_3})$ via the map $\trop$. Furthermore, the map
$\trop$ is an isometry over the bounded part of
$\Trop(I_{g,f_1,f_2,f_3})$.  Unfortunately, this procedure does not
always yield a complete description of these skeleta.

Our last shows that linear tropical modifications may not
suffice to unfold an end of a tropical plane elliptic cubic to produce
a cycle.
\begin{example}\label{ex:UnfoldImpossible}
  We consider the plane elliptic cubic $C$ defined by the equation
\[
g(x,y)=t^{10}x^3+xy^2+t^{11}y^3+x^2+4xy+2x+1.
\]
This cubic satisfies $\val(j(g))= -10<0$. As we can see from the right
side of Figure~\ref{fig:UnfoldImpossible}, the tropical curve
$\Trop(g)$ contains no cycle and has no bounded edge of higher
multiplicity. The triangle dual to the vertex $(0,0)$ has discriminant
$\Delta_{(0,0)^{\vee}}= d^4-8bd^2q+16b^2q^2-64acq^2$, which vanishes
at $\init_{(0,0)}(g)$. Likewise, the discriminant of the end $e$
adjacent to $(0,0)$ equal $\Delta_{e}=b^2-4ac$ and also vanishes at
$\init_{e}(g)$. Thus, $\init_e(g)$ has a unique component of
multiplicity 2. Since the end $e$ has multiplicity two,
\cite[Proposition 4.24]{BPR11} and~\eqref{eq:tropMultVsRelMult} imply
that the fiber of $\trop$ over a generic point in $e$ has either size
2 (with relative multiplicities $1$) or size 1 (with relative
multiplicity 2).

We attempt to unfold the edge $e$ by a linear tropical modification
along the vertical line $X=0$.  We want the coefficient
$\tilde{c}_{0,0}$ to have strictly positive valuation.  Therefore, our
lifting function $f=x+\AA$ must satisfy $\val(\AA)=0$ and
$1+2\AA_0+\AA_0^2=0$ for $\AA_0=\init_t(\AA)$. There is a unique
choice for such $\AA_0$, namely $\AA_0=-1$. Unfortunately, the
coefficient $\tilde{c}_{1,0}$ also has strictly positive valuation and
the method fails to produce a bounded weight two edge, as we see on
the left of Figure~\ref{fig:UnfoldImpossible}.
\end{example}
 \begin{figure}[tb]
   \begin{minipage}[c]{.47\linewidth}
      \includegraphics[scale=0.2]{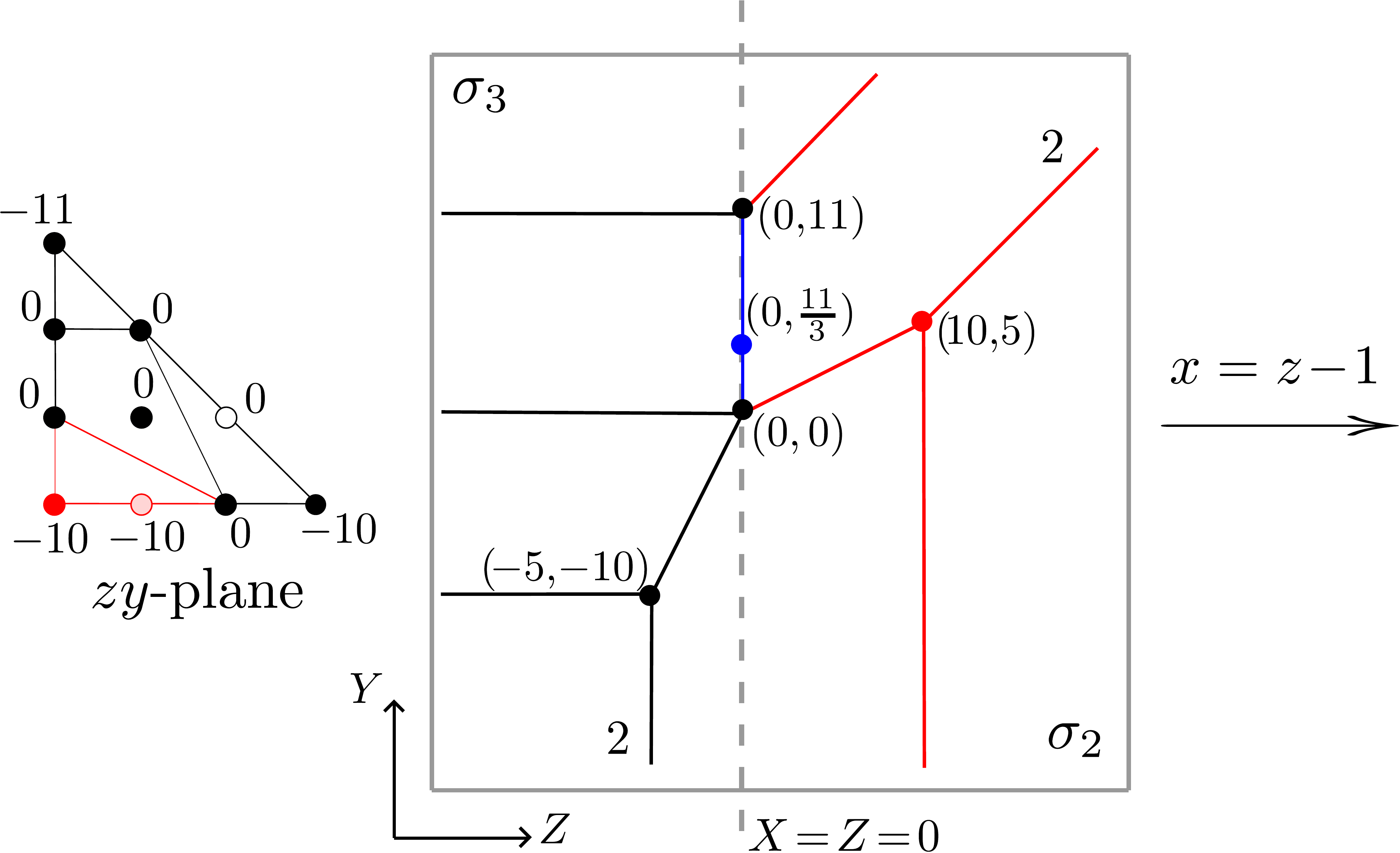}
    \end{minipage}
     \begin{minipage}[c]{.47\linewidth}      \includegraphics[scale=0.2]{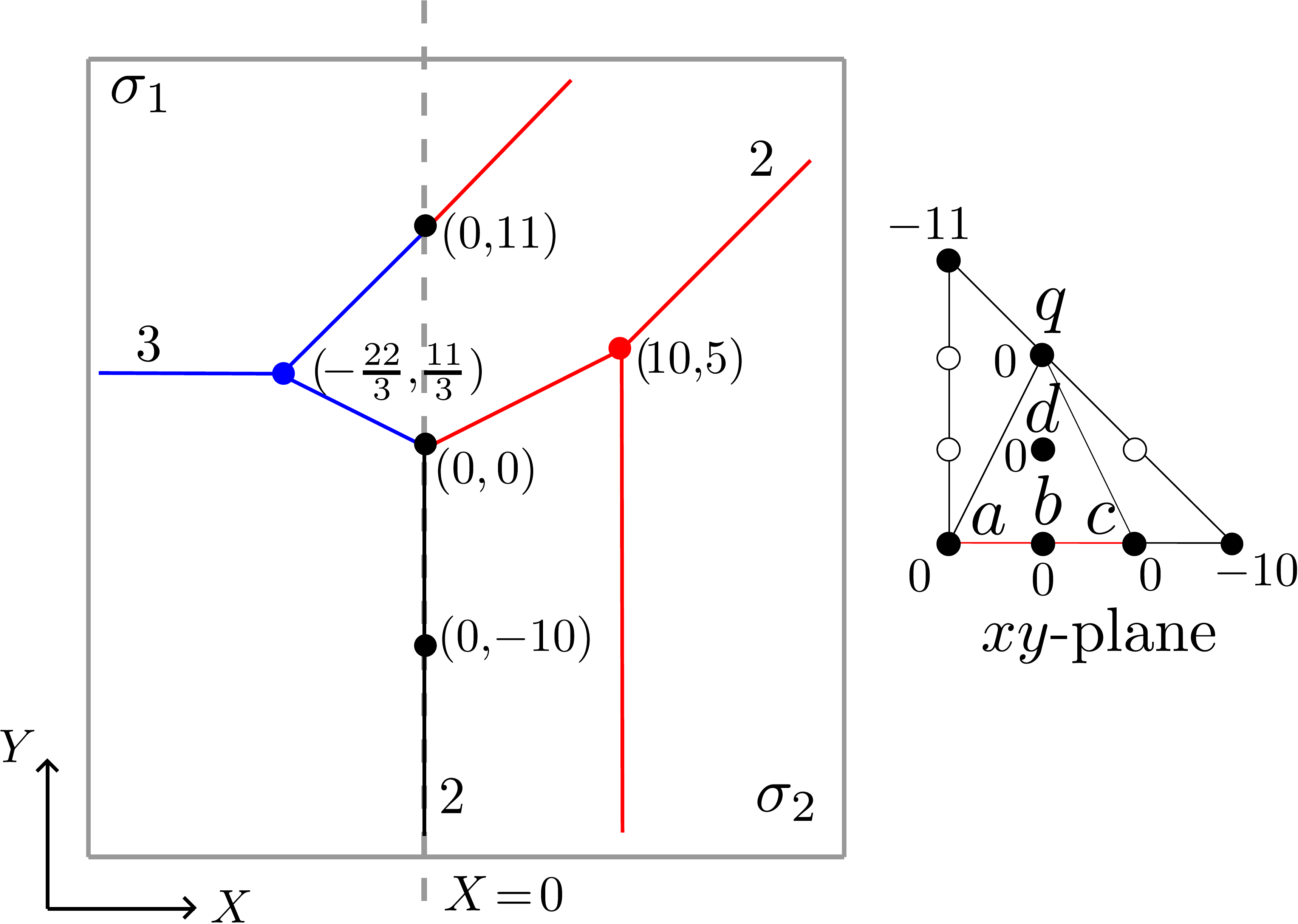}
    \end{minipage}
    \caption{We attempt to unfold a high multiplicity end with a
      linear tropical modification to produce a cycle.  The endmost
      figures represent the effect on the Newton subdivisions of $g$
      and $\tilde{g}$, respectively.  The coefficients $a,b,c,d,q$ on
      the rightmost polytope give the unknowns in the $A$-discriminant
      of the interior triangle.}
\label{fig:UnfoldImpossible}
 \end{figure}
\vspace{-2ex}

\section*{Acknowledgments}
We wish to thank Erwan Brugall\'{e}, Arne Buchholz, Ilia Itenberg,
Diane Maclagan, Thomas Markwig, Ralph Morrison, Bernd
Sturmfels, Till Wagner and Annette Werner for very fruitful
conversations. All the computations in this paper were done using the
\texttt{tropical.lib} library for \texttt{Singular}~\cite{JMM07a}.
The first author was supported by an Alexander von Humboldt
Postdoctoral Research Fellowship (Germany) and by an NSF postdoctoral
fellowship DMS-1103857 (USA). The second author was supported by
DFG-grant 4797/5-1 and by GIF-grant 1174-197.6/2011.

Part of this project was carried out during the 2013 program on \emph{Tropical Geometry and Topology} at the
Max-Planck Institut f\"ur Mathematik in Bonn, where the second author was in
residence. We thank MPI for their hospitality.

 \bibliographystyle{abbrv}



\end{document}